\documentclass[11pt]{article}
\usepackage[utf8]{inputenc}
\usepackage[a4paper, margin=0.8in]{geometry}
\usepackage{amsmath}
\usepackage{array}
\newcolumntype{P}[1]{>{\centering\arraybackslash}p{#1}}
\newcolumntype{M}[1]{>{\centering\arraybackslash}m{#1}}
\usepackage{xcolor}
\usepackage{amsfonts, amssymb}
\usepackage{amsthm}
\usepackage{mathabx}
\usepackage{cases}
\usepackage{esvect}
\usepackage[english]{babel}
\usepackage[ruled,linesnumbered]{algorithm2e}
\usepackage[T1]{fontenc}
\usepackage{multirow}
\usepackage{graphicx}
\usepackage{cite}
\usepackage[numbers,sort]{natbib}
 \usepackage{titlesec}
 \usepackage{caption}
\titleformat{\section}
  {\normalfont\fontfamily{ptm}\fontsize{11}{11}\bfseries}{\thesection}{1em}{}
\titleformat{\subsection}
  {\normalfont\fontfamily{ptm}\fontsize{10}{11}\bfseries}{\thesubsection}{1em}{}
\titleformat{\subsubsection}
  {\normalfont\fontfamily{ptm}\fontsize{10}{11}\selectfont}{\thesubsubsection}{1em}{}
\newtheorem{theorem}{Theorem}[section]
\newtheorem{corollary}[theorem]{Corollary}
\newtheorem{lemma}[theorem]{Lemma}
\newtheorem{proposition}[theorem]{Proposition}
\newtheorem{definition}[theorem]{Definition}
\newtheorem{remark}[theorem]{Remark}
\newtheorem{example}[theorem]{Example}
\newtheorem{observation}[theorem]{Observation}

\newtheorem{property}[theorem]{Property}
\usepackage{amsthm}
\newtheorem{manualtheoreminner}{Lemma}
\newenvironment{manualtheorem}[1]{
  \renewcommand\themanualtheoreminner{#1}
  \manualtheoreminner
}{\endmanualtheoreminner}
\usepackage{hyperref}
\hypersetup{
    colorlinks=true,    
    urlcolor=cyan,
    linkcolor=cyan,
    citecolor=cyan
}
\usepackage{mathtools}
\usepackage{amssymb}
\DeclarePairedDelimiter\ceil{\lceil}{\rceil}
\DeclarePairedDelimiter\floor{\lfloor}{\rfloor}
\providecommand{\keywords}[1]
{
  \small	
  \quad \quad \textbf{\textit{Keywords --}} #1
}

\newcommand{\conv}[1]{\text{conv}\left({#1}\right)}
\newcommand{\ch}[1]{\text{ch}\left({#1}\right)}
\newcommand{\pa}[1]{\text{pa}\left({#1}\right)}
\newcommand{\dep}[1]{\text{depth}\left({#1}\right)}
\newcommand{\eged}{\hfill $\blacksquare$}

\title{\large On Constrained Mixed-Integer DR-Submodular Minimization}
\author{ \small Qimeng Yu \quad Simge K\"u\c{c}\"ukyavuz\vspace{0.2cm}  \\ \small Department of Industrial Engineering and Management Sciences \\ \small Northwestern University, Evanston, IL, USA \\ \small \{kim.yu@u.northwestern.edu, simge@northwestern.edu\}}
\date{\small \today} 
\begin{document}
\maketitle

\begin{abstract}
\noindent DR-submodular functions encompass a broad class of functions which are generally non-convex and non-concave. We study the problem of minimizing any DR-submodular function, with continuous and general integer variables, under box constraints and possibly additional monotonicity constraints. We propose valid linear inequalities for the epigraph of any DR-submodular function under the constraints. We further provide the complete convex hull of such an epigraph, which, surprisingly, turns out to be polyhedral. We propose a polynomial-time exact separation algorithm for our proposed valid inequalities, with which we first establish the polynomial-time solvability of this class of mixed-integer nonlinear optimization problems. 
\end{abstract}
\keywords{DR-submodular minimization; mixed-integer variables; polyhedral study; convex hull.}

\section{Introduction}
For a finite non-empty ground set $N = \{1,2,\dots,n\}$, we denote its power set by $2^N$. A set function $f:2^N\rightarrow \mathbb{R}$ is called \emph{submodular} if 
\[f(X) + f(Y) \geq f(X\cup Y) + f(X\cap Y)\] 
for any $X,Y\in 2^N$. Intuitively, submodular set functions model diminishing returns (DR). To see this, $f$ is submodular if the following equivalent condition holds:
\[f(X\cup\{i\}) - f(X) \geq f(Y\cup\{i\}) - f(Y)\]
for all $X\subseteq Y\subseteq N$ and every $i\in N\backslash Y$. Optimization problems concerning such objective functions have received great interest in integer programming and combinatorial optimization, driven by numerous applications including set covering \citep{wolsey1982analysis}, graph cuts \citep{goemans1995improved}, facility location \citep{cornuejols1977uncapacitated}, image segmentation \citep{jegelka2011submodularity}, sensor placement \citep{krause2008near}, and influence propagation \citep{kempe2015maximizing}. It is known that unconstrained submodular minimization is solvable in polynomial time \citep{lovasz1983submodular, grotschel1981ellipsoid, iwata2001combinatorial, lee2015faster, orlin2009faster, cunningham1985submodular, schrijver2000combinatorial, mccormick2005submodular, iwata2009simple, iwata2008submodular}. Whereas, constrained submodular minimization problems are NP-hard in general \citep{svitkina2011submodular}. In addition, submodular set function maximization can be efficiently approximated with strong guarantees \citep{nemhauser1978analysis, nemhauser1978best, lee2010maximizing, calinescu2007maximizing, sviridenko2004note, orlin2018robust}. \\

The notion of submodularity is extendable to functions that are defined over more general domains than $2^N$ or equivalently $\{0,1\}^n$. DR-submodularity is one such extension. We let $\mathbf{e}^i\in\mathbb{R}^n$ be a vector with one in the $i$-th entry and zero everywhere else. Formally, 
\begin{definition}
\label{def:dr}
A function $f:\mathcal{X}\subseteq \mathbb{R}^n\rightarrow \mathbb{R}$ is \emph{DR-submodular} if 
\[f(\mathbf{x}+\alpha\mathbf{e}^i)-f(\mathbf{x})\geq f(\mathbf{y}+\alpha\mathbf{e}^i)-f(\mathbf{y})\] 
for every $i\in \{1,2,\dots, n\}$, for all $\mathbf{x},\mathbf{y}\in\mathcal{X}$ with $\mathbf{x}\leq\mathbf{y}$ component-wise, and for all $\alpha\in\mathbb{R}_+$ such that $\mathbf{x}+\alpha\mathbf{e}^i, \mathbf{y}+\alpha\mathbf{e}^i\in\mathcal{X}$. 
\end{definition}
DR-submodular functions encompass a wide range of functions which are highly nonlinear in general. The domain $\mathcal{X}$ can be discrete or continuous. When $\mathcal{X}$ is continuous, the DR-submodular function can be concave, convex, or neither as shown in Figure \ref{fig:venn}. For instance, the quadratic function $f(\mathbf{z}) = -z_1^2 - 13z_1z_2+50z_1+30z_2$ in Figure \ref{fig:nonconvex_nonconcave} is DR-submodular, non-convex, and non-concave. Its DR-submodularity follows from the observation \citep{bian2017guaranteed} that a twice differentiable function $f:\mathcal{X}\rightarrow \mathbb{R}$ is DR-submodular if and only if all of its Hessian entries are non-positive at every $\mathbf{z}\in\mathcal{X}$. We note that such Hessian matrices may be negative semidefinite or indefinite. DR-submodularity trivially holds for any convex function restricted to the domain of a simplex, which is the set $\bigtriangleup := \{\mathbf{z}\in\mathbb{R}^n:\sum_{i=1}^n z_i = 1, \mathbf{z}\geq \mathbf{0}\}$. To see this, for any $\mathbf{z}^1,\mathbf{z}^2 \in \bigtriangleup$, $\mathbf{z}^1 \leq \mathbf{z}^2$ only holds when $\mathbf{z}^1 = \mathbf{z}^2$. For any $i\in \{1,\dots, n\}$, $\mathbf{z}^1 + \alpha \mathbf{e}^i\in\bigtriangleup$ only when $\alpha = 0$. DR-submodular functions have found immense utility in applications, such as optimal budget allocation, revenue management, energy management in sensor networks, stability number of graphs in combinatorial optimization, and mean field inference in machine learning \citep{bian2017guaranteed, motzkin1965maxima, soma2017non, soma2018maximizing, sadeghi2021faster}. Thus, it is important to examine the optimization problems with DR-submodular objective functions. \\
\begin{figure}[ht]
   \centering
   \includegraphics[width=5cm]{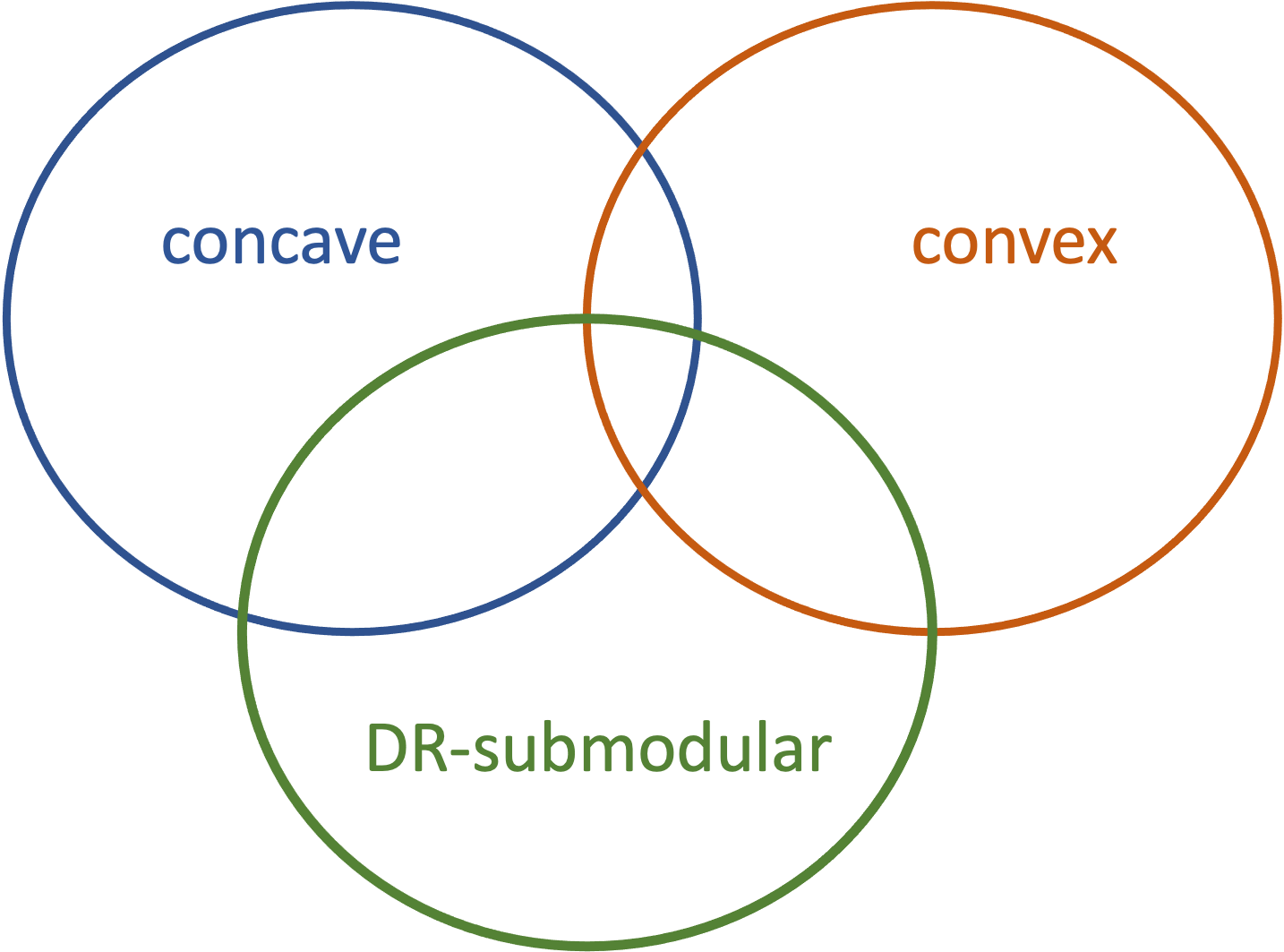} 
   \caption{Convex, concave, and DR-submodular functions (adapted from Figure 1 of \citep{bian2017guaranteed}).}
   \label{fig:venn}
\end{figure}

The current literature has predominantly focused on \emph{maximizing} DR-submodular functions. For example, approximation algorithms with provable guarantees \citep{bian2017guaranteed, bian2017continuous, hassani2017gradient, sadeghi2021faster} and a global optimization method \citep{medal2022spatial} have been proposed for constrained continuous DR-submodular maximization. Studies including \citep{soma2017non, soma2018maximizing} have also proposed approximation algorithms for DR-submodular maximization with integer variables. On the other hand, relatively few studies have considered \emph{minimizing} DR-submodular functions. \citet{ene2016reduction} provide a reduction of DR-submodular optimization with integer variables to submodular set function optimization. As noted by \citep{staib2017robust}, this reduction suggests that minimizing a DR-submodular function with pure integer variables under integer-valued box constraints is polynomial-time solvable, with time complexity dependent on the size of the decision space and logarithmic of the maximal upper bound value of the box constraints. A few works consider functions with another closely related notion of extended submodularity; that is, $f:\mathcal{X}\subseteq \mathbb{R}^n \rightarrow\mathbb{R}$ such that for all $\mathbf{x},\mathbf{y}\in\mathcal{X}$, $f(\mathbf{x}) + f(\mathbf{y}) \geq f(\mathbf{x}\wedge \mathbf{y}) + f(\mathbf{x}\vee \mathbf{y})$, where $\wedge$ and $\vee$ are the component-wise minimum and maximum operators, respectively. Such functions are referred to as \emph{submodular} \citep{bach2019submodular, staib2017robust, bian2017guaranteed, topkis1978minimizing, han2022polynomial} or \emph{lattice submodular} when $\mathcal{X}$ is an integer lattice \citep{ene2016reduction, soma2015generalization, soma2018maximizing}. The notion of (lattice) submodularity subsumes DR-submodularity as shown by \citep{bian2017guaranteed}, while the subclass of  DR-submodular functions has found more utility in real-world scenarios due to its natural diminishing returns property. \citet{bach2019submodular} considers the problem of minimizing (lattice) submodular functions restricted to box constraints and extends the results on Choquet integral by drawing connections with optimal transport. The proposed algorithm has time complexity dependent on the size of the problem as well as the upper bound values of the box constraints, making it pseudo-polynomial. \citet{topkis1978minimizing} explores the parametric (lattice) submodular minimization problems and discusses the properties of the minimizers. Concurrent with our paper, \citet{han2022polynomial} build upon \citep{topkis1978minimizing} and establish the polynomial-time solvability of a class of submodular minimization problems with mixed-binary variables under box constraints that are turned on or off by the binary variables. \\

To this end, it is unknown whether there is a polynomial-time algorithm for minimizing DR-submodular functions with \emph{mixed general integer} variables, under constraints that are beyond box constraints. We bridge this gap by examining DR-submodular minimization under box constraints, and possibly additional monotonicity constraints, with mixed general integer decision variables. Specifically, we would like to propose an algorithm with time complexity solely dependent on the size of the decision space and independent from the values of the upper bounds on the variables. The monotonicity constraints (see Section \ref{sect:prob_desc}) arise in submodular set function minimization over ring families \citep{lee2015faster, orlin2009faster, schrijver2000combinatorial} and monotone systems of linear inequalities (i.e., two variables per inequality with coefficients of opposite signs) \citep{cohen1991improved, shostak1981deciding, aspvall1980polynomial, hochbaum1994simple}. \\

We conduct a polyhedral study on this class of mixed-integer nonlinear optimization problems. The polyhedral approach has demonstrated its effectiveness in attaining global optimal solutions in submodular optimization, especially in the presence of complicating constraints. \citet{edmonds2003submodular} proposes extended polymatroid inequalities and with which establishes an explicit linear convex hull description for the epigraph of any submodular set function in this seminal work. For unconstrained submodular set function maximization, \citet{wolsey1999integer} provide a class of valid linear inequalities for the hypograph of any submodular set function, enabling the reformulation of the original nonlinear program as a mixed-integer linear program. Works including \citet{ahmed2011maximizing,yu2017maximizing,shi2022sequence,yu2017polyhedral,yu2021strong} further strengthen the aforementioned polyhedral results for constrained submodular set function minimization and maximization problems. \citet{yu2020polyhedral, yu2021exact} characterize the convex hulls of the epigraph and hypograph of another class of generalized submodular functions called the $k$-submodular functions (i.e., functions with  $k\geq 2$ set arguments that maintain submodularity), which yield efficient exact solution methods for the class of integer nonlinear optimization problems with $k$-submodular objective functions. \citet{gomez2018submodularity, atamturk2020submodularity, atamturk2020supermodularity, kilincc2020conic, atamturk2008polymatroids, atamturk2019lifted} exploit the underlying submodularity and improve the formulations of mixed-binary convex quadratic and conic optimization problems. The polyhedral approach has also been successfully applied to tackle submodular optimization in stochastic settings \citep{wu2018two, wu2019probabilistic, wu2020exact, kilincc2021joint, xie2019distributionally, zhang2018ambiguous, shen2022chance} as well as minimization of general set functions \citep{atamturk2021submodular}. Next, we provide a summary of our contributions.

\begin{figure}[ht]
   \centering
   \includegraphics[width=9cm]{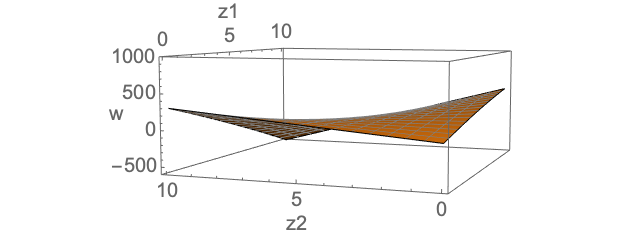} 
   \includegraphics[width=6.5cm]{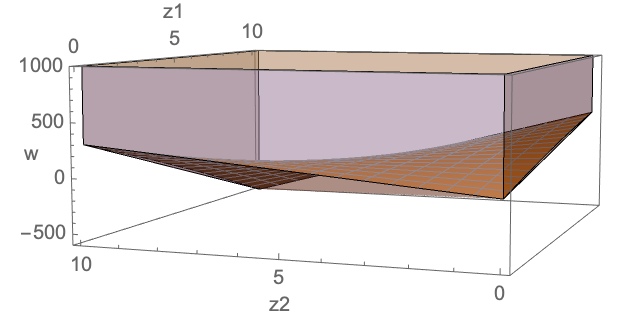} 
   \caption{A continuous DR-submodular function $f(\mathbf{z}) = -z_1^2 - 13z_1z_2+50z_1+30z_2$ that is non-convex and non-concave (left). The convex hull of its epigraph over the box-constrained set $\{\mathbf{z}\in\mathbb{R}\times \mathbb{Z}: 0\leq z_1, z_2 \leq 10\}$ (right).}
   \label{fig:nonconvex_nonconcave}
\end{figure}

\subsection{Our Contributions}
We study the problem of minimizing any DR-submodular function under box constraints, and possibly additional monotonicity constraints, with mixed-integer (including pure integer and pure continuous) decision variables. We introduce the novel \emph{DR-submodular inequalities} and provide the complete convex hull description of the epigraph of any DR-submodular function under the aforementioned constraints. Such a convex hull turns out to be polyhedral (see Figure \ref{fig:nonconvex_nonconcave} for an example). We further propose a polynomial-time exact separation algorithm for our proposed DR-submodular inequalities. In contrast to the existing literature, our approach avoids unary or binary representations of the integer variables. To the best of our knowledge, we are the first to establish the polynomial time complexity of this class of constrained mixed-integer nonlinear optimization problems.

\subsection{Outline}
We describe our problem and introduce our notation and assumptions in Section \ref{sect:prob_desc}. We then explore the properties of DR-submodular functions in Section \ref{sect:prop}. In Section \ref{sect:conv_ZGu}, we provide the complete convex hull description for the mixed-integer feasible set. We further derive helpful properties of such a convex hull in Section \ref{sect:prop_ZGu}. Next, we propose a novel class of linear inequalities, which we call the DR-submodular inequalities, and prove its validity for the epigraph of any DR-submodular function under the constraints of interest in Section \ref{sect:valid}. Lastly, we give the full characterization of the epigraph convex hull and propose an exact separation algorithm for the DR-submodular inequalities in Section \ref{sect:punchline}.

\section{Problem Description}
\label{sect:prob_desc}
Given a DR-submodular function $f:\mathcal{X}\subseteq \mathbb{R}^{n+m}\rightarrow\mathbb{R}$, we consider the minimization problem 
\begin{equation}
\label{eq:DR_min}
\min_{\mathbf{z}\in\mathcal{Z}(\mathcal{G},\mathbf{u})} f(\mathbf{z}), 
\end{equation}
where $\mathcal{Z}(\mathcal{G},\mathbf{u})$ is defined by box constraints and possibly monotonicity constraints: 
\begin{equation}
\label{eq:ZGU}
\mathcal{Z}(\mathcal{G},\mathbf{u}) := \{\mathbf{z}\in\mathbb{Z}^n\times\mathbb{R}^m : \mathbf{0}\leq \mathbf{z}\leq \mathbf{u}, z_i \leq z_j, \: \forall\: (i,j)\in\mathcal{A}\}.
\end{equation}
This feasible set is mixed-integer with $n$ discrete variables (not necessarily binary) and $m$ continuous variables. We let $N = \{i\in \{1,\dots,n+m\} : z_i\in\mathbb{Z}\}$ be the index set for the integer variables and $M = \{1,\dots,n+m\} \backslash N$ be the index set for the continuous variables. We allow $n$ or $m$ to be zero. When $n=0$, $\mathcal{Z}(\mathcal{G},\mathbf{u})$ is a continuous feasible region, and when $m=0$, all variables are discrete. Here, $\mathcal{G} = (\mathcal{V}, \mathcal{A})$ is a directed rooted forest with arcs pointing away from the root of each tree. We formally introduce the terms related to directed rooted forests in the next paragraph. The vertex set $\mathcal{V} = \{1,\dots, n+m\}$ is finite and corresponds to the indices of the decision variables. The arc set $\mathcal{A}\subset \mathcal{V}\times \mathcal{V}$ indicates the partial order on these variables. The vector $\mathbf{u}\in\mathbb{R}^{n+m}$ serves as the upper bounds on the decision variables $\mathbf{z}$; $0\leq u_i \leq \infty$ for every $i\in\mathcal{V}$. We denote the epigraph of $f$ under $\mathcal{Z}(\mathcal{G},\mathbf{u})$ by 
\[\mathcal{P}_f^{\mathcal{Z}(\mathcal{G},\mathbf{u})} := \{(\mathbf{z},w)\in\mathcal{Z}(\mathcal{G},\mathbf{z})\times\mathbb{R}: w\geq f(\mathbf{z})\}.\]
Problem \eqref{eq:DR_min} can be equivalently stated as 
\begin{equation}
\label{eq:DR_min_epi}
\min \left\{w: (\mathbf{z},w)\in\conv{\mathcal{P}_f^{\mathcal{Z}(\mathcal{G},\mathbf{u})}}\right\}.
\end{equation}
Surprisingly, despite the nonlinearity of $f$ and the mixed-integer restrictions of $\mathcal{Z}(\mathcal{G},\mathbf{u})$, we find out that $\conv{\mathcal{P}_f^{\mathcal{Z}(\mathcal{G},\mathbf{u})}}$ is polyhedral, which we will elaborate on in later sections. With our full characterization of $\conv{\mathcal{P}_f^{\mathcal{Z}(\mathcal{G},\mathbf{u})}}$, the mixed-integer nonlinear program \eqref{eq:DR_min} becomes a linear program with continuous variables. \\

We now introduce concepts related to directed rooted forests and set forth relevant notation. A \emph{tree} is an undirected graph where every pair of vertices are connected by exactly one path, and a \emph{forest} is a disjoint union of trees. We note that a tree is connected and contains no cycles. A \emph{directed tree} is a directed acyclic graph whose undirected counterpart is a tree. A tree is \emph{rooted} when one vertex is designated the root. A \emph{directed rooted tree} is a rooted tree with all of its arcs either pointing away or pointing toward its root. In this work, we assume that all the arcs in every directed rooted tree point away from the root. The \emph{height} of a directed rooted tree is the number of arcs in the longest directed path from its root to any vertex. If the directed rooted tree consists of only the root node, then we say its height is zero. A disjoint union of such directed rooted trees, each referred to as a component, is what we call a \emph{directed rooted forest}. Specifically, we assume that all the arcs in a directed rooted forest point away from the roots of its components. \\

We denote a directed rooted forest by $\mathcal{G} = (\mathcal{V},\mathcal{A})$, where $\mathcal{V}$ is the vertex set and $\mathcal{A}\subseteq \mathcal{V}\times\mathcal{V}$ is the arc set. For any $\mathcal{V}'\subset \mathcal{V}$, we define $\mathcal{G}(\mathcal{V}')$ to be a subgraph of $\mathcal{G}$ with vertices $\mathcal{V}'$ and arcs $\mathcal{A}\cap\mathcal{V}'\times\mathcal{V}'$. For any $i\in\mathcal{V}$, we let $R^+(i)\subseteq \mathcal{V}$ be the set of vertices that $i$ can reach along the arcs in $\mathcal{A}$. We call such vertices the \emph{descendants} of $i$. Similarly, we let $R^-(i)\subseteq \mathcal{V}$ denote the vertices that can reach $i$ following the arcs in $\mathcal{A}$, and we call them the \emph{ascendants} of $i$. We note that \emph{$i\in R^+(i)$} and  \emph{$i\in R^-(i)$} for all $i\in\mathcal{V}$, and that $R^+(i),R^-(i)$ belong to the same component in $\mathcal{G}$ because the components are disjoint. The \emph{parent} of $i\in\mathcal{V}$, or $\pa{i}$, is $j\in R^-(i)$ such that $(j,i)\in\mathcal{A}$. Every vertex $i$ has at most one parent because there is only one path from the root node of the component containing $i$ to $i$ itself. We call every $j\in R^+(i)$ with $(i,j)\in\mathcal{A}$ a \emph{child} of $i$. We represent the set of children of $i$ by $\ch{i}$ for any $i\in\mathcal{V}$. The \emph{depth} of any vertex $i\in\mathcal{V}$ is the number of arcs in the path from the root  to $i$ in the component that $i$ belongs to, and we denote it by $\dep{i}$.

\begin{figure}[htbp]
   \centering
   \includegraphics[width=6cm]{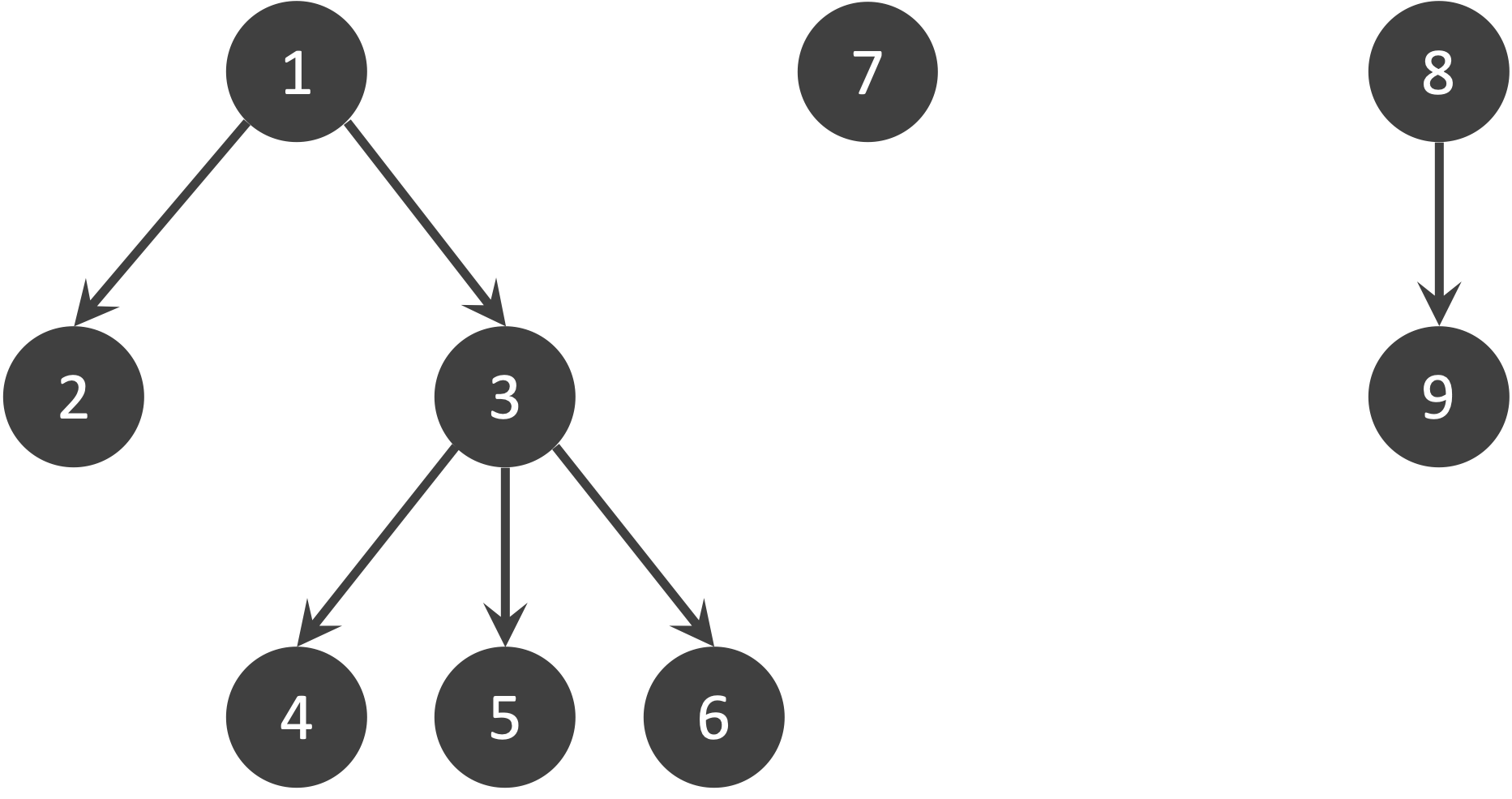} 
   \caption{An example of a directed rooted forest with all the arcs pointing away from the roots. }
   \label{fig:eg_forest}
\end{figure}
\begin{example}
The digraph $\mathcal{G}$ in Figure \ref{fig:eg_forest} is a directed rooted forest with $\mathcal{V} = \{1,2,\dots,9\}$ and $\mathcal{A} = \{(1,2),(1,3),\\(3,4),(3,5),(3,6),(8,9)\}$. It contains three disjoint directed rooted trees with root nodes $1, 7$ and $8$, respectively. All the arcs point away from the roots. The component with root vertex $1$ has height two. Take vertex 3 as an example, we observe that $R^+(3) = \{3,4,5,6\}$, $R^-(3) = \{1,3\}$, and $\dep{3} = 1$. This configuration entails that $z_1\leq z_2$, $z_8\leq z_9$, and $z_1\leq z_3 \leq z_4,z_5,z_6$ in $\mathcal{Z}(\mathcal{G},\mathbf{u})$. Variable $z_7$ has no partial order restrictions in this example, as vertex 7 has no parent or child. \eged
\end{example}

We consider directed rooted forests because directed rooted trees contain no cycles, and the monotonicity constraints do not imply equalities. We note that not all $z_i$ for $i\in \mathcal{V}$ have to be present in the partial order---every vertex in $\mathcal{G}$ is allowed to have no children and no parent. When $\mathcal{A} = \emptyset$, $\mathcal{Z}(\mathcal{G},\mathbf{u})$ reduces to the box-constrained mixed-integer set $\{\mathbf{z}\in\mathbb{Z}^n\times\mathbb{R}^m : \mathbf{0}\leq \mathbf{z}\leq \mathbf{u}\}$. For any $\alpha\in\mathbb{R}$, we let $\floor{\alpha}$ and $\ceil{\alpha}$ denote the floor and ceiling of $\alpha$, respectively; in particular, when $\alpha\in\mathbb{Z}$, we let $\floor{\alpha} = \alpha -1$ and $\ceil{\alpha} = \alpha$. We assume that $f$ is defined over $\mathcal{X}\supseteq \{\mathbf{z}\in\mathbb{R}^{|\mathcal{V}|}: \mathbf{0} \leq \mathbf{z} \leq \mathbf{u}, z_i \leq z_j, \: \forall\: (i,j)\in\mathcal{A}\}$, which is a set containing the continuous relaxation of $\mathcal{Z}(\mathcal{G},\mathbf{u})$. \\

Without loss of generality, we assume that the following holds in problem \eqref{eq:DR_min}. Given any $i\in N$, we let $u_i\in\mathbb{Z}$ because we may round its fractional upper bound down to the nearest integer. We also let $u_i >0$ for all $i\in \mathcal{V}$ to rid the trivial variables.  Naturally, we assume that the upper bounds follow the same partial order of $\mathcal{G}$; that is, if $(i,j)\in\mathcal{A}$, then $u_i \leq u_j$. Furthermore, we assume that $f(\mathbf{0}) = \mathbf{0}$. This can be achieved by shifting $f(\cdot)$ by a constant. \\

When problem \eqref{eq:DR_min} has a finite optimal objective value, we may assume $u_i < \infty$ for all $i\in\mathcal{V}$ without loss of generality. Suppose there exists $i\in \mathcal{V}$ such that $u_i = \infty$. Given that $\mathbf{u}$ follows the partial order imposed by $\mathcal{G}$, all descendants of $i$ are unbounded from above. That is, $u_j = \infty$ for all $j\in R^+(i)$. If problem \eqref{eq:DR_min} has a finite optimal objective value, then we may impose finite upper bounds on $j\in R^+(i)$ without affecting the optimal solutions. Let $\overline{\mathbf{z}}$ be an optimal solution to problem \eqref{eq:DR_min}. We show that $\mathbf{z}^*$, given by $z^*_k = \overline{z}_k$ for $k\in \mathcal{V}\backslash R^+(i)$ and 
\[z^*_j = \begin{cases}
\overline{z}_j, & \text{if } \overline{z}_j < \ceil{\overline{z}_{\pa{i}}}, \\
\ceil{\overline{z}_{\pa{i}}}, & \text{otherwise,}
\end{cases}
\] 
for $j\in R^+(i)$, is also an optimal solution. Here, $\overline{z}_{\pa{i}} = 0$ when $\pa{i}$ does not exist. Due to feasibility of $\overline{\mathbf{z}}$, $\overline{z}_{\pa{i}} \leq \overline{z}_j$ for all $j\in R^+(i)$. If $\overline{z}_{\pa{i}} \leq \overline{z}_j < \ceil{\overline{z}_{\pa{i}}}$, then $\overline{z}_{\pa{i}}\notin\mathbb{Z}$ and the variable $z_j$ must be continuous. By construction, $\mathbf{z}^*$ is feasible to problem \eqref{eq:DR_min}, and $\overline{z}_j - z^*_j\geq 0$ for $j\in R^+(i)$. We now construct $\mathbf{z}'\in\mathbb{R}^{|\mathcal{V}|}$ such that $z'_k = \overline{z}_k = z^*_k$ for $k\in\mathcal{V}\backslash R^+(i)$. Whereas for $j\in R^+(i)$, $z'_j = \overline{z}_j + (\overline{z}_j - z^*_j)$. If $j\in N$, then $\overline{z}_j, z^*_j\in\mathbb{Z}$, which makes $\mathbf{z}'$ feasible with respect to the integrality constraints. Now consider any $j^1,j^2\in R^+(i)$ such that $(j^1,j^2)\in\mathcal{A}$. If $\overline{z}_{j^1} \leq \overline{z}_{j^2} < \ceil{\overline{z}_{\pa{i}}}$, then $z'_{j^1} = \overline{z}_{j^1} + (\overline{z}_{j^1} - \overline{z}_{j^1}) \leq \overline{z}_{j^2} + (\overline{z}_{j^2} - \overline{z}_{j^2}) = z'_{j^2}$. Similarly, if $\ceil{\overline{z}_{\pa{i}}} \leq \overline{z}_{j^1} \leq \overline{z}_{j^2}$, then $z'_{j^1} = \overline{z}_{j^1} + (\overline{z}_{j^1} - \ceil{\overline{z}_{\pa{i}}}) \leq \overline{z}_{j^2} + (\overline{z}_{j^2} - \ceil{\overline{z}_{\pa{i}}}) = z'_{j^2}$. Lastly, we could have $\overline{z}_{j^1} < \ceil{\overline{z}_{\pa{i}}} \leq \overline{z}_{j^2}$. In this case, $z'_{j^1} = \overline{z}_{j^1} + (\overline{z}_{j^1} - \overline{z}_{j^1}) \leq \overline{z}_{j^2} \leq  \overline{z}_{j^2} + (\overline{z}_{j^2} - \ceil{\overline{z}_{\pa{i}}}) = z'_{j^2}$. Therefore, $\mathbf{z}'$ also satisfies the monotonicity constraints, and it is feasible to problem \eqref{eq:DR_min}. Given that $\overline{\mathbf{z}}$ is a minimizer, $f(\overline{\mathbf{z}}) \leq f(\mathbf{z}^*)$ and $f(\overline{\mathbf{z}}) \leq f(\mathbf{z}')$. We arbitrarily index the elements of $R^+(i)$ such that $R^+(i) = \{j^1, \dots, j^{|R^+(i)|}\}$. We observe that 
\begingroup
\allowdisplaybreaks
\begin{align*}
f(\overline{\mathbf{z}}) - f(\mathbf{z}^*) & = f\left(\mathbf{z}^* + \sum_{\ell=1}^{|R^+(i)|} (\overline{z}_{j^\ell} - z^*_{j^\ell})\mathbf{e}^{j^\ell} \right) - f(\mathbf{z}^*) \\
& = \sum_{\kappa=1}^{|R^+(i)|} \left[f\left(\mathbf{z}^* + \sum_{\ell=1}^{\kappa} (\overline{z}_{j^\ell} - z^*_{j^\ell})\mathbf{e}^{j^\ell} \right) - f\left(\mathbf{z}^* + \sum_{\ell=1}^{\kappa-1} (\overline{z}_{j^\ell} - z^*_{j^\ell})\mathbf{e}^{j^\ell} \right)\right] \\
& \geq \sum_{\kappa=1}^{|R^+(i)|} \left[f\left(\overline{\mathbf{z}} + \sum_{\ell=1}^{\kappa} (\overline{z}_{j^\ell} - z^*_{j^\ell})\mathbf{e}^{j^\ell} \right) - f\left(\overline{\mathbf{z}} + \sum_{\ell=1}^{\kappa-1} (\overline{z}_{j^\ell} - z^*_{j^\ell})\mathbf{e}^{j^\ell} \right)\right]  \\
& \quad \text{(by DR-submodularity of $f$)} \\
& = f(\mathbf{z}') - f(\overline{\mathbf{z}}),
\end{align*}
\endgroup
which means that $2f(\overline{\mathbf{z}}) \geq f(\mathbf{z}^*) + f(\mathbf{z}')$. We conclude that $f(\overline{\mathbf{z}}) = f(\mathbf{z}^*)$. An implication of the discussion above is that we may impose finite upper bounds on $z_j$ for $j\in R^+(i)$, such as $u_j = \ceil{u_{\pa{i}}}$, when problem \eqref{eq:DR_min} has a finite optimal objective value.

\section{Properties of DR-submodular Functions}
\label{sect:prop}
In this section, we provide a few useful properties of the DR-submodular functions. {We state these properties in a series of lemmas whose proofs are in Appendix \ref{sec:app}, along with other  proofs omitted in the main paper.} 

Let $\mathcal{X}_i\subseteq \mathbb{R}$ be a compact set for $i\in\{1,2,\dots,n+m\}$; particularly, $\mathcal{X}_i$ is a closed interval for every $i\in M$. Throughout this section, function $f:\prod_{i=1}^{n+m}\mathcal{X}_i \subseteq \mathbb{Z}^n\times \mathbb{R}^m \rightarrow \mathbb{R}$ is DR-submodular. 

\begin{lemma}
\label{lem:int_scale_increment}
Let any $\mathbf{z}\in\mathcal{X}$ and any index $i\in N$ be given. For all $\alpha,\beta\in\mathbb{Z}$ such that $0 < \alpha < \beta$ and $\mathbf{z}+\alpha \mathbf{e}^i, \mathbf{z}+\beta \mathbf{e}^i\in\mathcal{X}$, 
\[f(\mathbf{z}+\alpha \mathbf{e}^i)-f(\mathbf{z})\geq \frac{\alpha}{\beta}[f(\mathbf{z}+\beta \mathbf{e}^i)-f(\mathbf{z})].\] 
\end{lemma}

\begin{lemma}
\label{lem:con_scale_increment}
Let any $\mathbf{z}\in\mathcal{X}$ and any index $i\in M$ be given. For any $\alpha,\beta\in\mathbb{R}$ with $0 < \alpha < \beta$ such that $\mathbf{z}+\alpha \mathbf{e}^i, \mathbf{z}+\beta \mathbf{e}^i\in\mathcal{X}$, 
\[f(\mathbf{z}+\alpha \mathbf{e}^i)-f(\mathbf{z})\geq \frac{\alpha}{\beta}[f(\mathbf{z}+\beta \mathbf{e}^i)-f(\mathbf{z})].\] 
\end{lemma}

Lemma \ref{lem:con_scale_increment} is the continuous analogue of Lemma \ref{lem:int_scale_increment}. In fact, function $f$ has concavity-like properties along every non-negative or non-positive directions. 

\begin{lemma}
\label{lem:pos_dir_increment}
Let any $\mathbf{z}\in\mathcal{X}$ and any non-negative direction $\mathbf{d} = \sum_{i=1}^{n+m} d_i\mathbf{e}^i \geq \mathbf{0}$ be given. For a real number $0\leq t\leq 1$, such that $\mathbf{z} + \sum_{j=1}^i d_j\mathbf{e}^j, \mathbf{z} + t\sum_{j=1}^i d_j\mathbf{e}^j \in\mathcal{X}$ for all $i\in\{1,\dots,n+m\}$, the inequality 
\[f(\mathbf{z}+t\mathbf{d})-f(\mathbf{z})\geq t[f(\mathbf{z}+\mathbf{d})-f(\mathbf{z})]\] 
is satisfied. 
\end{lemma}

The inequality in Lemma \ref{lem:pos_dir_increment} is equivalent to
\begin{equation*}
\label{eq:p2}
f(\mathbf{z}+\mathbf{d}) - f(\mathbf{z}) \geq 1/(1-t) [f(\mathbf{z}+\mathbf{d}) - f(\mathbf{z}+t\mathbf{d})], 
\end{equation*}
and 
\begin{equation*}
\label{eq:p3}
f(\mathbf{z}+t\mathbf{d}) - f(\mathbf{z}+\mathbf{d}) \geq (1-t)[f(\mathbf{z}) - f(\mathbf{z}+\mathbf{d})].
\end{equation*}

As we apply Lemma \ref{lem:pos_dir_increment} in later sections, the condition $\mathbf{z} + \sum_{j=1}^i d_j\mathbf{e}^j, \mathbf{z} + t\sum_{j=1}^i d_j\mathbf{e}^j \in\mathcal{X}$ for a given $t$ and all $i\in\{1,\dots,n+m\}$ always holds because $f$ is assumed to be defined over the continuous relaxation of $\mathcal{Z}(\mathcal{G}, \mathbf{u})$.

\section{Full Description of $\conv{\mathcal{Z}(\mathcal{G}, \mathbf{u})}$}
\label{sect:conv_ZGu}
 In this section, we provide valid inequalities for $\mathcal{Z}(\mathcal{G}, \mathbf{u})$ and fully characterize its convex hull. We further state the explicit form of the extreme points in $\conv{\mathcal{Z}(\mathcal{G}, \mathbf{u})}$, which is crucial to the complete description of $\conv{\mathcal{P}_f^{\mathcal{Z}(\mathcal{G}, \mathbf{u})}}$. We observe that the matrix form of the linear constraints in $\mathcal{Z}(\mathcal{G}, \mathbf{u})$ is totally unimodular. Thus, when $\mathbf{u}\in\mathbb{Z}^{|\mathcal{V}|}$, the continuous relaxation of $\mathcal{Z}(\mathcal{G}, \mathbf{u})$ is $\conv{\mathcal{Z}(\mathcal{G}, \mathbf{u})}$. \\
 
However, when some upper bounds are non-integer, the box constraints and the monotonicity constraints no longer guarantee the integrality of the discrete variables. For ease of notation, we let 
\[\Psi := \{\psi\in \mathcal{V} : u_\psi \notin\mathbb{Z}, R^+(\psi)\cap N\neq \emptyset\}.\] 
This is a collection of the vertices corresponding to the non-integer upper-bounded variables that each has at least one discrete descendant. While not explicit in the definition, the fractional upper bound implies that $\psi\in M$. We impose Assumption \hyperlink{A1}{1} on this set of vertices. \\

\textbf{Assumption \hypertarget{A1}{1}.} There can be at most one $\psi\in\Psi$ in any directed path of $\mathcal{G}$. For any $\psi\in\Psi$, $\ch{\psi}\subseteq N$. \\

The collection of directed paths in $\mathcal{G}$ consists of every directed path from any root vertex to any leaf in the same component. Assumption \hyperlink{A1}{1} requires that no two elements of $\Psi$ fall along the same directed path. In addition, for any $\psi\in \Psi$, the child/children of $\psi$ correspond to discrete variables. By definition of $\Psi$, $\ch{\psi}\neq\emptyset$. Other descendants of $\psi$, that are not the children, can be either discrete or continuous if they exist. We note that no restrictions are imposed on any fractionally upper-bounded continuous variable $i\in\mathcal{V}$ such that all its descendants, $R^+(i)$, correspond to continuous variables. 

{
\begin{remark}
\label{remark:trivialA1}
Many instances of $\mathcal{Z}(\mathcal{G},\mathbf{u})$ trivially satisfy Assumption \hyperlink{A1}{1}, including the following examples. 
\begin{itemize}
\item[(a)] The feasible set $\mathcal{Z}(\mathcal{G},\mathbf{u})$ is defined by box constraints only. In other words, $\mathcal{A} = \emptyset$ in $\mathcal{G} = (\mathcal{V}, \mathcal{A})$. For any $i\in\mathcal{V}$ with $u_i \notin\mathbb{Z}$, $i \notin N$. When $\mathcal{A} = \emptyset$, $R^+(i) = \{i\}$, and $R^+(i)\cap N = \emptyset$. Thus $\Psi = \emptyset$, and Assumption \hyperlink{A1}{1} holds. 

\item[(b)] All variables are discrete. That is, $M = \emptyset$. By our assumption without loss of generality, $u_i \in\mathbb{Z}$ for all $i\in N$. Therefore, $\Psi$ has to be empty, and Assumption \hyperlink{A1}{1} is satisfied. 

\item[(c)] All variables are continuous. That is, $N = \emptyset$. In this case, $R^+(i)\cap N = \emptyset$ for all $i\in\mathcal{V}$, so $\Psi = \emptyset$. Assumption \hyperlink{A1}{1} is satisfied. 

\item[(d)] The upper bounds $\mathbf{u}\in\mathbb{Z}^{|\mathcal{V}|}$. Each variable $z_i$, $i\in\mathcal{V}$, can be either discrete or continuous. This case subsumes scenario (b). We note that $\Psi = \emptyset$ because $\mathbf{u}\in\mathbb{Z}^{|\mathcal{V}|}$, and Assumption \hyperlink{A1}{1} holds. 
\end{itemize}
\end{remark}
}

We now discuss an important implication of Assumption \hyperlink{A1}{1}. That is, if $\mathcal{G}$ satisfies Assumption \hyperlink{A1}{1}, then Property \ref{p} holds without loss of generality. 
\begin{property}
\label{p} 
For any $\psi\in \Psi$, $\ch{\psi} = \{\rho \}$ such that $\rho  \in N$ and $u_\rho  = \ceil{u_\psi}$. 
\end{property}
This property states that $\psi \in \Psi$, if exists, has a single child $\rho \in N$ whose upper bound is $\ceil{u_\psi}$. If any $\psi\in \Psi$ does not already satisfy Property \ref{p}, we insert a vertex $\rho \in N$ with $u_\rho  =  \ceil{u_\psi}$ between $\psi$ and the original children $\ch{\psi}$ of $\psi$. We will soon show that this update does not affect the solutions to the corresponding problem \eqref{eq:DR_min}. We represent the directed rooted forest after incorporating $\rho $ by $\mathcal{G}^E = (\mathcal{V}^E, \mathcal{A}^E)$, where $\mathcal{V}^E = \mathcal{V}\cup \{\rho \}$ and $\mathcal{A}^E = (\mathcal{A}\backslash \{(\psi, i) : i\in \ch{\psi} \text{ from $\mathcal{G}$}\} )\cup\{(\psi,\rho )\} \cup \{(\rho , j) : j\in \ch{\psi} \text{ from $\mathcal{G}$}\}$. Similarly, we let the extended upper bounds be $\mathbf{u}^E$. We illustrate Property \ref{p} and the aforementioned construction of $\mathcal{G}^E, \mathbf{u}^E$ in Example \ref{eg:p}. 

\begin{figure}[htbp]
   \centering
   \includegraphics[width=10cm]{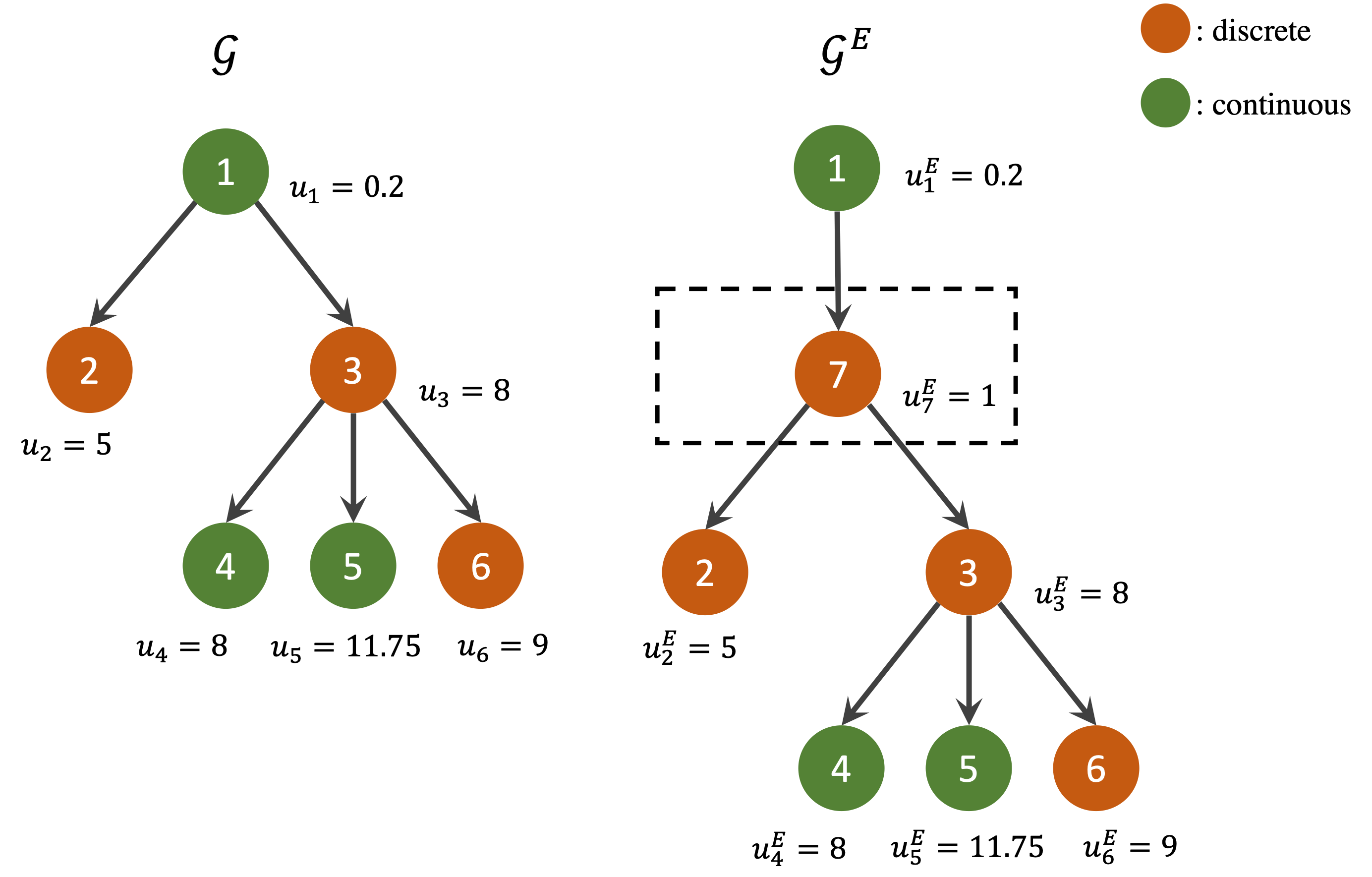} 
   \caption{An illustration of $\mathcal{G},\mathbf{u}$ and $\mathcal{G}^E,\mathbf{u}^E$ in Example \ref{eg:p}. }
   \label{fig:frac_tree_eg}
\end{figure}

\begin{example}
\label{eg:p}
Figure \ref{fig:frac_tree_eg} shows two directed rooted forests, $\mathcal{G}$ and $\mathcal{G}^E$, each with a single component. In $\mathcal{G}$, the only vertex associated with a fractional upper bound that has discrete descendants is the root, vertex $1$. The children of vertex 1, namely $2$ and $3$, correspond to discrete variables. Thus $\mathcal{G}$ satisfies Assumption \hyperlink{A1}{1}. In the extended counterpart $\mathcal{G}^E$, an additional vertex $\rho  = 7$ is inserted; it corresponds to a discrete variable with upper bound $u_7 = \ceil{u_1}$. Property \ref{p} holds in $\mathcal{G}^E$ and $\mathbf{u}^E$. \eged
\end{example}

We next show that problem \eqref{eq:DR_min} over the original feasible set $\mathcal{Z}(\mathcal{G}, \mathbf{u})$ is equivalent to another DR-submodular minimization problem over the extended feasible set $\mathcal{Z}(\mathcal{G}^E, \mathbf{u}^E)$. In what follows, we consider a single element $\psi$ of $\Psi$ and let  $\rho$ be the newly added child of $\psi$. For ease of discussion, we argue for the case where only $\rho$ is added to attain $\mathcal{G}^E$. If multiple adjustments need to be made for Property \ref{p}, our argument still holds for every intermediate step at which one vertex is added and in turn addresses the general case. \\ 

Based on the given DR-submodular function $f:\mathcal{X}\subseteq \mathbb{R}^{|\mathcal{V}|} \rightarrow \mathbb{R}$, we construct a new function $F: \mathcal{X}^E\subseteq \mathbb{R}^{|\mathcal{V}\cup\{\rho \}|} \rightarrow \mathbb{R}$. For any $\mathbf{x}\in \mathcal{X}^E$, we define 
\begin{equation}
\label{def:F}
F(\mathbf{x}) := f(\mathbf{z}^\mathbf{x}),
\end{equation}
where $\mathbf{z}^\mathbf{x} := [x_i]_{i\in \mathcal{V}}$. We know that $\mathbf{z}^\mathbf{x}\in\mathcal{X}$ for every $\mathbf{x}\in\mathcal{X}^E$, and $f$ is defined over $\mathcal{X}$, so $F$ is defined. 

\begin{lemma}
\label{lem:FisDR}
Function $F: \mathcal{X}^E\rightarrow\mathbb{R}$ defined by \eqref{def:F} is DR-submodular. 
\end{lemma}

Due to Lemma \ref{lem:FisDR}, the following optimization problem: 
\begin{equation}
\label{eq:DR_min_E}
\min_{\mathbf{x}\in\mathcal{Z}(\mathcal{G}^E, \mathbf{u}^E)} F(\mathbf{x})
\end{equation}
is a DR-submodular minimization problem over a mixed-integer feasible set that satisfies Property \ref{p}. In the following lemma and proposition, we relate the feasible and optimal solutions of the original problem \eqref{eq:DR_min} to those of the extended counterpart \eqref{eq:DR_min_E}. 

\begin{lemma}
\label{lem:feas_AE}
Given any $\overline{\mathbf{z}}\in\mathcal{Z}(\mathcal{G}, \mathbf{u})$, the vector
\[\overline{\mathbf{x}} := \begin{cases}
\overline{z}_i, & i\in \mathcal{V}, \\
 \ceil{\overline{z}_{\psi}}, &  i = \rho , 
\end{cases}\]
is an element of $\mathcal{Z}(\mathcal{G}^E, \mathbf{u}^E)$. Furthermore, for any ${\mathbf{x}}\in\mathcal{Z}(\mathcal{G}^E, \mathbf{u}^E)$, $\mathbf{z}^{{\mathbf{x}}}\in\mathcal{Z}(\mathcal{G}, \mathbf{u})$.
\end{lemma}

In the next proposition, we show that the optimal solutions to \eqref{eq:DR_min} and \eqref{eq:DR_min_E} are also closely connected. 

\begin{proposition}
Suppose $\mathcal{G}$ and $\mathbf{u}$ satisfy Assumption \hyperlink{A1}{1}. If $\hat{\mathbf{x}}$ is an optimal solution to \eqref{eq:DR_min_E} over $\mathcal{Z}(\mathcal{G}^E, \mathbf{u}^E)$, then $\mathbf{z}^{\hat{\mathbf{x}}}$ is an optimal solution to \eqref{eq:DR_min} over $\mathcal{Z}(\mathcal{G}, \mathbf{u})$. Conversely, if $\mathbf{z}^*$ is an optimal solution to \eqref{eq:DR_min}, then 
\[\mathbf{x}^* := \begin{cases}
{z}^*_i, & i\in \mathcal{V}, \\
 \ceil{{z}^*_{\psi}}, &  i = \rho , 
\end{cases}\]
is an optimal solution to \eqref{eq:DR_min_E}. 
\end{proposition}
\begin{proof} 
First, we consider the case where $\hat{\mathbf{x}}$ is optimal in \eqref{eq:DR_min_E}. The solution $\mathbf{z}^{\hat{\mathbf{x}}}$ is feasible in \eqref{eq:DR_min} by Lemma \ref{lem:feas_AE}. For a contradiction, suppose there exists $\hat{\mathbf{z}}\in\mathcal{Z}(\mathcal{G},\mathbf{u})$ such that $f(\hat{\mathbf{z}}) < f(\mathbf{z}^{\hat{\mathbf{x}}})$. We define a vector $\hat{\mathbf{x}}'$ such that $\hat{x}'_i = \hat{z}_i$ for $i\in \mathcal{V}$, and $\hat{x}'_\rho  = \ceil{\hat{z}_{\psi}}$. Feasibility of $\hat{\mathbf{x}}'$ follows from Lemma \ref{lem:feas_AE}. We observe that $F(\hat{\mathbf{x}}') = f(\hat{\mathbf{z}}) < f(\mathbf{z}^{\hat{\mathbf{x}}}) = F(\hat{\mathbf{x}})$, which is a contradiction. For the converse, we first notice that $\mathbf{x}^*\in\mathcal{Z}(\mathcal{G}^E,\mathbf{u}^E)$ due to Lemma \ref{lem:feas_AE}. By construction, $\mathbf{z}^{\mathbf{x}^*} = \mathbf{z}^*$. Suppose there exists $\overline{\mathbf{x}}\in\mathcal{Z}(\mathcal{G}^E,\mathbf{u}^E)$ such that $F(\overline{\mathbf{x}})<F(\mathbf{x}^*)$. Then $f(\mathbf{z}^{\overline{\mathbf{x}}}) = F(\overline{\mathbf{x}})<F(\mathbf{x}^*) = f(\mathbf{z}^{\mathbf{x}^*}) = f( \mathbf{z}^*)$. We again reach a contradiction. 
\end{proof}

Under Assumption \hyperlink{A1}{1}, it is without loss of generality to assume that Property \ref{p} holds in $\mathcal{Z}(\mathcal{G},\mathbf{u})$. Property \ref{p}  gives $\conv{\mathcal{Z}(\mathcal{G},\mathbf{u})}$ a special structure that simplifies the characterization of its extreme points, which are crucial to describing $\conv{\mathcal{P}_f^{\mathcal{Z}(\mathcal{G},\mathbf{u})}}$. In what follows, all $\mathcal{Z}(\mathcal{G},\mathbf{u})$ satisfy Assumption \hyperlink{A1}{1} as well as Property \ref{p}. Every $\psi\in\Psi$ has exactly one child, so by slightly abusing notation, we denote the index of the child of $\psi$ by $\ch{\psi}$. \\

Next, we derive non-trivial valid inequalities for $\mathcal{Z}(\mathcal{G}, \mathbf{u})$. For any $\psi\in\Psi$, recall that $u_\psi\notin\mathbb{Z}$. The constraint $z_\psi \leq z_{\ch{\psi}}$ is equivalent to the inequality $-z_{\ch{\psi}} - (u_\psi - z_\psi) \leq -u_\psi$. The mixed-integer rounding (MIR) inequality: 
\[-z_{\ch{\psi}} - \frac{u_\psi - z_\psi}{1 - (- u_\psi - \floor{-u_\psi})} \leq \floor{-u_\psi},\]
or equivalently
 \[-z_{\ch{\psi}} - \frac{u_\psi - z_\psi}{u_\psi - \floor{u_\psi}} \leq -\ceil{u_\psi},\]
is known to be valid for the set $\{\mathbf{z}\in\mathbb{R}_+^{n+m}: z_{\ch{\psi}}\in\mathbb{Z}, -z_{\ch{\psi}}- (u_\psi - z_\psi) \leq -u_\psi, z_\psi-u_\psi \leq 0\}$ (Proposition 6.1, pg. 243 of \citep{wolsey1999integer}). Thus the MIR inequality is valid for the more restrictive set $\mathcal{Z}(\mathcal{G}, \mathbf{u})$. After rearranging the terms, the MIR inequality becomes 
\[-z_{\ch{\psi}} + \frac{z_\psi}{u_\psi - \floor{u_\psi}} \leq \frac{\floor{u_\psi}(\ceil{u_\psi} - u_\psi)}{u_\psi - \floor{u_\psi}},\]
with which we construct the following set:
\begin{subequations}
\begin{alignat}{2}
\mathcal{CZ} := \{\mathbf{z}\in\mathbb{R}^{n+m}_+: \quad & z_i \leq u_i, && \:\forall\:i\in\mathcal{V}, \label{eq:ub} \\
&z_i - z_j \leq 0, && \:\forall\: (i,j)\in\mathcal{A}, \label{eq:mn} \\
& -z_{\ch{\psi}} + \frac{z_\psi}{u_\psi - \floor{u_\psi}} \leq \frac{\floor{u_\psi}(\ceil{u_\psi} - u_\psi)}{u_\psi - \floor{u_\psi}}, && \:\forall\: \psi\in\Psi  \label{eq:mir} \}.
\end{alignat}
\end{subequations}
The set $\mathcal{CZ}$ is defined by the non-negativity constraints, the box constraints with upper bound $\mathbf{u}$, the monotonicity constraints given by $\mathcal{G}$, and all the MIR inequalities involving $\psi\in\Psi$ and their children. We claim that $\mathcal{CZ} = \conv{\mathcal{Z}(\mathcal{G}, \mathbf{u})}$, and we justify this statement by exploring the extreme points of $\mathcal{CZ}$.  \\

For any $\mathcal{S}\subseteq\mathcal{V}$ and every $i\in\mathcal{V}$, we define
\begin{equation}
\label{eq:i_triangle}
i^\vartriangle(\mathcal{S}) := 
\begin{cases}
\arg\max\{\dep{j}: j\in \mathcal{S}\cap R^-(i)\}, & \text{if } \mathcal{S}\cap R^-(i)\neq\emptyset, \\
0, & \text{otherwise.}
\end{cases}
\end{equation}
That is, $i^\vartriangle(\mathcal{S})$ is an element of $\mathcal{S}$ such that it is the ascendant of $i$ with maximum depth. If no ascendant of $i$ is in $\mathcal{S}$, then we let $i^\vartriangle(\mathcal{S})$ be zero. We note that when $\mathcal{S}\cap R^-(i)\neq\emptyset$, $\arg\max\{\dep{j}: j\in \mathcal{S}\cap R^-(i)\}$ is unique because in the component of $\mathcal{G}$ containing $i$,  $R^-(i)$ is a directed path from the root to $i$. For $\mathcal{S}\subseteq\mathcal{V}$ and $i\in\mathcal{V}$, let
\begin{equation}
\label{eq:sigma_i}
\sigma_i(\mathcal{S}) := \{j\in R^+(i) : j\notin R^+(k), \forall\: k\in\mathcal{S}\cap [R^+(i)\backslash \{i\}]\}. 
\end{equation}
{When $i\in\mathcal{S}$, this set contains all the vertices in $\mathcal{V}$ whose closest (maximum depth) ascendant in $\mathcal{S}$ is $i$. That is, $\sigma_i(\mathcal{S}) = \{j\in\mathcal{V}: j^\vartriangle(\mathcal{S}) = i\}$.} Let $u_0 = 0$. We define a set of points $P(\mathcal{S})\in\mathbb{R}^{|\mathcal{V}|}$ for any $\mathcal{S}\subseteq \mathcal{V}$ by the following:
\begin{equation}
\label{ext}
P(\mathcal{S})_i = 
\begin{cases}
\floor{u_{\psi}}, & \text{if } i^\vartriangle(\mathcal{S}) = \psi\in\Psi, \text{ and } \ch{\psi}\notin\mathcal{S},   \\
u_{i^\vartriangle(\mathcal{S})}, & \text{otherwise,} \\
\end{cases}
\end{equation}
for every $i\in\mathcal{V}$. Intuitively, $\mathcal{S}$ determines the variables that attain the highest possible values allowed by the box constraints and MIR inequalities. The remaining variables assume the minimal values required by the monotonicity constraints. We note that distinct subsets of $\mathcal{V}$ may yield the same $P(\cdot)$. We illustrate the construction of $P(\cdot)$ in the next example. 

\begin{example}
Consider the directed rooted forest $\mathcal{G}^E$ provided in Figure \ref{fig:frac_tree_eg}, with the associated box constraints upper bounds $\mathbf{u}^E$. Suppose $\mathcal{S}=\{1,3,5,6\}$. Then $i^\vartriangle(\mathcal{S}) = 1$ for $i=1,2,7$, $i^\vartriangle(\mathcal{S}) = 3$ for $i=3,4$, and $i^\vartriangle(\mathcal{S}) = i$ for $i=5,6$. Since $\psi = 1 \in\Psi$ and $\ch{\psi} = 7 \notin\mathcal{S}$, $P(\mathcal{S})_j = \floor{0.2} = 0$ for $j\in\sigma_1(\mathcal{S})=\{1,2,7\}$. Thus $P(\mathcal{S}) = [0, 0, 8, 8, 11.75, 9, 0]^\top$. We note that $P(\{1,3,4,5,6\}) = [0, 0, 8, 8, 11.75, 9, 0]^\top = P(\{1,3,5,6\})$. \eged
\end{example}

\begin{lemma}
\label{lem:F_P_feas}
For any $\mathcal{S}\subseteq\mathcal{V}$, $P(\mathcal{S})\in\mathcal{Z}(\mathcal{G},\mathbf{u})$. 
\end{lemma}

We would like to show that the extreme points of $\mathcal{CZ}$ are exactly $\{P(\mathcal{S})\}_{\mathcal{S}\subseteq \mathcal{V}}$. We prove this claim by showing that $\{P(\mathcal{S})\}_{\mathcal{S}\subseteq \mathcal{V}}$ are the optimal solutions to the following linear program with an arbitrary objective $\mathbf{a}\in\mathbb{R}^{|\mathcal{V}|}$:
\begin{subequations}
\label{frac_convZ_lp}
\begin{alignat}{3}
\min \quad & \mathbf{a}^\top\mathbf{z} && &&\\
\text{s.t. }  & z_i \leq u_i, && \: \forall \: i\in\mathcal{V}, && \quad \quad (p_i)\\
& z_i -z_j \leq 0, &&\:\forall \: (i,j)\in \mathcal{A}, && \quad\quad (q_{i,j}) \\
& -z_{\ch{\psi}} + \frac{z_\psi}{u_\psi - \floor{u_\psi}} \leq \frac{\floor{u_\psi}(\ceil{u_\psi} - u_\psi)}{u_\psi - \floor{u_\psi}}, && \:\forall \: \psi\in\Psi, && \quad\quad (r_\psi) \\
& \mathbf{z}\geq \mathbf{0}. && &&
\end{alignat}
\end{subequations} 

Problem \eqref{frac_convZ_lp} belongs to the class of linear programs with two variables per inequality, for which many efficient algorithms have been proposed \citep{cohen1991improved, shostak1981deciding, aspvall1980polynomial, hochbaum1994simple}. However, we are interested in the explicit form of the optimal solutions to problem \eqref{frac_convZ_lp}, which is not available in the existing literature. In what follows, we examine the dual problem of \eqref{frac_convZ_lp} in order to derive the corresponding optimal primal-dual solution pairs in their explicit forms. The dual problem of \eqref{frac_convZ_lp} is 

\begin{subequations}
\label{frac_convZ_dual}
\begin{alignat}{3}
\max \quad & \sum_{i\in\mathcal{V}} u_ip_i + \sum_{\psi\in\Psi}\frac{\floor{u_\psi}(\ceil{u_\psi} - u_\psi)}{u_\psi - \floor{u_\psi}} r_\psi && && \label{obj:frac_chain_dual}\\
\text{s.t. }  & p_i + \sum_{j\in\ch{i}}q_{i,j} - q_{\pa{i},i} \leq a_i, && \:\forall\: i\in\mathcal{V}\backslash\bigcup_{\psi\in\Psi} \{\psi, \ch{\psi}\},  && \label{constr:frac_chain_dual_1}\\
& p_\psi + q_{\psi,\rho } - q_{\pa{\psi},\psi} + \frac{r_\psi}{u_\psi - \floor{u_\psi}} \leq a_\psi, && \:\forall\: \psi\in\Psi, && \label{constr:frac_chain_dual_2}\\
& p_\rho  + \sum_{j\in\ch{\rho }} q_{\rho ,j} - q_{\psi,\rho } - r_\psi \leq a_\rho , && \:\forall\: \rho = \ch{\psi}, \psi\in\Psi, && \label{constr:frac_chain_dual_3}\\
& [\mathbf{p},\mathbf{q},\mathbf{r}] \leq \mathbf{0}. && && \label{constr:frac_chain_dual_nonpos}
\end{alignat}
\end{subequations}

We let $q_{\pa{i},i} = 0$ for any $i\in\mathcal{V}$ such that $\pa{i}$ does not exist.

\begin{lemma}
\label{lem:convZ_primal_feasible}
For every $\mathcal{S}\subseteq \mathcal{V}$, $P(\mathcal{S})$ is feasible to problem \eqref{frac_convZ_lp}. 
\end{lemma}

\begin{observation}
\label{obs:tree_multi_comp}
Suppose $\mathcal{G}$ consists of $g\geq 1$ disjoint components. We remark that any two variables $z_i$ and $z_j$, where $i$ and $j$ belong to two distinct components of $\mathcal{G}$, are not linked by any constraint in problem \eqref{frac_convZ_lp}. We may decompose problem \eqref{frac_convZ_lp} into $g$ subproblems, namely $LP^i$ for $i\in\{1,\dots, g\}$, such that each subproblem $LP^i$ concerns only the variables that correspond to the vertices in the $i$-th component of $\mathcal{G}$. An optimal solution to \eqref{frac_convZ_lp} is the concatenation of the optimal solutions to $LP^i$ for $i\in\{1,\dots, g\}$. 
\end{observation}

Based on the previous observation, it suffices to solve problem \eqref{frac_convZ_lp} when $\mathcal{G}$ has only one component. Throughout the following discussion on the primal-dual solution pairs to problem \eqref{frac_convZ_lp}, we assume that $\mathcal{G}$ is a directed rooted tree. In case there exists $\psi\in\Psi$ in $\mathcal{G}$, we notice that the subtrees $\mathcal{G}(R^+(\psi^1))$ and $\mathcal{G}(R^+(\psi^2))$ are disjoint for any $\psi^1,\psi^2\in\Psi$ such that $\psi^1\neq\psi^2$, due to Assumption \hyperlink{A1}{1}. In Section \ref{sect:subtree}, we first explain how to solve a subproblem with respect to the directed tree $\mathcal{G}(R^+(\psi))$ rooted at $\psi$ for any $\psi\in\Psi$. The solutions to such subproblems will then help us construct an optimal solution to problem \eqref{frac_convZ_lp} over the directed rooted tree $\mathcal{G}$, which we elaborate on in Section \ref{sect:wholetree}.

\subsection{Directed Trees Rooted at $\psi\in\Psi$} 
\label{sect:subtree}
In this section, we study the special case where the root vertex of $\mathcal{G}$ is $\psi\in \Psi$. For any such $\mathcal{G}$, $\Psi = \{\psi\}$ by Assumption \hyperlink{A1}{1}, which means that in the dual problem \eqref{frac_convZ_dual}, $\mathbf{r} = r_\psi \in\mathbb{R}^1$. Thus we drop the subscript $\psi$ of $r$ in this section. We also use $\rho$ to denote $\ch{\psi}$ for brevity. We propose a way to determine $\mathcal{S}^*\subseteq \mathcal{V}$ such that $\overline{\mathbf{z}} = P(\mathcal{S}^*)$ is a minimizer of problem \eqref{frac_convZ_lp}. Let the height of $\mathcal{G}$ be $h$. Given $\psi, \rho \in\mathcal{V}$, we know that $h\geq 1$. Let $\mathcal{S}^{h+1} = \emptyset$. In the order of $d = h, h-1, \dots, 2$, we compute all $s^i$ for $i\in\mathcal{V}$ with $\dep{i} = d$, where
\begin{equation}
\label{si}
s^i = \sum_{j\in \sigma_i(\mathcal{S}^{d+1})} a_j
\end{equation} 
for any $i\in\mathcal{V}$. This is the sum of the objective coefficients corresponding to the vertices among descendants of $i$, including $i$ itself, which cannot be reached by any vertex in $\mathcal{S}^{\dep{i}+1}$. Once we have all $s^i$ for $i\in\mathcal{V}$ with $\dep{i} = d$, we construct
\begin{equation}
\label{Sd}
\mathcal{S}^d := \{i\in\mathcal{V} : \dep{i} = d, s^i < 0\} \cup \mathcal{S}^{d+1}.
\end{equation}
The set $\mathcal{S}^d$ contains $\mathcal{S}^{d+1}$ and further includes vertices with depth $d$ whose $s^i < 0$. The intuition behind constructing $\mathcal{S}^d$ is that, assigning a positive value to $i\in\mathcal{V}$ with $s^i < 0$ is ideal because this results in a positive reduction in the objective value. At the end of the aforementioned procedure, we obtain $\mathcal{S}^2$, which enables us to compute $s^\rho$. Let
\[\mathcal{R} := \begin{cases}
\{\psi\}, & a_\psi < 0, \\
\emptyset, & \text{otherwise.}
\end{cases}\]
Now we construct  $\mathcal{S}^*$ as follows:
\[\mathcal{S}^* := \begin{cases}
\mathcal{S}^2, & -s^{\rho } \leq \sum_{j\in \mathcal{R}} a_j, \hfill \hypertarget{C1}{(C1)}\\
\mathcal{S}^2\cup \mathcal{R}, & \sum_{j\in \mathcal{R}} a_j < -s^{\rho } \leq (u_\psi - \floor{u_\psi})\sum_{j\in \mathcal{R}} a_j, \hfill{\hypertarget{C2}{(C2)}}\\
\mathcal{S}^2\cup \mathcal{R}\cup \{\rho \}, & (u_\psi - \floor{u_\psi})\sum_{j\in \mathcal{R}} a_j < -s^{\rho } \leq 0, \hfill \hypertarget{C3}{(C3)}\\
\mathcal{S}^2\cup \mathcal{R}\cup\{\rho \}, & \text{otherwise.} \hfill \hypertarget{C4}{(C4)}
\end{cases}
\]

Despite the identical form of $\mathcal{S}^*$, (\hyperlink{C3}{C3}) and (\hyperlink{C4}{C4}) are stated separately because their corresponding dual solutions are distinct. We note that 
\[\sum_{j\in \mathcal{R}} a_j = \min(0, a_{\psi}) \leq (u_\psi - \floor{u_\psi})\min(0, a_{\psi}) \leq 0\] because $\min(0, a_{\psi}) \leq 0$ and $0 < u_\psi - \floor{u_\psi} < 1$. We claim that $\overline{\mathbf{z}} = P(\mathcal{S}^*)$ is an optimal solution to problem \eqref{frac_convZ_lp}. We next construct solutions $[\overline{\mathbf{p}}, \overline{\mathbf{q}},\overline{r}]$ to problem \eqref{frac_convZ_dual} in each of the four cases (\hyperlink{C1}{C1})-(\hyperlink{C4}{C4}). Then we show that the proposed primal-dual solution pairs are optimal by strong duality. \\

We first consider (\hyperlink{C1}{C1}) and (\hyperlink{C4}{C4}) jointly and provide the following dual solution---for $i\in\mathcal{V}$, 
\begin{equation}
\label{p14}
\overline{p}_i := \begin{cases}
s^j, & i\in \mathcal{S}^*, \\
0, & \text{otherwise.}
\end{cases}
\end{equation}
For any $(i,j)\in\mathcal{A}$, 
\begin{equation}
\label{q14}
\overline{q}_{i,j} := 
\begin{cases}
-s^j, & j\notin\mathcal{S}^*, \\
0, & \text{otherwise,}
\end{cases}
\end{equation}
and 
\begin{equation}
\label{r14}
\overline{r} = 0.
\end{equation}

\begin{lemma}
\label{lem:feas_14}
The solution \eqref{p14}-\eqref{r14} is feasible to problem \eqref{frac_convZ_dual} in cases (\hyperlink{C1}{C1}) and (\hyperlink{C4}{C4}).
\end{lemma}

\begin{lemma}
\label{lem:obj_14}
In both cases (\hyperlink{C1}{C1}) and (\hyperlink{C4}{C4}), the primal objective values at $\overline{\mathbf{z}} = P(\mathcal{S}^*)$ match the dual objective values at $[\overline{\mathbf{p}}, \overline{\mathbf{q}},\overline{r}]$ given in \eqref{p14}-\eqref{r14}.
\end{lemma}

Next, we consider case (\hyperlink{C2}{C2}). Recall that $\sum_{j\in \mathcal{R}} a_j < -s^{\rho } \leq (u_\psi - \floor{u_\psi})\sum_{j\in \mathcal{R}} a_j$, and $\mathcal{S}^* = \mathcal{S}^2 \cup \mathcal{R}$. We note that $s^{\rho } \geq - (u_\psi - \floor{u_\psi})\sum_{j\in \mathcal{R}} a_j \geq 0$, and $\sum_{j\in \mathcal{R}} a_j  + s^{\rho } < 0$. It must be that $\sum_{j\in \mathcal{R}} a_j = a_\psi < -s^{\rho } \leq 0$, which implies that $\mathcal{R} = \{\psi\}$ and $\mathcal{S}^* = \{\psi\}\cup\mathcal{S}^2$. Moreover, $s^\psi = a_\psi + s^{\rho } < 0$, and $0 \leq s^{\rho } +  (u_\psi - \floor{u_\psi}) a_\psi$. We propose the following dual solution for (\hyperlink{C2}{C2}): 
\begin{equation}
\label{p2}
\overline{p}_i := \begin{cases}
s^i, & i\in \mathcal{S}^2, \\
0, & \text{otherwise.}
\end{cases}
\end{equation}
for $i\in\mathcal{V}$, and for any $(i,j)\in\mathcal{A}$, 
\begin{equation}
\label{q2}
\overline{q}_{i,j} := 
\begin{cases}
-s^\rho  - \overline{r}, & j = \rho , \\
-s^j, & j\notin\mathcal{S}^*\cup \{\rho \}, \\
0, & \text{otherwise,}
\end{cases}
\end{equation}
and 
\begin{equation}
\label{r2}
\overline{r} = \frac{u_\psi - \floor{u_\psi}}{\ceil{u_\psi}- u_\psi}\left(a_\psi + s^\rho \right).
\end{equation}

\begin{lemma}
\label{lem:feas_2}
The solution \eqref{p2}-\eqref{r2} is feasible to problem \eqref{frac_convZ_dual} in case (\hyperlink{C2}{C2}).
\end{lemma}

\begin{lemma}
\label{lem:obj_2}
In case (\hyperlink{C2}{C2}), the primal objective value at $\overline{\mathbf{z}} = P(\mathcal{S}^*)$ matches the dual objective value at $[\overline{\mathbf{p}}, \overline{\mathbf{q}},\overline{r}]$ given in \eqref{p2}-\eqref{r2}.
\end{lemma}

In the last case (\hyperlink{C3}{C3}), $\mathcal{S}^* = \mathcal{S}^2\cup \mathcal{R}\cup \{\rho \}$, and $(u_\psi - \floor{u_\psi})\sum_{j\in \mathcal{R}} a_j < -s^{\rho } \leq 0$. Given that $s^{\rho } \geq 0$ and $(u_\psi - \floor{u_\psi})\sum_{j\in \mathcal{R}} a_j  + s^{\rho } < 0$, we know that $a_\psi < 0$ and $\psi \in \mathcal{S}^*$. In other words, $\mathcal{S}^* = \mathcal{S}^2\cup  \{\psi, \rho \}$. Below is our proposed dual solution for (\hyperlink{C3}{C3}): 
\begin{equation}
\label{p3}
\overline{p}_i := \begin{cases}
(u_\psi - \floor{u_\psi})a_\psi  + s^{\rho }, & i = \rho , \\
s^i, & i\in \mathcal{S}^2, \\
0, & \text{otherwise,}
\end{cases}
\end{equation}
for $i\in\mathcal{V}$, and for any $(i,j)\in\mathcal{A}$, 
\begin{equation}
\label{q3}
\overline{q}_{i,j} := 
\begin{cases}
-s^j, & j\notin\mathcal{S}^*, \\
0, & \text{otherwise,}
\end{cases}
\end{equation}
and 
\begin{equation}
\label{r3}
\overline{r} = (u_\psi - \floor{u_\psi})a_\psi. 
\end{equation}

In the next two lemmas, we show dual feasibility of $[\overline{\mathbf{p}}, \overline{\mathbf{q}},\overline{r}]$ given by \eqref{p3}-\eqref{r3} and compare the primal and dual objective values at the proposed solutions.

\begin{lemma}
\label{lem:feas_3}
The solution \eqref{p3}-\eqref{r3} is feasible to problem \eqref{frac_convZ_dual} in case (\hyperlink{C3}{C3}).
\end{lemma}

\begin{lemma}
\label{lem:obj_3}
In case (\hyperlink{C3}{C3}), the primal objective value at $\overline{\mathbf{z}} = P(\mathcal{S}^*)$ matches the dual objective value at the solution defined by \eqref{p3}-\eqref{r3}.
\end{lemma}

\begin{proposition}
\label{prop:root_opt}
The solution $\overline{\mathbf{z}} = P(\mathcal{S}^*)$, where $\mathcal{S}^*$ is determined by (\hyperlink{C1}{C1})-(\hyperlink{C4}{C4}), is optimal to problem \eqref{frac_convZ_lp} when the root of $\mathcal{G}$ is in $\Psi$. 
\end{proposition}
\begin{proof}
We have shown in Lemma \ref{lem:convZ_primal_feasible} that the primal solution is feasible. By Lemmas \ref{lem:feas_14}-\ref{lem:obj_3}, the proposed dual solutions are feasible and attain dual objective values that equate the corresponding primal objective values. Following from strong duality, the proposed primal and the dual solutions are optimal to problems \eqref{frac_convZ_lp} and \eqref{frac_convZ_dual}, respectively. 
\end{proof}

\subsection{General Directed Rooted Trees}
\label{sect:wholetree}
We utilize the results from the previous section to construct an optimal solution to the general problem \eqref{frac_convZ_lp} over a directed rooted tree $\mathcal{G}$ as follows. First, for every $\psi\in\Psi$, we solve a subproblem of \eqref{frac_convZ_lp} over the subtree rooted at $\psi$ and obtain $\mathcal{S}^*(\psi)$ based on (\hyperlink{C1}{C1})-(\hyperlink{C4}{C4}). We denote the optimal dual solution by $[\mathbf{p}^*(\psi), \mathbf{q}^*(\psi), r^*(\psi)]$. If $\mathcal{G}$ itself is a directed rooted tree with root $\psi\in\Psi$, then we have found the optimal solution. Now suppose $\mathcal{G}$ is not a directed rooted tree with root in $\Psi$. We then create a new directed rooted tree $\mathcal{G}' = (\mathcal{V}', \mathcal{A}')$ by removing $\bigcup_{\psi\in\Psi}\bigcup_{i\in \mathcal{S}^*(\psi)} R^+(i)$ from $\mathcal{G}$. Consider any $\psi'\in\Psi$ such that the subproblem of \eqref{frac_convZ_lp} over the directed tree rooted at $\psi'$ falls under (\hyperlink{C2}{C2}) or (\hyperlink{C3}{C3}). By the discussions on (\hyperlink{C2}{C2}) and (\hyperlink{C3}{C3}) in Section \ref{sect:subtree}, we know that $\psi'\in\mathcal{S}^*(\psi')$. Therefore, $\mathcal{V}'$ does not contain any vertex from $R^+(\psi')$. In other words, if any $i\in\mathcal{V}'$ belongs to $R^+(\psi)$ for any $\psi\in\Psi$, it is implied that $i\notin\mathcal{S}^*(\psi)$, and the subproblem of \eqref{frac_convZ_lp} over the directed tree rooted at $\psi$ falls under either (\hyperlink{C1}{C1}) or (\hyperlink{C4}{C4}).  \\

We let the height of $\mathcal{G}'$ be $h$, and $h\geq 0$. Let $\mathcal{S}^{h+1} = \emptyset$. In the order of $d = h, h-1, \dots, 0$, we compute $s^i$ as in \eqref{si} for $i\in \mathcal{V}'$ with $\dep{i} = d$, and we iteratively build $\mathcal{S}^d$ given by \eqref{Sd}. In the end, we obtain $\mathcal{S}^0$. We claim that $P\left(\mathcal{S}^0 \cup \bigcup_{\psi\in\Psi} \mathcal{S}^*(\psi)\right)$ is an optimal solution to problem \eqref{frac_convZ_lp}. According to Lemma \ref{lem:convZ_primal_feasible}, this proposed primal solution is feasible. We propose the following dual solution $[\overline{\mathbf{p}}, \overline{\mathbf{q}}, \overline{\mathbf{r}}]$: 
\begin{equation}
\label{pgen}
\overline{p}_i := \begin{cases}
s^j, & i\in \mathcal{S}^0, \\
0, & \text{otherwise.}
\end{cases}
\end{equation}
for $i\in\mathcal{V}'$,  
\begin{equation}
\label{qgen}
\overline{q}_{i,j} := 
\begin{cases}
-s^j, & j\notin\mathcal{S}^0, \\
0, & \text{otherwise,}
\end{cases}
\end{equation}
for any $(i,j)\in\mathcal{A}'$, and
\begin{equation}
\label{rgen}
\overline{r}_\psi = 0
\end{equation} 
if $\psi\in \mathcal{V}'$. For all $i\notin\mathcal{V}'$, $i\in R^+(\psi)$ for some $\psi\in\Psi$. We let $\overline{p}_i = p^*(\psi)_i$. For any $\psi\notin \mathcal{V}'$, we let $\overline{r}_\psi = r^*(\psi)$. Lastly, any $(i,j)\in\mathcal{A}$ such that $i,j\notin\mathcal{V}'$ must satisfy $i,j\in R^+(\psi)$ for some $\psi\in\Psi$. We let $\overline{q}_{i,j} = q^*(\psi)_{i,j}$. For any $(i,j)\in\mathcal{A}$ such that $i\in\mathcal{V}'$ while $j\notin\mathcal{V}'$, we let $\overline{q}_{i,j} = 0$. \\

The construction of $\mathcal{S}^0$, which we have just described, closely resembles the construction of $\mathcal{S}^*(\cdot)$ in Section \ref{sect:subtree}. For every $i\in\mathcal{V}'$ that belongs to $R^+(\psi)$ for any $\psi\in\Psi$, we make the following two observations. First, $i\notin\mathcal{S}^0$ because otherwise $i\in\mathcal{S}^*(\psi)$, and $i$ would not have been present in $\mathcal{V}'$. Second, the proposed dual solution \eqref{pgen}-\eqref{rgen} coincides with $[\mathbf{p}^*(\psi), \mathbf{q}^*(\psi), r^*(\psi)]$ in all the entries involving $i$. In the next proposition, we address the feasibility and optimality of this dual solution, and we show optimality of the proposed primal solution using strong duality.

\begin{proposition}
The proposed solution $\overline{\mathbf{z}} = P\left(\mathcal{S}^0 \cup \bigcup_{\psi\in\Psi} \mathcal{S}^*(\psi)\right)$ is an optimal solution to problem \eqref{frac_convZ_lp}. 
\end{proposition}
\begin{proof}
We first verify that the proposed dual solution $[\overline{\mathbf{p}}, \overline{\mathbf{q}}, \overline{\mathbf{r}}]$ for \eqref{frac_convZ_dual} is feasible. The non-positivity constraints are immediately satisfied by construction. For all $i\in\mathcal{V}'$, $\overline{p}_i + \sum_{j\in\ch{i}}\overline{q}_{i,j} - \overline{q}_{\pa{i},i} \leq a_i$ holds by the same argument as in Lemma \ref{lem:feas_14}. The remaining constraints in \eqref{frac_convZ_dual} are satisfied because of Lemmas \ref{lem:feas_14}, \ref{lem:feas_2} and \ref{lem:feas_3}. Thus, $[\overline{\mathbf{p}}, \overline{\mathbf{q}}, \overline{\mathbf{r}}]$ is dual feasible. We next show that the primal objective value at $\overline{\mathbf{z}} = P\left(\mathcal{S}^0 \cup \bigcup_{\psi\in\Psi} \mathcal{S}^*(\psi)\right)$ matches the dual objective value at $[\overline{\mathbf{p}}, \overline{\mathbf{q}}, \overline{\mathbf{r}}]$. Based on the construction of $\mathcal{G}'$ and Lemmas \ref{lem:obj_14}, \ref{lem:obj_2}, and \ref{lem:obj_3}, we have $\sum_{i\in\mathcal{V}\backslash\mathcal{V}'} a_i \overline{z}_i = \sum_{i\in\mathcal{V}\backslash \mathcal{V}'} u_i\overline{p}_i + \sum_{\psi\in\Psi}\floor{u_\psi}(\ceil{u_\psi} - u_\psi)/(u_\psi - \floor{u_\psi}) \overline{r}_\psi$. In addition, $\sum_{i\in\mathcal{V}'} a_i \overline{z}_i = \sum_{i\in\mathcal{V}'} u_i \overline{p}_i$ due to Lemma \ref{lem:obj_14}. Hence, $\overline{\mathbf{z}}$ is an optimal solution to problem \eqref{frac_convZ_lp} with respect to $\mathcal{G}$ by strong duality. 
\end{proof}

\begin{figure}[htbp]
   \centering
   \includegraphics[width=10cm]{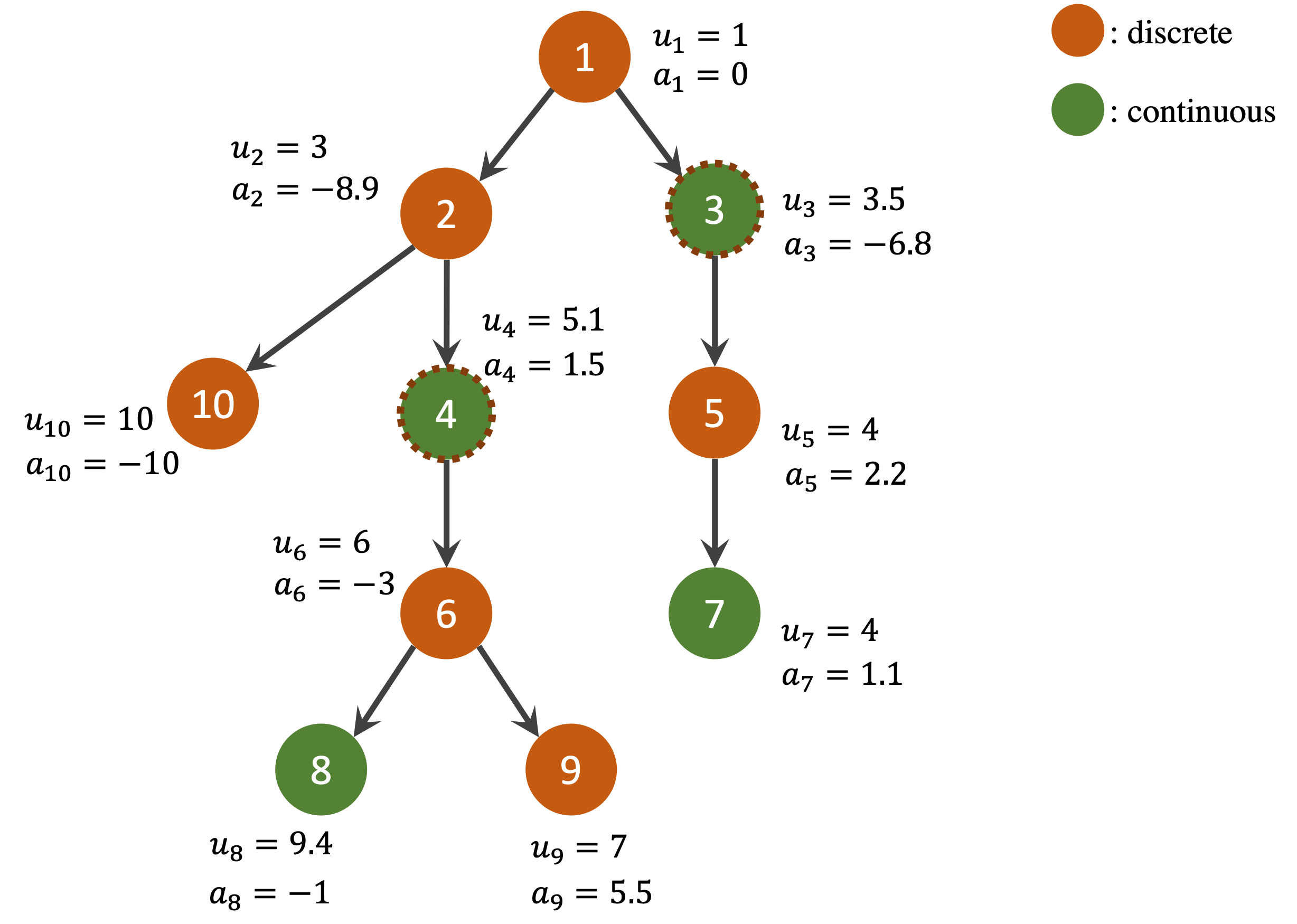} 
   \caption{An instance of problem \eqref{frac_convZ_lp}. }
   \label{fig:lp_demo}
\end{figure}

\begin{example}
Consider an instance of problem \eqref{frac_convZ_lp} provided in Figure \ref{fig:lp_demo}. The set of fractionally upper-bounded continuous variables with discrete descendants is $\Psi = \{3,4\}$. The subproblem over the directed rooted tree comprised of vertices 3,5, and 7 falls under (\hyperlink{C3}{C3}), with $\mathcal{S}^*(3) = \{3, 5\}$, 
\[p^*(3)_5 = -0.1, p^*(3)_3 = p^*(3)_7 = 0, \] 
\[q^*(3)_{3,5} = 0, q^*(4)_{5,7} = -1.1, \] 
and $r^*(3) = -3.4$. The subproblem over the directed rooted tree comprised of vertices 4,6,8, and 9 falls under (\hyperlink{C1}{C1}), with $\mathcal{S}^*(4) = \{8\}$, 
\[p^*(4)_4 = p^*(4)_6 = p^*(4)_9 = 0, p^*(4)_8 = -1, \] 
\[q^*(4)_{4,6} = -4, q^*(4)_{6,8} = 0, q^*(4)_{6,9} = -2.5,\] 
and $r^*(4) = 0$. Now we construct $\mathcal{G}' = (\mathcal{V}', \mathcal{A}')$, where $\mathcal{V}' = \{1,2,4,6,9,10\}$ and $\mathcal{A}' = \{(1,2),(2,10),(2,4),(4,6),\\(6,9)\}$. We obtain $\mathcal{S}^0 = \{2,10\}$. Thus the primal solution we construct is $\overline{\mathbf{z}} = P(\{2,3,5,8,10\}) = [0,3,3.5,3,4,3,\\4,9.4, 3, 10]^\top$ which is optimal.  \eged
\end{example}

For any $P(\mathcal{S})$, $\mathcal{S}\subseteq\mathcal{V}$, we can reversely construct an objective vector $\mathbf{a}\in\mathbb{R}^{|\mathcal{V}|}$, such that $P(\mathcal{S})$ is an optimal solution to the corresponding problem \eqref{frac_convZ_lp}. Let $h$ denote the height of $\mathcal{G}$. In the order of $d=h,h-1,\dots, 1,0$, we may assign any value to $a_i$, for $i\in\mathcal{V}$ with $\dep{i}=d$, to make $s^i$ and $a_\psi$ negative or non-negative as desired. Additionally, we note that even though our discussion above focuses on $\mathcal{G}$ with a single component, our results directly apply to general directed rooted forests $\mathcal{G}$, by Observation \ref{obs:tree_multi_comp}. We now draw a conclusion on the extreme points of $\conv{\mathcal{Z}(\mathcal{G}, \mathbf{u})}$ in the corollary below.

\begin{corollary}
\label{coro:CZ_ext_frac}
The extreme points of $\mathcal{CZ}$, defined by \eqref{eq:ub}-\eqref{eq:mir}, are exactly $\{P(\mathcal{S})\}_{\mathcal{S}\subseteq \mathcal{V}}$ given by \eqref{ext}.
\end{corollary}

\begin{theorem}
\label{thm:convZ_frac}
The set $\mathcal{CZ}$ defined by \eqref{eq:ub}-\eqref{eq:mir} is $\conv{\mathcal{Z}(\mathcal{G},\mathbf{u})}$.  
\end{theorem}
\begin{proof}
The MIR inequalities are valid for $\mathcal{Z}(\mathcal{G},\mathbf{u})$, so $\mathcal{CZ}\supseteq \conv{\mathcal{Z}(\mathcal{G},\mathbf{u})}$. By Corollary \ref{coro:CZ_ext_frac}, $\{P(\mathcal{S}): \mathcal{S}\subseteq \mathcal{V}\}$ are the extreme points of $\mathcal{CZ}$. Lemma \ref{lem:F_P_feas} shows that $P(\mathcal{S}) \in \mathcal{Z}(\mathcal{G},\mathbf{u})$ for any $\mathcal{S}\subseteq \mathcal{V}$, which implies that $\mathcal{CZ}\subseteq \conv{\mathcal{Z}(\mathcal{G},\mathbf{u})}$. Hence, $\mathcal{CZ} = \conv{\mathcal{Z}(\mathcal{G},\mathbf{u})}$.
\end{proof}

{
\begin{remark}
The MIR inequalities are valid for $\mathcal{Z}(\mathcal{G}, \mathbf{u})$ even when Assumption \hyperlink{A1}{1} does not hold. To see this, each MIR inequality is valid for a relaxation of $\mathcal{Z}(\mathcal{G}, \mathbf{u})$, defined by the box and monotonicity constraints only involving the pair of variables in this MIR inequality. If Assumption \hyperlink{A1}{1} fails to hold, then the MIR inequalities are not necessarily facet-defining, and additional classes of non-trivial inequalities are needed to fully describe $\conv{\mathcal{Z}(\mathcal{G}, \mathbf{u})}$.
\end{remark}

\begin{remark}
In some special cases, the MIR inequalities \eqref{eq:mir} are not needed for the full description of $\conv{\mathcal{Z}(\mathcal{G},\mathbf{u})}$. Recall the special cases of $\mathcal{Z}(\mathcal{G},\mathbf{u})$ described in Remark \ref{remark:trivialA1}. The set $\Psi$ is empty in any such instance of $\mathcal{Z}(\mathcal{G},\mathbf{u})$. Therefore, the MIR inequalities are void in the construction of $\mathcal{CZ}$, which is exactly $\conv{\mathcal{Z}(\mathcal{G},\mathbf{u})}$.
\end{remark}
}

Theorem \ref{thm:convZ_frac} implies that any $\overline{\mathbf{z}}\in \conv{\mathcal{Z}(\mathcal{G}, \mathbf{u})}$ can be written as a convex combination of $P(\mathcal{S})$ for all $\mathcal{S}\subseteq \mathcal{V}$. In the next section, we show that for any $\overline{\mathbf{z}} \in \conv{\mathcal{Z}(\mathcal{G}, \mathbf{u})}$, we can determine a subset of $|\mathcal{V}|+1$ extreme points of $\conv{\mathcal{Z}(\mathcal{G}, \mathbf{u})}$, such that $\overline{\mathbf{z}}$ is a convex combination of these particular extreme points. We also provide the coefficients of the convex combination in an explicit form. These special subsets of extreme points and the corresponding coefficients of the convex combinations will appear in the valid inequalities that we propose later.

\section{Properties of $\conv{\mathcal{Z}(\mathcal{G}, \mathbf{u})}$}
\label{sect:prop_ZGu}
Let $\delta = (\delta(1), \delta(2), \dots, \delta(|\mathcal{V}|))$ be any permutation of $\mathcal{V}$. Given any $i\in\mathcal{V}$, we denote its order in $\delta$ by $\delta^{-1}(i)$; in other words, $\delta(\delta^{-1}(i)) = i$. We refer to the set of all permutations of $\mathcal{V}$ by $\mathfrak{S}(\mathcal{V})$. A \emph{partial permutation} of $\mathcal{V}$ is any $\delta'\in\mathfrak{S}(\mathcal{V}')$ for some $\mathcal{V}'\subseteq\mathcal{V}$. With any $\delta\in\mathfrak{S}(\mathcal{V})$, we define $\mathcal{T}^{\delta,k}=\{\delta(1), \dots, \delta(k)\}$ for $k\in\{1,2,\dots,|\mathcal{V}|\}$. In particular, $\mathcal{T}^{\delta, 0} = \emptyset$. The set of extreme points associated with $\delta$ is $\{P(\mathcal{T}^{\delta,k})\}_{k=1}^{|\mathcal{V}|}$. To ease notation, we will use $P(\mathcal{T}^{\delta,k})$ and $P(\delta,k)$ interchangeably. 

\begin{figure}[ht]
   \centering
   \includegraphics[width=9cm]{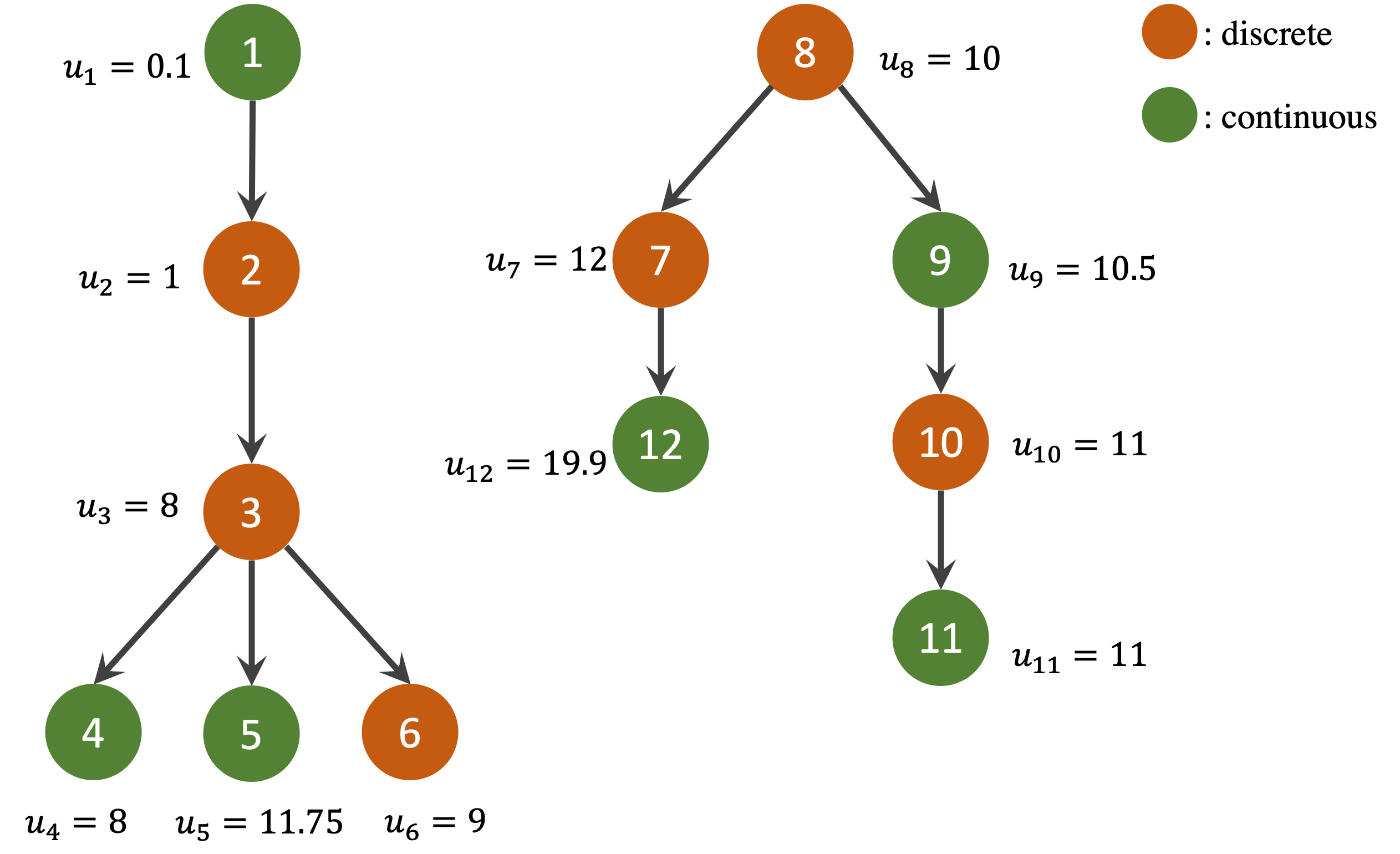} 
   \caption{A directed rooted forest for Example \ref{eg:valid_perm}. }
   \label{fig:valid_perm}
\end{figure}

\begin{example}
\label{eg:delta_extreme_points}
Consider the directed rooted forest $\mathcal{G}$ in Figure \ref{fig:valid_perm} and $\delta = (6,4,7,5,2,3,9,1,11,8,10,12) \in \mathfrak{S}(\mathcal{V})$. By \eqref{ext}, the extreme points of $\conv{\mathcal{Z}(\mathcal{G}, \mathbf{u})}$ associated with $\delta$ are listed in Table \ref{tab:delta_extreme_points}. \eged

\begin{table}[htbp]
   \centering
   \resizebox{\textwidth}{!}{
   \begin{tabular}{c | c | c | cccccccccccc  } 
      \hline
      $k$ & $\delta(k)$ & $\mathcal{T}^{\delta, k}$ & \multicolumn{12}{c}{$P(\mathcal{T}^{\delta,k}) = P(\delta, k)$} \\
      \hline
      0 & -- & $\emptyset$ & [0, & 0, & 0, & 0, & 0, & 0, & 0, & 0, & 0, & 0, & 0, & 0] \\
      1 & 6 & $\{6\}$ & [0, & 0, & 0, & 0, & 0, & 9, & 0, & 0, & 0, & 0, & 0, & 0] \\
      2 & 4 & $\{4,6\}$ & [0, & 0, & 0, & 8, & 0, & 9, & 0, & 0, & 0, & 0, & 0, & 0]\\
      3 & 7 & $\{4,6,7\}$ & [0, & 0, & 0, & 8, & 0, & 9, & 12, & 0, & 0, & 0, & 0, & 12]\\
      4 & 5 & $\{4,5,6,7\}$ & [0, & 0, & 0, & 8, & 11.75, & 9, & 12, & 0, & 0, & 0, & 0, & 12]\\
      5 & 2 & $\{2,4,5,6,7\}$ & [0, & 1, & 1, & 8, & 11.75, & 9, & 12, & 0, & 0, & 0, & 0, & 12]\\
      6 & 3 & $\{2,3,4,5,6,7\}$ & [0, & 1, & 8, & 8, & 11.75, & 9, & 12, & 0, & 0, & 0, & 0, & 12]\\
      7 & 9 & $\{2,3,4,5,6,7,9\}$ & [0, & 1, & 8, & 8, & 11.75, & 9, & 12, & 0, & 10, & 10, & 10, & 12]\\
      8 & 1 & $\{1,2,3,4,5,6,7,9\}$ & [0.1, & 1, & 8, & 8, & 11.75, & 9, & 12, & 0, & 10, & 10, & 10, & 12]\\
      9 & 11 & $\{1,2,3,4,5,6,7,9,11\}$ & [0.1, & 1, & 8, & 8, & 11.75, & 9, & 12, & 0, & 10, & 10, & 11, & 12]\\
      10 & 8 & $\{1,2,3,4,5,6,7,8,9,11\}$ & [0.1, & 1, & 8, & 8, & 11.75, & 9, & 12, & 10, & 10, & 10, & 11, & 12] \\
      11 & 10 & $\{1,2,3,4,5,6,7,8,9,10,11\}$ & [0.1, & 1, & 8, & 8, & 11.75, & 9, & 12, & 10, & 10.5, & 11, & 11, & 12]\\
      12 & 12 & $\{1,2,3,4,5,6,7,8,9,10,11,12\}$ & [0.1, & 1, & 8, & 8, & 11.75, & 9, & 12, & 10, & 10.5, & 11, & 11, & 19.9] \\
      \hline 
   \end{tabular} }
   \caption{The extreme points of $\conv{\mathcal{Z}(\mathcal{G}, \mathbf{u})}$ associated with $\delta$ in Example \ref{eg:delta_extreme_points}.}
   \label{tab:delta_extreme_points}
  \end{table} 
\end{example}

Recall that two distinct subsets of $\mathcal{V}$ may be associated with the same extreme point. To avoid redundancy in $\{P(\delta, k)\}_{k=1}^{|\mathcal{V}|}$, we focus on the permutations that are \emph{valid} as described in the next definition. 

\begin{definition}
\label{def:tree_valid_perm}
A permutation $\delta\in\mathfrak{S}(\mathcal{V})$ is considered \emph{valid} if all conditions below hold. 
\begin{enumerate}
\item[(1)] For any two distinct vertices $i,j\in \mathcal{V}$ such that $j\in R^+(i)$ and $u_i = u_j$, $\delta^{-1}(i) > \delta^{-1}(j)$ is satisfied. 
\item[(2)] If there exists $\psi\in\Psi$ with $\floor{u_\psi} = 0$, then $\delta^{-1}(\ch{\psi}) < \delta^{-1}(\psi)$.
\item[(3)] Consider any $\psi\in\Psi$ such that there exists $i\in R^-(\psi)$ with $u_i = \floor{u_\psi}$. The permutation $\delta$ must \emph{avoid} simultaneously satisfying $\delta^{-1}(i) < \delta^{-1}(\psi)$ and $\delta^{-1}(\psi) < \delta^{-1}(\ch{\psi})$.
\end{enumerate}
\end{definition}

Intuitively, a valid permutation $\delta$ is consistent with the reversed monotonicity order for all the vertices along the same directed path in $\mathcal{G}$, that share the same upper bound. Moreover, when $u_{\ch{\psi}} = 1$ for any $\psi\in\Psi$, $\ch{\psi}$ must precede $\psi$ in a valid permutation. Condition (3) entails that if $\psi\in\Psi$ has any ascendant $i\in\mathcal{V}$ with upper bound $\floor{u_\psi}$, and if $\delta^{-1}(\psi) < \delta^{-1}(\ch{\psi})$, then $i$ must not precede $\psi$ in the valid permutation. It is implied that $i\neq \psi$ because $u_\psi \notin\mathbb{Z}$. 

\begin{example}
\label{eg:valid_perm}
For the directed rooted forest $\mathcal{G}$ in Figure \ref{fig:valid_perm}, the permutation $\tau = (1,2,3,4,5,6,7,8,9,10,11,12)$ is not valid. It violates requirement (1) in Definition \ref{def:tree_valid_perm} because vertices 3 and 4 are connected and $u_3 = u_4$, while $\tau^{-1}(3) < \tau^{-1}(4)$. This permutation also violates requirement (2) because $\psi = 1$ and $\floor{u_1} = 0$; however, $\tau^{-1}(2)  > \tau^{-1}(1)$. Requirement (3) is violated as well. To see this, $9\in\Psi$, $8\in R^-(9)$ with $u_8 = \floor{u_9}$, and $\tau^{-1}(10) > \tau^{-1}(9)$. Whereas, $\tau^{-1}(8) < \tau^{-1}(9)$. In contrast, the permutation $\delta = (6,4,7,5,2,3,9,1,11,8,10,12)$ is valid. \eged
\end{example}

{
\begin{remark}
\label{remark:special_valid_perm}
In the special cases described in Remark \ref{remark:trivialA1}, $\Psi = \emptyset$. Therefore, any permutation $\delta \in \mathfrak{S}(\mathcal{V})$ that satisfies condition (1) in Definition \ref{def:tree_valid_perm} is valid. In particular, in case (a) of Remark \ref{remark:trivialA1}, every $\delta \in \mathfrak{S}(\mathcal{V})$ is a valid permutation. To see this, condition (1) in Definition \ref{def:tree_valid_perm} trivially holds because $\mathcal{A} = \emptyset$ and $R^+(i) = \{i\}$ for all $i\in\mathcal{V}$ in this special case. 
\end{remark}
}

\begin{observation}
For any valid $\delta\in\mathfrak{S}(\mathcal{V})$ and any $k\in\{1,\dots,|\mathcal{V}|\}$, $P(\delta, k) - P(\delta, k-1)$ is a non-negative vector with at least one positive entry. Non-negativity follows from the fact that $\mathcal{T}^{\delta, k-1}\subsetneq \mathcal{T}^{\delta,k}$. The difference contains at least one positive entry because validity of $\delta$ avoids redundancy in the sequence of extreme points associated with $\delta$. {To see this, suppose there exist $k_1, k_2 \in \{1,\dots, |\mathcal{V}|\}$ such that $P(\delta,k_1) = P(\delta,k_2)$ and $k_1 < k_2$. This implies that we have a subsequence of identical extreme points because of component-wise monotonicity. Now consider any consecutive pair of extreme points within this subsequence, $P(\delta,k)$ and $P(\delta,k+1)$, where $k_1 \leq k < k_2$. We note that $P(\delta,k) = P(\delta,k+1)$; in other words, $P(\mathcal{T}^{\delta, k}) = P(\mathcal{T}^{\delta,k} \cup \{\delta(k+1)\})$. The following list enumerates all the possible scenarios under which these two vectors can be identical. 
\begin{enumerate}
\item[(1)] There exists $k' \in \{1,\dots, k\}$ such that $u_{\delta(k+1)} = u_{\delta(k')}$ and $\delta(k+1) \in R^+(\delta(k'))$. 

\item[(2)] The vertex $\delta(k+1)$ satisfies $\delta(k+1) \in\Psi$, $0 < u_{\delta(k+1)} <1$, and $\delta^{-1}(\ch{\delta(k+1)}) > k+1$. 

\item[(3)] The vertex $\delta(k+1)$ satisfies $\delta(k+1) \in\Psi$ and $u_{\delta(k+1)} > 1$. In addition, there exists $i\in R^-(\delta(k+1))$ with $u_i = \floor{u_{\delta(k+1)}}$. The permutation $\delta$ is such that $\delta^{-1}(i) < k+1 < \delta^{-1}(\ch{\delta(k+1)})$. 
\end{enumerate}
By Definition \ref{def:tree_valid_perm}, none of these three scenarios holds when $\delta$ is valid. Therefore, the sequence of extreme points associated with $\delta$ are distinct, given validity of $\delta$. }
\end{observation}

We next show that, given any $\overline{\mathbf{z}} \in \conv{\mathcal{Z}(\mathcal{G}, \mathbf{u})}$, we can determine a valid permutation $\delta\in\mathfrak{S}(\mathcal{V})$ such that $\overline{\mathbf{z}}$ is a convex combination of $\{P(\delta, k)\}_{k=0}^{|\mathcal{V}|}$. We build such a permutation sequentially. In the intermediate steps, it is important to check whether a partial permutation leads to a valid full permutation in the end. Thus, we extend the notion of validity to partial permutations of $\mathcal{V}$. 
\begin{definition}
\label{def:part_valid}
A partial permutation $\delta'\in\mathfrak{S}(\mathcal{V}')$ is \emph{valid} if there exists any valid $\delta\in\mathfrak{S}(\mathcal{V})$ such that $\delta'(i) = \delta(i)$ for $i\in\{1,\dots, |\mathcal{V}'|\}$. 
\end{definition}

Suppose $\delta^k = (\delta(1),\dots,\delta(k))\in\mathfrak{S}(\mathcal{V}')$ where $|\mathcal{V}'| = k < |\mathcal{V}|$ is a valid partial permutation. We provide Algorithm \ref{alg:tree_valid} to construct a set $\Delta$ that contains all $i\in\mathcal{V}\backslash\mathcal{V}'$ such that $(\delta(1),\dots,\delta(k), \delta(k+1)=i)$ is valid.  

\begin{algorithm}[H]
\label{alg:tree_valid}
\SetAlgoLined
\textbf{Input} $\mathcal{G}=(\mathcal{V},\mathcal{A})$, $\mathbf{u}$, and a valid partial permutation $\delta'\in\mathfrak{S}(\mathcal{V}')$\;
$\Delta \leftarrow \mathcal{V}\backslash \mathcal{V}'$\; 
\For{$i \in \mathcal{V}\backslash \mathcal{V}'$}{
    \If{$\exists \: j\in R^+(i)\backslash (\{i\}\cup\mathcal{V}')$ with $u_j = u_i$}{ 
        $\Delta \leftarrow {\Delta}\backslash\{i\}$\;
    }
    \If{$i\in\Psi$}{
        \If{$u_i <1$ and $\ch{i}\notin\mathcal{V}'$}{
            $\Delta \leftarrow {\Delta}\backslash\{i\}$\;
        }
        \If{$\exists\: j\in R^-(i)\cap \mathcal{V}'$ with $u_j = \floor{u_i}$ and $\ch{i}\notin \mathcal{V}'$}{
            $\Delta \leftarrow {\Delta}\backslash\{i\}$\;
        }
    }
   }
\textbf{Output} $\Delta$.
\caption{\texttt{Valid\_Candidates}}
\end{algorithm} 

Algorithm \ref{alg:tree_valid} examines every $i\in\mathcal{V}\backslash \mathcal{V}'$ and excludes $i$ from $\Delta$ if the partial permutation $(\delta'(1), \dots, \delta'(|\mathcal{V}'|), i)$ is not valid by Definition \ref{def:part_valid}. Lines 4-6 of Algorithm  \ref{alg:tree_valid} ensure that condition (1) in Definition \ref{def:tree_valid_perm} is satisfied. Lines 8-11 and lines 11-13 address conditions (2) and (3) in Definition \ref{def:tree_valid_perm}, respectively.

\begin{observation}
\label{obs:Delta}
The output $\Delta$ from Algorithm \ref{alg:tree_valid} is non-empty when $\mathcal{V}'\subsetneq\mathcal{V}$. For a contradiction, suppose $\Delta = \emptyset$. If there exists $i\in\mathcal{V}\backslash\mathcal{V}'$ that is disqualified for $\Delta$ due to line 4 of Algorithm \ref{alg:tree_valid}, then $i\notin\Psi$ by Assumption \hyperlink{A1}{1}. {To see this, $i\in \Psi$  implies that $i$ corresponds to a fractionally upper bounded continuous variable with at least one discrete descendant. By line 4 of Algorithm \ref{alg:tree_valid}, $u_j = u_i\notin \mathbb{Z}$, and due to the monotonicity ordering, $j$ also has at least one discrete descendant just like $i$. These observations suggest that  $j\in\Psi$. However, under Assumption \hyperlink{A1}{1}, no two distinct members of $\Psi$ ($i$ and $j$ in this case) can fall along the same directed path. Therefore, $i\notin\Psi$.} The vertex $\ell = \arg\max\{\dep{j}: j\in R^+(i)\backslash(\{i\}\cup \mathcal{V}'), u_j = u_i\}$ should be in $\Delta$. {We note that $\ell\notin \Psi$ because of Assumption \hyperlink{A1}{1}.} Given that $\Delta = \emptyset$, it must be that every $i\in\mathcal{V}\backslash\mathcal{V}'$ belongs to $\Psi$ and satisfies the condition in either line 8 or line 11. By definition of $\Psi$, $\ch{i}$ exists. We notice that $\ch{i}\notin\mathcal{V}'$, and {$\ch{i} \notin \Psi$ by Assumption \hyperlink{A1}{1}}. Thus it should have been included in $\Delta$. We reach a contradiction. 
\end{observation}

Next, we show that any $\overline{\mathbf{z}}\in\mathbb{R}^{|\mathcal{V}|}$ can be written as an affine combination of $\{P(\mathcal{T}^{\delta,k})\}_{k=0}^{|\mathcal{V}|}$ for any valid permutation $\delta\in\mathfrak{S}(\mathcal{V})$. We now define two functions that are closely related to the affine combination coefficients. For every $\psi\in\Psi$, let $\eta_\psi:\mathbb{R}^{|\mathcal{V}|} \rightarrow\mathbb{R}$ be
\begin{equation}
\eta_\psi(\mathbf{z}) = \frac{z_\psi - (u_\psi - \floor{u_\psi}) z_{\ch{\psi}}}{u_{\ch{\psi}} - u_\psi}.
\end{equation} 
Recall $u_\psi\notin\mathbb{Z}$ and $u_{\ch{\psi}} = \ceil{u_\psi}$, so the denominator of $\eta_\psi(\mathbf{z})$ is strictly positive. Let $z_0 = 0$ and $u_0 = 0$. We further define a function $\mathbf{t}:\mathfrak{S}(\mathcal{V})\times\mathbb{R}^{|\mathcal{V}|}\rightarrow\mathbb{R}^{|\mathcal{V}|}$. For $k\in\{1,\dots,|\mathcal{V}|\}$, we let $i=\delta(k)$ for brevity, and 
\begin{numcases}{t(\delta, \mathbf{z})_k := } 
\frac{\eta_\psi(\mathbf{z}) - z_{i^{\vartriangle}(\mathcal{T}^{\delta, k-1})}}{\floor{u_\psi}- u_{i^\vartriangle(\mathcal{T}^{\delta, k-1})}}, & if $i = \psi\in\Psi$ and $\ch{\psi}\notin\mathcal{T}^{\delta, k-1}$, \label{t_1} \\
\frac{z_i - \eta_\psi(\mathbf{z})}{u_i - \floor{u_\psi}}, & else if $i^\vartriangle(\mathcal{T}^{\delta, k-1}) = \psi\in\Psi$, \label{t_2} \\
\frac{z_i- z_{i^{\vartriangle}(\mathcal{T}^{\delta, k-1})}}{u_i - u_{i^\vartriangle(\mathcal{T}^{\delta, k-1})}}, & otherwise.  \label{t_3} 
\end{numcases}
In addition, we let $t(\delta, {\mathbf{z}})_0 = 1$ and $t(\delta, {\mathbf{z}})_{|\mathcal{V}|+1} = 0$. \\

In case \eqref{t_2}, $i^\vartriangle(\mathcal{T}^{\delta, k-1}) = \psi\in\Psi$ implies that $\ch{\psi}\notin\mathcal{T}^{\delta, k-1}$. The conditions under which \eqref{t_3} applies are (1) $i = \psi\in\Psi$ and $\ch{\psi}\in\mathcal{T}^{\delta, k-1}$, or (2) $i, i^\vartriangle(\mathcal{T}^{\delta, k-1})\notin \Psi$. A partial permutation $\delta' = (\delta(1),\dots,\delta(k))$ is sufficient to evaluate $t(\delta, \mathbf{z})_k$ for any $\mathbf{z}\in\mathbb{R}^{|\mathcal{V}|}$ and $k\in\{1,\dots,|\mathcal{V}|\}$. Thus by abusing notation, we allow the first argument of $t(\cdot,\cdot)_k$ to be any valid partial permutation $\delta'\in\mathfrak{S}(\mathcal{V}')$ where $|\mathcal{V}'|\geq k$. 

{
\begin{remark}
The function $\mathbf{t}(\cdot, \cdot)$ can be simplified in special cases of $\mathcal{Z}(\mathcal{G},\mathbf{u})$. Recall cases (a)--(d) described in Remark \ref{remark:trivialA1}. In case (a), the feasible set $\mathcal{Z}(\mathcal{G},\mathbf{u})$ is defined by box constraints only. As mentioned in Remark \ref{remark:special_valid_perm}, any $\delta\in\mathfrak{S}(\mathcal{V})$ is valid in this case. For $k\in\{1,\dots,|\mathcal{V}|\}$, $t(\delta,\mathbf{z})_k := z_{\delta(k)}/u_{\delta(k)}$. This is because $\Psi = \emptyset$ as discussed in Remark \ref{remark:trivialA1}. Moreover, $\mathcal{A} = \emptyset$, so $\delta(k)^{\vartriangle}(\mathcal{T}^{\delta, k-1}) = 0$ by definition, and $z_0 = u_0 = 0$. In cases (b), (c), and (d) of Remark \ref{remark:trivialA1}, $\Psi = \emptyset$ as well. Therefore, $t(\delta,\mathbf{z})_k := (z_{\delta(k)}- z_{{\delta(k)}^{\vartriangle}(\mathcal{T}^{\delta, k-1})})/(u_{\delta(k)} - u_{{\delta(k)}^\vartriangle(\mathcal{T}^{\delta, k-1})})$ for all $k\in\{1,\dots,|\mathcal{V}|\}$. 
\end{remark}
}

\begin{observation}
\label{obs:denom}
The denominators in all cases of $t(\cdot, \cdot)_k$ are strictly positive for $k\in\{1,\dots,|\mathcal{V}|\}$. Recall that $i = \delta(k)$. By definition, $i^\vartriangle(\cdot) \in R^-(i)$ so the denominators must be non-negative. In \eqref{t_1}, $\delta^{-1}(\ch{\psi}) > \delta^{-1}(\psi)$. Property (2) from Definition \ref{def:tree_valid_perm} of $\delta$ ensures that when $i\in\Psi$ with $\floor{u_i} = 0$, $t(\delta,\mathbf{z})_k$ never falls under \eqref{t_1}. Property (3) in Definition \ref{def:tree_valid_perm} guarantees that $\floor{u_\psi} \neq u_{i^\vartriangle(\mathcal{T}^{\delta, k-1})}$. In \eqref{t_2}, $i$ is a descendant of $\psi$, so $u_i \geq \ceil{u_\psi} = u_{\ch{\psi}}$. By property (1) (see Definition \ref{def:tree_valid_perm}) of $\delta$, $u_i \neq u_{i^\vartriangle(\mathcal{T}^{\delta, k-1})}$. 
\end{observation}

\begin{example}
\label{eg:t_eg}
Consider the directed rooted forest $\mathcal{G}$ in Figure \ref{fig:valid_perm} and the valid permutation $\delta = (6,4,7,5,2,3,9,1,\\11,8,10,12)$. We notice that vertices 1 and 9 are the only non-integer upper-bounded continuous variables that have discrete descendants, so $\Psi = \{1, 9\}$. The form of $\mathbf{t}(\delta,\mathbf{z})$ is provided in Table \ref{tab:t_eg}.  \eged
\begin{table}[htbp]
   \centering
   \begin{tabular}{c | c | l} 
      \hline
      $k$ & $\delta(k)$ & $t(\delta,\mathbf{z})_k$  \\
      \hline
      1 & 6 & $z_6/u_6 = z_6/9$  \\
      2 & 4 & $z_4/u_4 = z_4/8$ \\
      3 & 7 & $z_7/u_7 = z_7/12$ \\
      4 & 5 & $z_5/u_5 = z_5/11.75$ \\
      5 & 2 & $z_2/u_2 = z_2$ \\
      6 & 3 & $(z_3 - z_2)/(u_3-u_2) = (z_3-z_2)/7$ \\
      7 & 9 & $\eta_9(\mathbf{z})/\floor{u_9} = (z_9-0.5 z_{10})/5$ \\
      8 & 1 & $z_1/u_1 = 10 z_1$ \\
      9 & 11 & $[z_{11} - \eta_9(\mathbf{z})]/(u_{11} - \floor{u_9}) = z_{11} - 2z_9 + z_{10}$ \\
      10 & 8 & $z_8/u_8 = z_8/10$  \\
      11 & 10 & $[z_{10} - \eta_9(\mathbf{z})]/(u_{10} - \floor{u_9}) = 2z_{10} - 2z_9$ \\
      12 & 12 & $(z_{12} - z_7)/(u_{12} - u_7) = (z_{12} - z_7)/7.9$ \\
      \hline
   \end{tabular}
   \caption{Form of $\mathbf{t}(\delta,\mathbf{z})$ in Example \ref{eg:t_eg}.}
   \label{tab:t_eg}
  \end{table} 
\end{example}

\begin{observation}
\label{obs:ch_t}
Suppose $i^\vartriangle(\mathcal{T}^{\delta, k-1}) = \psi\in\Psi$ as in \eqref{t_2}. If $i = \ch{\psi}$, then 
\begin{align*}
t(\delta,\mathbf{z})_k  = z_{\ch{\psi}} - \frac{z_\psi - (u_\psi - \floor{u_\psi}) z_{\ch{\psi}}}{u_{\ch{\psi}} - u_\psi}  = \frac{(u_{\ch{\psi}} - u_\psi + u_\psi - \floor{u_\psi}) z_{\ch{\psi}} - z_\psi}{u_{\ch{\psi}} - u_\psi}  = \frac{z_{\ch{\psi}} - z_\psi}{u_{\ch{\psi}} - u_\psi},
\end{align*} which coincides with \eqref{t_3}.
\end{observation}

Function $\mathbf{t}(\delta,\mathbf{z})$ can be expressed as $\mathbf{T}^\delta\mathbf{z}$, where $\mathbf{T}^\delta\in\mathbb{R}^{|\mathcal{V}|\times|\mathcal{V}|}$. For $k\in\{1,\dots,|\mathcal{V}|\}$, we again let $i=\delta(k)$ for brevity. The $k$-th row of matrix $\mathbf{T}^\delta$ is defined as follows.

\begin{enumerate}
\item[(i)] If $i = \psi\in\Psi$ and $\ch{\psi}\notin\mathcal{T}^{\delta, k-1}$ as in \eqref{t_1}, then
\begin{equation}
\label{T1}
\mathbf{T}^\delta_{k,j} = 
\begin{cases}
1/[(u_{\ch{\psi}} - u_\psi)(\floor{u_\psi} - u_{i^\vartriangle(\mathcal{T}^{\delta, k-1})})], & j = \psi, \\
-1/(\floor{u_\psi} - u_{i^\vartriangle(\mathcal{T}^{\delta, k-1})}), & j = i^\vartriangle(\mathcal{T}^{\delta, k-1}), \\
-(u_\psi - \floor{u_\psi})/[(u_{\ch{\psi}} - u_\psi)(\floor{u_\psi} - u_{i^\vartriangle(\mathcal{T}^{\delta, k-1})})], & j = \ch{\psi}, \\
0, & \text{otherwise.} 
\end{cases}
\end{equation}

\item[(ii)] Else suppose $i^\vartriangle(\mathcal{T}^{\delta, k-1}) = \psi\in\Psi$ as in \eqref{t_2}. When $i = \ch{\psi}$, by Observation \ref{obs:ch_t},
\begin{equation}
\label{T2ch}
\mathbf{T}^\delta_{k,j} := 
\begin{cases}
1/(u_{\ch{\psi}} - u_\psi), & j = \ch{\psi}, \\
-1/(u_{\ch{\psi}} - u_\psi), & j = \psi, \\
0, & \text{otherwise.} 
\end{cases}
\end{equation}
If $i \neq \ch{\psi}$, then 
\begin{equation}
\label{T2}
\mathbf{T}^\delta_{k,j} := 
\begin{cases}
1/(u_i - \floor{u_\psi}), & j = i, \\
-1/[(u_{\ch{\psi}}-u_\psi)(u_i - \floor{u_\psi})], & j = \psi, \\
(u_\psi - \floor{u_\psi})/[(u_{\ch{\psi}}-u_\psi)(u_i - \floor{u_\psi})], & j = \ch{\psi}, \\
0, & \text{otherwise.} 
\end{cases}
\end{equation}

\item[(iii)] Otherwise, $i$ and $i^\vartriangle(\mathcal{T}^{\delta, k-1})$ do not belong to $\bigcup_{\psi\in\Psi} \{\psi, \ch{\psi}\}$ as in \eqref{t_3}. Then
\begin{equation}
\label{T3}
\mathbf{T}^\delta_{k,j} = 
\begin{cases}
1/(u_i - u_{i^\vartriangle(\mathcal{T}^{\delta, k-1})}), & j = i, \\
-1/(u_i - u_{i^\vartriangle(\mathcal{T}^{\delta, k-1})}), & j = i^\vartriangle(\mathcal{T}^{\delta, k-1}), \\
0, & \text{otherwise.} 
\end{cases}
\end{equation}
\end{enumerate}

Recall $P(\delta, 0) = \mathbf{0}$. Given a valid permutation $\delta\in\mathfrak{S}(\mathcal{V})$, we define $\mathbf{D}^\delta\in \mathbb{R}^{|\mathcal{V}|\times |\mathcal{V}|}$ by 
\begin{equation}
\mathbf{D}^\delta_{\cdot, k} := P(\delta,k) - P(\delta, k-1), 
\end{equation}
for $k\in\{1,\dots, |\mathcal{V}|\}$.

\begin{observation}
\label{obs:p_inv_diag}
For any valid $\delta\in\mathfrak{S}(\mathcal{V})$ and $k\in\{1,\dots,|\mathcal{V}|\}$, $\sum_{j=1}^{|\mathcal{V}|}\mathbf{T}^\delta_{k,j}\mathbf{D}^\delta_{j,k} = 1$. We explain this observation by case. For consistency, we let $i = \delta(k)$. 

\begin{enumerate}
\item[(i)] If $i = \psi\in\Psi$ and $\ch{\psi}\notin\mathcal{T}^{\delta, k-1}$, then 
\begin{align*}
\sum_{j=1}^{|\mathcal{V}|}\mathbf{T}^\delta_{k,j}\mathbf{D}^\delta_{j,k} & = \frac{\floor{u_\psi} - u_{i^\vartriangle(\mathcal{T}^{\delta, k-1})}}{(u_{\ch{\psi}} - u_\psi)(\floor{u_\psi} - u_{i^\vartriangle(\mathcal{T}^{\delta, k-1})})} - \frac{0}{\floor{u_\psi}  - u_{i^\vartriangle(\mathcal{T}^{\delta, k-1})}} \\
& \quad - \frac{(\floor{u_\psi} - u_{i^\vartriangle(\mathcal{T}^{\delta, k-1})})(u_\psi - \floor{u_\psi})}{(u_{\ch{\psi}} - u_\psi)(\floor{u_\psi} - u_{i^\vartriangle(\mathcal{T}^{\delta, k-1})})} \\
& = \frac{1 - u_\psi + \floor{u_\psi}}{u_{\ch{\psi}} - u_\psi} \\
& = 1. 
\end{align*}

\item[(ii)] Now consider the case where $i^\vartriangle(\mathcal{T}^{\delta, k-1}) = \psi\in\Psi$. When $i = \ch{\psi}$, 
\begin{align*}
\sum_{j=1}^{|\mathcal{V}|}\mathbf{T}^\delta_{k,j}\mathbf{D}^\delta_{j,k}  = \frac{1}{u_{\ch{\psi}} - u_\psi} - \frac{u_\psi - \floor{u_\psi}}{u_{\ch{\psi}} - u_\psi}  = \frac{1 + \floor{u_\psi} - u_\psi}{u_{\ch{\psi}} - u_\psi} = 1.
\end{align*}
When $i \neq \ch{\psi}$, 
\begin{align*}
\sum_{j=1}^{|\mathcal{V}|}\mathbf{T}^\delta_{k,j}\mathbf{D}^\delta_{j,k} = \frac{u_i - \floor{u_\psi}}{u_i - \floor{u_\psi}} + 0 = 1.
\end{align*}

\item[(iii)] Otherwise, $\sum_{j=1}^{|\mathcal{V}|}\mathbf{T}^\delta_{k,j}\mathbf{D}^\delta_{j,k} = \mathbf{T}^\delta_{k,i} (u_i - u_{i^\vartriangle(\mathcal{T}^{\delta, k-1})}) + 0 = 1$.
\end{enumerate}
\end{observation}

We next argue that, given any valid $\delta\in\mathfrak{S}(\mathcal{V})$, $\sum_{j=1}^{|\mathcal{V}|}\mathbf{T}^\delta_{k,j}\mathbf{D}^\delta_{j,k'} = 0$ for all $k, k'\in\{1,\dots,|\mathcal{V}|\}$ such that $k \neq k'$. Suppose $\delta(k)$ and $\delta(k')$ belong to two disjoint components of $\mathcal{G}$, $G^1 = (V^1, A^1)$ and $G^2 = (V^2, A^2)$ respectively. Then $\mathbf{D}^\delta_{j,k'} = 0$ for all $j\notin V^2$, and $\mathbf{T}^\delta_{k,\ell} = 0$ for $\ell\in V^2$. Therefore, $\sum_{j=1}^{|\mathcal{V}|}\mathbf{T}^\delta_{k,j}\mathbf{D}^\delta_{j,k'} = \sum_{j\in V^2}\mathbf{T}^\delta_{k,j}\mathbf{D}^\delta_{j,k'} + \sum_{j\notin V^2}\mathbf{T}^\delta_{k,j}\mathbf{D}^\delta_{j,k'} = 0$. In Observations \ref{obs:p_inv_off_diag_lowkprime} and \ref{obs:p_inv_off_diag_highkprime}, we focus on the cases in which $\delta(k)$ and $\delta(k')$ belong to the same component of $\mathcal{G}$.

\begin{observation}
\label{obs:p_inv_off_diag_lowkprime}
Given any valid $\delta\in\mathfrak{S}(\mathcal{V})$ and $k, k'\in\{1,\dots,|\mathcal{V}|\}$ such that $k > k'$, we note that $\sum_{j=1}^{|\mathcal{V}|}\mathbf{T}^\delta_{k,j}\mathbf{D}^\delta_{j,k'} = 0$. In what follows, we again let $i = \delta(k)$. 

\begin{enumerate}
\item[(i)] In the case of \eqref{T1}, $i = \psi\in\Psi$ and $\delta^{-1}(\ch{\psi}) > \delta^{-1}(\psi)$. Given that $k' < k$, $\sigma_{\delta(k')}(\mathcal{T}^{\delta, k'})$ either contains all of $\psi$, $\ch{\psi}$ and $i^\vartriangle(\mathcal{T}^{\delta, k-1})$, or none of them. This is because otherwise there would exist a higher depth ascendant of $i$ than $i^\vartriangle(\mathcal{T}^{\delta, k-1})$ that precedes $i$ in $\delta$, which cannot happen. Therefore, $\mathbf{D}^\delta_{\psi,k'} = \mathbf{D}^\delta_{\ch{\psi},k'} =\mathbf{D}^\delta_{i^\vartriangle(\mathcal{T}^{\delta, k-1}),k'}$. It follows that 
\begin{align*}
\sum_{j=1}^{|\mathcal{V}|}\mathbf{T}^\delta_{k,j}\mathbf{D}^\delta_{j,k'} & = \frac{\mathbf{D}^\delta_{i,k'}}{(u_{\ch{\psi}} - u_\psi)(\floor{u_\psi} - u_{i^\vartriangle(\mathcal{T}^{\delta, k-1})})} -\frac{\mathbf{D}^\delta_{i,k'}}{\floor{u_\psi} - u_{i^\vartriangle(\mathcal{T}^{\delta, k-1})}}\\
& \quad -\frac{\mathbf{D}^\delta_{i,k'}(u_\psi - \floor{u_\psi})}{(u_{\ch{\psi}} - u_\psi)(\floor{u_\psi} - u_{i^\vartriangle(\mathcal{T}^{\delta, k-1})})} \\
& =  \frac{1-(u_{\ch{\psi}} - u_\psi) - (u_\psi - \floor{u_\psi})}{(u_{\ch{\psi}} - u_\psi)(\floor{u_\psi} - u_{i^\vartriangle(\mathcal{T}^{\delta, k-1})})}\mathbf{D}^\delta_{i,k'} \\
& = \frac{1-u_{\ch{\psi}} + u_\psi - u_\psi + \floor{u_\psi}}{(u_{\ch{\psi}} - u_\psi)(\floor{u_\psi} - u_{i^\vartriangle(\mathcal{T}^{\delta, k-1})})} \mathbf{D}^\delta_{i,k'}\\
& = 0. 
\end{align*}

\item[(ii)] Now we consider the case of \eqref{T2ch}-\eqref{T2}, where $i^\vartriangle(\mathcal{T}^{\delta,k-1}) = \psi \in \Psi$. Following the same reasoning as in (i), $\mathbf{D}^\delta_{\psi,k'} = \mathbf{D}^\delta_{\ch{\psi},k'} =\mathbf{D}^\delta_{i,k'}$. When $i=\ch{\psi}$, 
\[\sum_{j=1}^{|\mathcal{V}|}\mathbf{T}^\delta_{k,j}\mathbf{D}^\delta_{j,k'}  = \frac{1-1}{u_{\ch{\psi}} - u_\psi}\mathbf{D}^\delta_{\psi, k'} = 0.\]
When $i\neq \ch{\psi}$,
\begin{align*}
\sum_{j=1}^{|\mathcal{V}|}\mathbf{T}^\delta_{k,j}\mathbf{D}^\delta_{j,k'} & = \left[\frac{1}{u_i - \floor{u_\psi}} -\frac{1}{(u_{\ch{\psi}}-u_\psi)(u_i - \floor{u_\psi})} + \frac{u_\psi - \floor{u_\psi}}{(u_{\ch{\psi}}-u_\psi)(u_i - \floor{u_\psi})} \right]\mathbf{D}^\delta_{i, k'} \\
& =  \frac{(u_{\ch{\psi}}-u_\psi) - 1 + u_\psi - \floor{u_\psi}}{(u_{\ch{\psi}}-u_\psi)(u_i - \floor{u_\psi})}\mathbf{D}^\delta_{i, k'} \\
& = \frac{u_{\ch{\psi}}-u_\psi + u_\psi - u_{\ch{\psi}}}{(u_{\ch{\psi}}-u_\psi)(u_i - \floor{u_\psi})}\mathbf{D}^\delta_{i, k'} \\
& = 0.
\end{align*}

\item[(iii)] Suppose $i$ and $i^\vartriangle(\mathcal{T}^{\delta, k-1})$ do not belong to $\Psi$ as in \eqref{T3}. Due to the fact that $k' < k$, $\sigma_{\delta(k)}(\mathcal{T}^{\delta, k'})$ either contains both $i$ and $i^\vartriangle(\mathcal{T}^{\delta, k-1})$, or neither, which means that $\mathbf{D}^\delta_{i,k'}=\mathbf{D}^\delta_{i^\vartriangle(\mathcal{T}^{\delta, k-1}),k'}$. Therefore, $\sum_{j=1}^{|\mathcal{V}|}\mathbf{T}^\delta_{k,j}\mathbf{D}^\delta_{j,k'} = [1/(u_i - u_{i^\vartriangle(\mathcal{T}^{\delta, k-1})}) - 1/(u_i - u_{i^\vartriangle(\mathcal{T}^{\delta, k-1})})]\mathbf{D}^\delta_{i,k'} = 0$.
\end{enumerate}
\end{observation}

\begin{observation}
\label{obs:p_inv_off_diag_highkprime}
We now verify that $\sum_{j=1}^{|\mathcal{V}|}\mathbf{T}^\delta_{k,j}\mathbf{D}^\delta_{j,k'} = 0$ for $k < k' \leq |\mathcal{V}|$. Recall that we assume $\delta(k)$ and $\delta(k')$ to belong to the same component of $\mathcal{G}$. 

\begin{enumerate} 
\item[(i)] Suppose $\mathbf{T}^\delta_{k,j}$ assumes the form of \eqref{T1}. That is, $i = \psi\in\Psi$ and $\ch{\psi}\notin\mathcal{T}^{\delta, k-1}$. If $\delta(k') \neq \ch{\psi}$, then $\mathbf{D}^\delta_{q,k'} = 0$ for all $q\in\{\psi, \ch{\psi}, i^\vartriangle(\mathcal{T}^{\delta, k-1})\}$, and $\sum_{j=1}^{|\mathcal{V}|}\mathbf{T}^\delta_{k,j}\mathbf{D}^\delta_{j,k'} = 0$. When $\delta(k') = \ch{\psi}$, we have $\mathbf{D}^\delta_{\ch{\psi},k'} = 1$, $\mathbf{D}^\delta_{\psi,k'} = u_{\psi} - \floor{u_\psi}$ and $\mathbf{D}^\delta_{i^\vartriangle(\mathcal{T}^{\delta, k-1}),k'} = 0$. In turn,
\begin{align*}
\sum_{j=1}^{|\mathcal{V}|}\mathbf{T}^\delta_{k,j}\mathbf{D}^\delta_{j,k'} = \frac{u_{\psi} - \floor{u_\psi}}{(u_{\ch{\psi}} - u_\psi)(\floor{u_\psi} - u_{i^\vartriangle(\mathcal{T}^{\delta, k-1})})} - \frac{u_\psi - \floor{u_\psi}}{(u_{\ch{\psi}} - u_\psi)(\floor{u_\psi} - u_{i^\vartriangle(\mathcal{T}^{\delta, k-1})})}  = 0.
\end{align*}

\item[(ii)] Now suppose $i^\vartriangle(\mathcal{T}^{\delta, k-1}) = \psi\in\Psi$. When $i = \ch{\psi}$, we have $P(\delta, k)_\psi = u_\psi$ and $P(\delta, k)_{\ch{\psi}} = u_{\ch{\psi}}$, which are the highest attainable values in these two entries. Thus, $\mathbf{D}^\delta_{\ch{\psi},k'} = \mathbf{D}^\delta_{\psi,k'} = 0$, and $\sum_{j=1}^{|\mathcal{V}|}\mathbf{T}^\delta_{k,j}\mathbf{D}^\delta_{j,k'} = 0$. \\

Suppose $i \neq \ch{\psi}$. If $\delta(k') \neq \ch{\psi}$, then $\mathbf{D}^\delta_{i,k'} = \mathbf{D}^\delta_{\ch{\psi},k'} = \mathbf{D}^\delta_{\psi,k'} = 0$ and $\sum_{j=1}^{|\mathcal{V}|}\mathbf{T}^\delta_{k,j}\mathbf{D}^\delta_{j,k'} = 0$. If $\delta(k') = \ch{\psi}$, then $\mathbf{D}^\delta_{\ch{\psi},k'} = 1$, $\mathbf{D}^\delta_{\psi,k'} = u_{\psi} - \floor{u_\psi}$ and $\mathbf{D}^\delta_{i,k'} = 0$. We have 
\begin{align*}
\sum_{j=1}^{|\mathcal{V}|}\mathbf{T}^\delta_{k,j}\mathbf{D}^\delta_{j,k'} = -\frac{u_{\psi} - \floor{u_\psi}}{(u_{\ch{\psi}}-u_\psi)(u_i - \floor{u_\psi})} + \frac{u_\psi - \floor{u_\psi}}{(u_{\ch{\psi}}-u_\psi)(u_i - \floor{u_\psi})}  = 0.
\end{align*}

\item[(iii)] If $\mathbf{T}^\delta_{k,j}$ assumes the form of \eqref{T3}, then $i$ and $i^\vartriangle(\mathcal{T}^{\delta, k-1})$ do not belong to $\bigcup_{\psi\in\Psi} \{\psi, \ch{\psi}\}$. We note that $P(\delta, k)_i$ and $P(\delta, k)_{i^\vartriangle(\mathcal{T}^{\delta, k-1})}$ have attained the respective upper bounds $u_i$ and $u_{i^\vartriangle(\mathcal{T}^{\delta, k-1})}$. Thus $\mathbf{D}^\delta_{i,k'} = \mathbf{D}^\delta_{i^\vartriangle(\mathcal{T}^{\delta, k-1}),k'} = 0$. It follows that $\sum_{j=1}^{|\mathcal{V}|}\mathbf{T}^\delta_{k,j}\mathbf{D}^\delta_{j,k'} = 0$.
\end{enumerate}
\end{observation}

\begin{lemma}
\label{lem:TD_inv}
Given a valid $\delta\in\mathfrak{S}(\mathcal{V})$, $\mathbf{T}^\delta\mathbf{D}^\delta  = \mathbf{D}^\delta\mathbf{T}^\delta = \mathbf{I}$. 
\end{lemma}
\begin{proof}
This lemma follows from Observations \ref{obs:p_inv_diag}-\ref{obs:p_inv_off_diag_highkprime}. 
\end{proof}

\begin{lemma}
\label{lem:tree_affine}
Any $\overline{\mathbf{z}}\in\mathbb{R}^{|\mathcal{V}|}$ can be written as an affine combination of $\{P(\delta,k)\}_{k=0}^{|\mathcal{V}|}$ for any valid $\delta\in\mathfrak{S}(\mathcal{V})$, namely
\[\overline{\mathbf{z}} = \sum_{k=0}^{|\mathcal{V}|} \left[t(\delta, \overline{\mathbf{z}})_k-t(\delta, \overline{\mathbf{z}})_{k+1}\right] P(\delta, k).\]
\end{lemma}
\begin{proof}
Lemma \ref{lem:TD_inv} proves that 
\[\overline{\mathbf{z}} = \mathbf{D}^\delta\mathbf{T}^\delta\overline{\mathbf{z}} = \sum_{k=1}^{|\mathcal{V}|} \left[P(\delta, k) - P(\delta, k-1)\right] t(\delta, \overline{\mathbf{z}})_k = \sum_{k=0}^{|\mathcal{V}|} \left[t(\delta, \overline{\mathbf{z}})_k-t(\delta, \overline{\mathbf{z}})_{k+1}\right] P(\delta, k).\]
By construction, $\sum_{k=0}^{|\mathcal{V}|} \left[t(\delta, \overline{\mathbf{z}})_k-t(\delta, \overline{\mathbf{z}})_{k+1}\right] = t(\delta, \overline{\mathbf{z}})_{0} + \sum_{k=1}^{|\mathcal{V}|} \left[-t(\delta, \overline{\mathbf{z}})_k+t(\delta, \overline{\mathbf{z}})_{k}\right] - t(\delta, \overline{\mathbf{z}})_{|\mathcal{V}|+1} = 1$. Thus this combination is affine. 
\end{proof}

In fact, we may find a particular valid permutation $\delta\in\mathfrak{S}(\mathcal{V})$ for any given $\overline{\mathbf{z}}\in \conv{\mathcal{Z}(\mathcal{G},\mathbf{u})}$, such that the aforementioned linear combination is a \emph{convex} combination.

\begin{algorithm}[H]
\label{alg:tree_greedy}
\SetAlgoLined
\textbf{Input} $\mathcal{G}$, $\mathbf{u}$, $\overline{\mathbf{z}}\in\conv{\mathcal{Z}(\mathcal{G},\mathbf{u})}$\;
$\delta^0 \leftarrow ()$\; 
\For{$k=1,2,\dots, |\mathcal{V}|$}{
    ${\Delta} \leftarrow \text{\texttt{Valid\_Candidates}}(\mathcal{G}, \mathbf{u}, \delta^{k-1})$ (Algorithm \ref{alg:tree_valid})\;
    $i^* \leftarrow \arg\max_{i\in{\Delta}} t((\delta^{k-1}(1),\dots, \delta^{k-1}(k-1), \delta^{k-1}(k)=i),\overline{\mathbf{z}})_k$\; \tcp{Break ties arbitrarily in case of multiple maximizers.}
    $\delta^{k} \leftarrow (\delta^{k-1}(1),\dots, \delta^{k-1}(k-1), i^*)$\;
}
\textbf{Output} $\delta^{|\mathcal{V}|}$.
\caption{\texttt{Permutation\_Finder}}
\end{algorithm} 

By Observation \ref{obs:Delta}, $\Delta\neq\emptyset$ in every iteration, and the partial permutation $\delta^k$ for every $k\in\{1,2,\dots, |\mathcal{V}|\}$ is valid. Therefore, the output from Algorithm \ref{alg:tree_greedy} is a valid and full permutation of $\mathcal{V}$.

\begin{lemma}
\label{lem:first_and_one}
If $\delta$ is the output of Algorithm \ref{alg:tree_greedy} given $\mathcal{G},\mathbf{u}$ and $\overline{\mathbf{z}} \in \conv{\mathcal{Z}(\mathcal{G}, \mathbf{u})}$, then $t(\delta,\overline{\mathbf{z}})_1 \leq 1$. 
\end{lemma}
\begin{proof}
By definition, $\delta(1)^\vartriangle(\mathcal{T}^{\delta,0}) = 0$, so $t(\delta,\overline{\mathbf{z}})_1 = \overline{z}_{\delta(1)}/u_{\delta(1)}$, or $t(\delta,\overline{\mathbf{z}})_1 =\eta_{\delta(1)}/\floor{u_{\delta(1)}}$. In the former case, $t(\delta,\overline{\mathbf{z}})_1 \leq 1$ because $\overline{z}_{\delta(1)} \leq u_{\delta(1)}$. In the latter, $\delta(1) \in \Psi$. By validity of $\delta$, $\floor{u_{\delta(1)}} \geq 1$. Since $\overline{\mathbf{z}}$ satisfies the MIR inequality \eqref{eq:mir} involving $\delta(1)$ and $\ch{\delta(1)}$: 
 \[-\overline{z}_{\ch{\delta(1)}} + \frac{\overline{z}_{\delta(1)}}{u_{\delta(1)} - \floor{u_{\delta(1)}}} \leq \frac{\floor{u_{\delta(1)}}(u_{\ch{\delta(1)}} - u_{\delta(1)})}{u_{\delta(1)} - \floor{u_{\delta(1)}}}. \]
Therefore, 
\begin{align*}
t(\overline{\delta},\overline{\mathbf{z}})_1 & = \frac{\overline{z}_{\delta(1)} - (u_{\delta(1)} - \floor{u_{\delta(1)}})\overline{z}_{\ch{\delta(1)}}}{\floor{u_{\delta(1)}}(u_{\ch{\delta(1)}} - u_{\delta(1)})} \\
& = \left[-\overline{z}_{\ch{\delta(1)}} + \frac{\overline{z}_{\delta(1)}}{u_{\delta(1)} - \floor{u_{\delta(1)}}}\right]\frac{u_{\delta(1)} - \floor{u_{\delta(1)}}}{\floor{u_{\delta(1)}}(u_{\ch{\delta(1)}} - u_{\delta(1)})} \\
& \leq \frac{\floor{u_{\delta(1)}}(u_{\ch{\delta(1)}} - u_{\delta(1)})}{u_{\delta(1)} - \floor{u_{\delta(1)}}} \frac{u_{\delta(1)} - \floor{u_{\delta(1)}}}{\floor{u_{\delta(1)}}(u_{\ch{\delta(1)}} - u_{\delta(1)})} \\
& = 1.
\end{align*}
\end{proof}

In the following two observations, we consider permutation $\delta$ returned by Algorithm \ref{alg:tree_greedy} given $\mathcal{G},\mathbf{u}$ and $\overline{\mathbf{z}}\in\conv{\mathcal{Z}(\mathcal{G}, \mathbf{u})}$. We argue that $t(\delta,\overline{\mathbf{z}})_{k+1} \leq t(\delta,\overline{\mathbf{z}})_k$ for any $k\in\{1,\dots, |\mathcal{V}|-1\}$. 

\begin{observation}
\label{obs:monotone_invalid}
Consider any $k\in\{1,\dots, |\mathcal{V}|-1\}$. There are three scenarios under which $\delta(k+1)$ is an invalid candidate in the $k$-th iteration and becomes a valid candidate in the $(k+1)$-th iteration.  

\begin{enumerate}
\item[(i)] The first scenario corresponds to line 4 of Algorithm \ref{alg:tree_valid}. In particular, $\delta(k+1)$ must be the parent of $\delta(k)$ and $u_{\delta(k)} = u_{\delta(k+1)}$. It is implied that $\delta(k),\delta(k+1)\notin\Psi$. Moreover, $\delta(k)^\vartriangle(\mathcal{T}^{\delta, k-1}) = \delta(k+1)^\vartriangle(\mathcal{T}^{\delta, k-1})$, which we denote by $\phi$; $\phi$ could be zero. For ease of notation, let 
\begin{equation}
\label{nu}
\nu_{\phi} = 
\begin{cases}
\eta_{\phi}(\overline{\mathbf{z}}), & \phi\in\Psi, \\
\overline{z}_{\phi}, & \phi\notin\Psi, 
\end{cases}
\end{equation}
and 
\begin{equation}
\label{gamma}
\gamma_{\phi} = 
\begin{cases}
\floor{u_{\phi}}, &  \phi\in\Psi, \\
u_{\phi}, &  \phi\notin\Psi. 
\end{cases}
\end{equation}
Now, $t(\delta,\overline{\mathbf{z}})_k = (\overline{z}_{\delta(k)}-\nu_\phi)/(u_{\delta(k)} - \gamma_\phi)$ and $t(\delta,\overline{\mathbf{z}})_{k+1} = (\overline{z}_{\delta(k+1)}-\nu_\phi)/(u_{\delta(k+1)} - \gamma_\phi)$. From the observation that $u_{\delta(k)} = u_{\delta(k+1)}$, the denominators of $t(\delta,\overline{\mathbf{z}})_k$ and $t(\delta,\overline{\mathbf{z}})_{k+1}$ are the same. Given that $\overline{z}_{\delta(k+1)} \leq \overline{z}_{\delta(k)}$, we have $t(\delta,\overline{\mathbf{z}})_{k+1} \leq t(\delta,\overline{\mathbf{z}})_k$. 

\item[(ii)] The second scenario corresponds to line 8 of Algorithm \ref{alg:tree_valid}, in which $\delta(k+1)\in\Psi$ and $0< u_{\delta(k+1)} < 1$. This tells us that $\delta(k+1)$ is a root vertex. Based on the fact that $\delta(k+1)$ is a valid candidate in the $(k+1)$-th iteration, $\delta(k)$ must be $\ch{\delta(k+1)}$. Thus $u_{\delta(k)} = 1$. In this case, $t(\delta,\overline{\mathbf{z}})_k = (\overline{z}_{\delta(k)}-0)/(u_{\delta(k)}-0) = \overline{z}_{\delta(k)}$ and similarly $t(\delta,\overline{\mathbf{z}})_{k+1} = \overline{z}_{\delta(k+1)}/u_{\delta(k+1)}$. We know that $\overline{\mathbf{z}}$ satisfies the MIR inequality $-z_{\delta(k)} + z_{\delta(k+1)}/(u_{\delta(k+1)}-\floor{u_{\delta(k+1)}}) = -z_{\delta(k)} + z_{\delta(k+1)}/u_{\delta(k+1)} \leq 0$. Thus $\overline{z}_{\delta(k+1)}/u_{\delta(k+1)} \leq \overline{z}_{\delta(k)}$; that is, $t(\delta,\overline{\mathbf{z}})_{k+1} \leq t(\delta,\overline{\mathbf{z}})_k$.

\item[(iii)] The last scenario corresponds to line 11 of Algorithm \ref{alg:tree_valid}. Here, $\delta(k+1)\in\Psi$ and $\delta(k+1)^\vartriangle(\mathcal{T}^{\delta,k-1})$ must be $\pa{\delta(k+1)}$ with $u_{\delta(k+1)^\vartriangle(\mathcal{T}^{\delta,k-1})} = \floor{u_{\delta(k+1)}}$. We note that $\delta(k+1)^\vartriangle(\mathcal{T}^{\delta,k-1}) = \pa{\delta(k+1)}$ because otherwise there exists a descendant of $\delta(k+1)^\vartriangle(\mathcal{T}^{\delta,k-1})$ with the same upper bound but succeeds $\delta(k+1)^\vartriangle(\mathcal{T}^{\delta,k-1})$ in the order of $\delta$---this cannot happen when $\delta$ is valid. Vertex $\delta(k+1)$ may only become a valid candidate in the $(k+1)$-th iteration if $\delta(k) = \ch{\delta(k+1)}$. We infer that $\delta(k)^\vartriangle(\mathcal{T}^{\delta,k-1}) = \pa{\delta(k+1)}$ as well. Thus \[t(\delta,\overline{\mathbf{z}})_k = \frac{\overline{z}_{\delta(k)} - \overline{z}_{\pa{\delta(k+1)}}}{\floor{u_{\delta(k+1)}} + 1 - \floor{u_{\delta(k+1)}}} = \overline{z}_{\delta(k)} - \overline{z}_{\pa{\delta(k+1)}},\] and 
\[t(\delta,\overline{\mathbf{z}})_{k+1} = \frac{\overline{z}_{\delta(k+1)} - \overline{z}_{\pa{\delta(k+1)}}}{u_{\delta(k+1)} - \floor{u_{\delta(k+1)}}}.\] 
Let $k' = \delta^{-1}(\pa{\delta(k+1)})$. In the $k'$-th iteration, $\delta(k+1)$ is a valid candidate because $\floor{\delta(k+1)}\neq 0$ and at this point no ascendant of $\delta(k+1)$ with upper bound $\floor{\delta(k+1)}$ precedes it in the partial permutation. By line 5 of Algorithm \ref{alg:tree_greedy}, we have 
\[\frac{\overline{z}_{\pa{\delta(k+1)}} - \overline{z}_{\pa{\delta(k+1)}^\vartriangle(\mathcal{T}^{\delta, k'-1})}}{u_{\pa{\delta(k+1)}} - u_{\pa{\delta(k+1)}^\vartriangle(\mathcal{T}^{\delta, k'-1})}} \geq \frac{\eta_{\delta(k+1)}({\mathbf{\overline{z}}}) - \overline{z}_{\pa{\delta(k+1)}^\vartriangle(\mathcal{T}^{\delta, k'-1})}}{\floor{u_{\delta(k+1)}} - u_{\pa{\delta(k+1)}^\vartriangle(\mathcal{T}^{\delta, k'-1})}}.\]
Recall $u_{\pa{\delta(k+1)}} = \floor{u_{\delta(k+1)}}$. Thus $\overline{z}_{\pa{\delta(k+1)}} \geq \eta_{\delta(k+1)}({\mathbf{\overline{z}}})$. Equivalently, 
\begin{align*}
& \quad \quad  \overline{z}_{\pa{\delta(k+1)}} \geq \frac{\overline{z}_{\delta(k+1)} - (u_{\delta(k+1)} - \floor{u_{\delta(k+1)}}) \overline{z}_{\ch{{\delta(k+1)}}}}{u_{\ch{{\delta(k+1)}}} - u_{\delta(k+1)}} \\
& \Rightarrow \overline{z}_{\pa{\delta(k+1)}} \geq \frac{\overline{z}_{\delta(k+1)} - (u_{\delta(k+1)} - \floor{u_{\delta(k+1)}}) \overline{z}_{\delta(k)}}{\ceil{u_{\delta(k+1)}} - u_{\delta(k+1)}} \\
& \Rightarrow (\ceil{u_{\delta(k+1)}} - u_{\delta(k+1)}) \overline{z}_{\pa{\delta(k+1)}} \geq \overline{z}_{\delta(k+1)} - (u_{\delta(k+1)} - \floor{u_{\delta(k+1)}}) \overline{z}_{\delta(k)} \\
& \Rightarrow (u_{\delta(k+1)} - \floor{u_{\delta(k+1)}}) \overline{z}_{\delta(k)} + (\ceil{u_{\delta(k+1)}} - u_{\delta(k+1)} - 1) \overline{z}_{\pa{\delta(k+1)}} \geq \overline{z}_{\delta(k+1)} - \overline{z}_{\pa{\delta(k+1)}} \\
& \Rightarrow (u_{\delta(k+1)} - \floor{u_{\delta(k+1)}}) \overline{z}_{\delta(k)} - (u_{\delta(k+1)} - \floor{u_{\delta(k+1)}}) \overline{z}_{\pa{\delta(k+1)}} \geq \overline{z}_{\delta(k+1)} - \overline{z}_{\pa{\delta(k+1)}} \\
& \Rightarrow (u_{\delta(k+1)} - \floor{u_{\delta(k+1)}}) (\overline{z}_{\delta(k)} - \overline{z}_{\pa{\delta(k+1)}}) \geq \overline{z}_{\delta(k+1)} - \overline{z}_{\pa{\delta(k+1)}} \\
& \Rightarrow \overline{z}_{\delta(k)} - \overline{z}_{\pa{\delta(k+1)}} \geq \frac{\overline{z}_{\delta(k+1)} - \overline{z}_{\pa{\delta(k+1)}}}{u_{\delta(k+1)} - \floor{u_{\delta(k+1)}}}\\
& \Rightarrow t(\delta,\overline{\mathbf{z}})_k \geq t(\delta,\overline{\mathbf{z}})_{k+1}.
\end{align*}
\end{enumerate}
In summary, $t(\delta,\overline{\mathbf{z}})_{k} \geq t(\delta,\overline{\mathbf{z}})_{k+1}$ when $\delta(k+1)$ is not a valid candidate in the $k$-th iteration. 
\end{observation}

\begin{observation}
\label{obs:monotone_valid}
If $\delta(k+1)$ is a valid candidate in the $k$-th iteration,  $k\in\{1,\dots, |\mathcal{V}|-1\}$, then 
\begin{equation}
t(({\delta}(1), \dots, {\delta}(k-1), {\delta}(k)),\overline{\mathbf{z}})_k \geq t(({\delta}(1), \dots, {\delta}(k-1), {\delta}(k+1)),\overline{\mathbf{z}})_k.
\end{equation}
If $\delta(k)$ is not involved in the evaluation of $t(\delta,\overline{\mathbf{z}})_{k+1}$, then 
\[t(\delta,\overline{\mathbf{z}})_k \geq t(({\delta}(1), \dots, {\delta}(k-1), {\delta}(k+1)),\overline{\mathbf{z}})_k =t(\delta,\overline{\mathbf{z}})_{k+1}.\]

There are two scenarios in which $\delta(k)$ is involved in the evaluation of $t(\delta,\overline{\mathbf{z}})_{k+1}$, and we explore them next.
\begin{enumerate}
\item[(i)] In the first scenario, $\delta(k+1)\in R^+(\delta(k))$, and $\delta(k) = \delta(k+1)^\vartriangle(\mathcal{T}^{\delta, k})$. Let $\phi = \delta(k)^\vartriangle(\mathcal{T}^{\delta, k-1})$. We note that $\delta(k+1)^\vartriangle(\mathcal{T}^{\delta, k-1}) = \phi$. By abusing notation, we use $\nu$ and $\gamma$ defined in \eqref{nu}-\eqref{gamma} to avoid verbose discussion by case. With this set of notation, 
\[t(\delta,\overline{\mathbf{z}})_k = \frac{\nu_{\delta(k)} - \nu_{\phi}}{\gamma_{\delta(k)} - \gamma_{\phi}}, \quad  t(\delta,\overline{\mathbf{z}})_{k+1} = \frac{\nu_{\delta(k+1)} - \nu_{\delta(k)}}{\gamma_{\delta(k+1)} - \gamma_{\delta(k)}},\]
and  
\[t(({\delta}(1), \dots, {\delta}(k-1), {\delta}(k+1)),\overline{\mathbf{z}})_k = \frac{\nu_{\delta(k+1)} - \nu_{\phi}}{\gamma_{\delta(k+1)} - \gamma_{\phi}}.\] 
Given $t(\delta,\overline{\mathbf{z}})_k \geq t(({\delta}(1), \dots, {\delta}(k-1), {\delta}(k+1)),\overline{\mathbf{z}})_k$, we observe that 
\begingroup
\allowdisplaybreaks
\begin{align*}
& \hspace{14pt} \frac{\nu_{\delta(k)} - \nu_{\phi}}{\gamma_{\delta(k)} - \gamma_{\phi}} \geq \frac{\nu_{\delta(k+1)} - \nu_{\phi}}{\gamma_{\delta(k+1)} - \gamma_{\phi}} \\
& \Rightarrow \frac{(\gamma_{\delta(k+1)} - \gamma_{\phi})(\nu_{\delta(k)} - \nu_{\phi})}{\gamma_{\delta(k)} - \gamma_{\phi}} \geq \nu_{\delta(k+1)} - \nu_{\phi} \\
& \hspace{14pt} \text{($\gamma_{\delta(k+1)} > \gamma_{\phi}$ because the partial permutation $({\delta}(1), \dots, {\delta}(k-1), {\delta}(k+1))$ is valid)}\\
& \Rightarrow \frac{(\gamma_{\delta(k+1)} - \gamma_{\phi})(\nu_{\delta(k)} - \nu_{\phi})}{\gamma_{\delta(k)} - \gamma_{\phi}} + \nu_{\phi} - \nu_{\delta(k)} \geq \nu_{\delta(k+1)} - \nu_{\delta(k)} \\
& \Rightarrow \frac{(\gamma_{\delta(k+1)} - \gamma_{\phi})- (\gamma_{\delta(k)} - \gamma_{\phi})}{\gamma_{\delta(k)} - \gamma_{\phi}} (\nu_{\delta(k)} - \nu_{\phi}) \geq \nu_{\delta(k+1)} - \nu_{\delta(k)} \\
& \Rightarrow \frac{\gamma_{\delta(k+1)} - \gamma_{\delta(k)}}{\gamma_{\delta(k)} - \gamma_{\phi}} (\nu_{\delta(k)} - \nu_{\phi}) \geq \nu_{\delta(k+1)} - \nu_{\delta(k)} \\
& \Rightarrow \frac{\nu_{\delta(k)} - \nu_{\phi}}{\gamma_{\delta(k)} - \gamma_{\phi}} \geq \frac{\nu_{\delta(k+1)} - \nu_{\delta(k)}}{\gamma_{\delta(k+1)} - \gamma_{\delta(k)}} \quad \text{($\gamma_{\delta(k+1)} > \gamma_{\delta(k)}$ for validity of $\delta$)}\\
& \Rightarrow t(\delta,\overline{\mathbf{z}})_k \geq t(\delta,\overline{\mathbf{z}})_{k+1}.
\end{align*}
\endgroup
 
\item[(ii)] Suppose $\delta(k+1)\in R^-(\delta(k))$. Any term involving $\delta(k)$ is only present in $t(\delta,\overline{\mathbf{z}})_{k+1}$ when $\delta(k+1)\in\Psi$ and $\delta(k) = \ch{\delta(k+1)}$. These conditions imply that $\delta(k+1)^\vartriangle(\mathcal{T}^{\delta, k-1}) = \delta(k)^\vartriangle(\mathcal{T}^{\delta, k-1})$, which we denote by $\phi$. Due to validity of $\delta$, $u_\phi < \floor{u_{\delta(k+1)}}$. For ease of notation, we let $\nu_\phi$ and $\gamma_\phi$ be defined as in \eqref{nu} and \eqref{gamma} respectively. We observe that
\begingroup
\allowdisplaybreaks
\begin{align*}
& \hspace{14pt} t(\delta,\overline{\mathbf{z}})_k \geq t(({\delta}(1), \dots, {\delta}(k-1), {\delta}(k+1)),\overline{\mathbf{z}})_k \\
& \Rightarrow \frac{\overline{z}_{\delta(k)} - \nu_{\phi}}{u_{\delta(k)} - \gamma_{\phi}} \geq \frac{\eta_{\delta(k+1)}(\overline{\mathbf{z}}) - \nu_{\phi}}{\floor{u_{\delta(k+1)}} - \gamma_{\phi}} \\
& \Rightarrow \frac{(\floor{u_{\delta(k+1)}} - \gamma_{\phi})(\overline{z}_{\delta(k)} - \nu_{\phi})}{u_{\delta(k)} - \gamma_{\phi}} \geq \eta_{\delta(k+1)}(\overline{\mathbf{z}}) - \nu_{\phi} \\
&  \hspace{14pt} \text{($\floor{u_{\delta(k+1)}} > \gamma_{\phi}$ because the partial permutation $({\delta}(1), \dots, {\delta}(k-1), {\delta}(k+1))$ is valid)} \\
& \Rightarrow \frac{(\floor{u_{\delta(k+1)}} - \gamma_{\phi})(\overline{z}_{\delta(k)} - \nu_{\phi})}{u_{\delta(k)} - \gamma_{\phi}} \geq \frac{\overline{z}_{\delta(k+1)}-(u_{\delta(k+1)} - \floor{u_{\delta(k+1)}})\overline{z}_{\delta(k)}}{u_{\delta(k)} - u_{\delta(k+1)}} - \nu_{\phi} \\
& \Rightarrow  \frac{(u_{\delta(k)} - \gamma_{\phi}-1)(\overline{z}_{\delta(k)} - \nu_{\phi})}{u_{\delta(k)} - \gamma_{\phi}} \geq \frac{\overline{z}_{\delta(k+1)}-(u_{\delta(k+1)} - \floor{u_{\delta(k+1)}})\overline{z}_{\delta(k)}}{u_{\delta(k)} - u_{\delta(k+1)}} - \nu_{\phi} \\
&  \hspace{14pt} \text{(recall $\floor{u_{\delta(k+1)}} = u_{\delta(k)}-1$)} \\
& \Rightarrow \overline{z}_{\delta(k)} - \nu_{\phi} - \frac{\overline{z}_{\delta(k)} - \nu_{\phi}}{u_{\delta(k)} - \gamma_{\phi}} \geq \frac{\overline{z}_{\delta(k+1)}-(u_{\delta(k+1)} - \floor{u_{\delta(k+1)}})\overline{z}_{\delta(k)}}{u_{\delta(k)} - u_{\delta(k+1)}} - \nu_{\phi} \\
& \Rightarrow (u_{\delta(k)} - u_{\delta(k+1)})\overline{z}_{\delta(k)}  - \frac{(u_{\delta(k)} - u_{\delta(k+1)})(\overline{z}_{\delta(k)} - \nu_{\phi})}{u_{\delta(k)} - \gamma_{\phi}} \geq \overline{z}_{\delta(k+1)}-(u_{\delta(k+1)} - \floor{u_{\delta(k+1)}})\overline{z}_{\delta(k)} \\
& \Rightarrow (u_{\delta(k+1)} - \floor{u_{\delta(k+1)}} + u_{\delta(k)} - u_{\delta(k+1)})\overline{z}_{\delta(k)}  - \frac{(u_{\delta(k)} - u_{\delta(k+1)})(\overline{z}_{\delta(k)} - \nu_{\phi})}{u_{\delta(k)} - \gamma_{\phi}} \geq \overline{z}_{\delta(k+1)} \\
& \Rightarrow \overline{z}_{\delta(k)} - \nu_{\phi} - \frac{(u_{\delta(k)} - u_{\delta(k+1)})(\overline{z}_{\delta(k)} - \nu_{\phi})}{u_{\delta(k)} - \gamma_{\phi}} \geq \overline{z}_{\delta(k+1)} - \nu_{\phi}\\
& \Rightarrow \frac{(u_{\delta(k)} - \gamma_{\phi})-(u_{\delta(k)} - u_{\delta(k+1)})}{u_{\delta(k)} - \gamma_{\phi}}(\overline{z}_{\delta(k)} - \nu_{\phi}) \geq \overline{z}_{\delta(k+1)} - \nu_{\phi}\\
& \Rightarrow \frac{\overline{z}_{\delta(k)} - \nu_{\phi}}{u_{\delta(k)} - \gamma_{\phi}} \geq \frac{\overline{z}_{\delta(k+1)} - \nu_{\phi}}{u_{\delta(k+1)} -  \gamma_{\phi}}\\
& \Rightarrow t(\delta,\overline{\mathbf{z}})_k \geq t(\delta,\overline{\mathbf{z}})_{k+1}.
\end{align*}
\endgroup
\end{enumerate}
Hence, $t(\delta,\overline{\mathbf{z}})_{k} \geq t(\delta,\overline{\mathbf{z}})_{k+1}$ when $\delta(k+1)$ is a valid candidate in the $k$-th iteration, for any $k\in\{1,\dots,|\mathcal{V}|-1\}$. 
\end{observation}

\begin{lemma}
\label{lem:monotone_t}
Let $\delta$ be the permutation returned by Algorithm \ref{alg:tree_greedy} given $\mathcal{G},\mathbf{u}$ and $\overline{\mathbf{z}}\in\conv{\mathcal{Z}(\mathcal{G}, \mathbf{u})}$. For all $k\in\{1,\dots,|\mathcal{V}|-1\}$, $t(\delta,\overline{\mathbf{z}})_k \geq t(\delta,\overline{\mathbf{z}})_{k+1}$. 
\end{lemma}
\begin{proof}
This lemma follows from Observations \ref{obs:monotone_invalid} and \ref{obs:monotone_valid}.
\end{proof}

\begin{lemma}
\label{lem:last_nonneg}
If $\delta$ is the output from Algorithm \ref{alg:tree_greedy} for $\mathcal{G},\mathbf{u}$ and $\overline{\mathbf{z}}\in\conv{\mathcal{Z}(\mathcal{G}, \mathbf{u})}$, then $t(\delta,\overline{\mathbf{z}})_{|\mathcal{V}|} \geq 0$. 
\end{lemma}
\begin{proof}
We first observe that $t(\delta,\overline{\mathbf{z}})_{|\mathcal{V}|} = \frac{\overline{z}_i - \overline{z}_{\pa{i}}}{u_i - u_{\pa{i}}}$. Even if $\delta(|\mathcal{V}|) = \ch{\psi}$ for some $\psi\in\Psi$, Observation \ref{obs:ch_t} justifies this claim. The denominator $u_i - u_{\pa{i}} > 0$ due to Observation \ref{obs:denom}. The numerator $\overline{z}_i - \overline{z}_{\pa{i}} \geq 0$ because $\mathbf{z}$ satisfies the monotonicity constraints. Therefore, $t(\delta,\overline{\mathbf{z}})_{|\mathcal{V}|} \geq 0$. 
\end{proof}

\begin{proposition}
\label{prop:tree_conv_z}
Let $\delta$ be the output from Algorithm \ref{alg:tree_greedy} given $\overline{\mathbf{z}}\in\conv{\mathcal{Z}(\mathcal{G},\mathbf{u})}$. Then \[\overline{\mathbf{z}} = \sum_{k=0}^{|\mathcal{V}|} \left[t(\delta, \overline{\mathbf{z}})_k-t(\delta, \overline{\mathbf{z}})_{k+1}\right] P(\delta, k),\]
where $\sum_{k=0}^{|\mathcal{V}|} \left[t(\delta, \overline{\mathbf{z}})_k-t(\delta, \overline{\mathbf{z}})_{k+1}\right] = 1$ and $t(\delta, \overline{\mathbf{z}})_k-t(\delta, \overline{\mathbf{z}})_{k+1} \geq 0$ for every $k\in\{0,\dots,|\mathcal{V}|\}$. 
\end{proposition}
\begin{proof}
Lemma \ref{lem:tree_affine} shows that $\sum_{k=0}^{|\mathcal{V}|} \left[t(\delta, \overline{\mathbf{z}})_k-t(\delta, \overline{\mathbf{z}})_{k+1}\right] P(\delta, k)$ is an affine combination of $\{P(\delta, k)\}_{k=1}^{|\mathcal{V}|}$ that equates $\overline{\mathbf{z}}$. For $k\in\{0,\dots,|\mathcal{V}|\}$, $t(\delta, \overline{\mathbf{z}})_k - t(\delta, \overline{\mathbf{z}})_{k+1} \geq 0$ by Lemmas \ref{lem:first_and_one}, \ref{lem:monotone_t}, and \ref{lem:last_nonneg}, which completes the proof.
\end{proof}

\begin{remark}
Algorithm \ref{alg:tree_greedy} completes after $|\mathcal{V}|$ iterations. The steps in each iteration are dominated by sorting at most $|\mathcal{V}|$ values. Therefore, this is an $\mathcal{O}(|\mathcal{V}|^2\log|\mathcal{V}|)$ algorithm. 
\end{remark}

In the next section, we propose a class of inequalities using $\{P(\delta, k)\}_{k=0}^{|\mathcal{V}|}$ and $\mathbf{t}(\delta,\mathbf{z})$. We argue that these inequalities are valid for the epigraph $\mathcal{P}_f^{\mathcal{Z}(\mathcal{G},\mathbf{u})}$.

\section{Valid Inequalities for $\mathcal{P}_f^{\mathcal{Z}(\mathcal{G},\mathbf{u})}$}
\label{sect:valid}
Let $\delta\in\mathfrak{S}(\mathcal{V})$ be a valid permutation. We define a \emph{DR inequality} associated with $\delta$ by
\begin{equation}
\label{eq:DR}
w \geq \sum_{k=1}^{|\mathcal{V}|} t(\delta,\mathbf{z})_k[f(P(\delta,k)) - f(P(\delta,k-1))],
\end{equation}
where $P(\cdot)\in\mathbb{R}^{|\mathcal{V}|}$ is given by \eqref{ext} and function $\mathbf{t}(\cdot,\mathbf{z})$ is defined in \eqref{t_1}-\eqref{t_3}. DR inequalities are linear and homogeneous. 

\begin{example}
Again, consider the directed rooted forest $\mathcal{G}$ in Figure \ref{fig:valid_perm} and the valid permutation $\delta = (6,4,7,5,2,\\3,9,1,11,8,10,12)$. Recall $\{P(\delta, k)\}_{k=0}^{|\mathcal{V}|}$ from Example \ref{eg:delta_extreme_points} and $\mathbf{t}(\delta, \mathbf{z})$ from Example \ref{eg:t_eg}. The DR inequality associated with $\delta$ is 
\begin{align*}
w & \geq f(P(\delta, 1)) z_6/9 + [f(P(\delta, 2)) - f(P(\delta,1))]z_4/8 \\
& \quad\quad + [f(P(\delta, 3)) - f(P(\delta,2))]z_7/12 + [f(P(\delta, 4)) - f(P(\delta,3))]z_5/11.75 \\
& \quad \quad + [f(P(\delta, 5)) - f(P(\delta,4))]z_2 + [f(P(\delta, 6)) - f(P(\delta,5))](z_3-z_2)/7 \\
& \quad\quad + [f(P(\delta, 7)) - f(P(\delta,6))](z_9 - 0.5z_{10})/5 +  10[f(P(\delta, 8)) - f(P(\delta,7))]z_1 \\
& \quad \quad + [f(P(\delta, 9)) - f(P(\delta,8))](-2z_9+z_{10}+z_{11}) + [f(P(\delta, 10)) - f(P(\delta,9))]z_8/10 \\
& \quad\quad  + 2[f(P(\delta, 11)) - f(P(\delta,10))](z_{10} - z_9) + [f(P(\delta, 12)) - f(P(\delta,11))](z_{12} - z_7)/7.9.
\end{align*}
When $\mathbf{z} = P(\delta,9)$, for instance, the right-hand side of this inequality becomes 
\begin{align*}
& \quad f(P(\delta, 1)) + [f(P(\delta, 2)) - f(P(\delta,1))] \\
& \quad\quad + [f(P(\delta, 3)) - f(P(\delta,2))] + [f(P(\delta, 4)) - f(P(\delta,3))] \\
& \quad \quad + [f(P(\delta, 5)) - f(P(\delta,4))] + [f(P(\delta, 6)) - f(P(\delta,5))] \\
& \quad\quad + [f(P(\delta, 7)) - f(P(\delta,6))] +  [f(P(\delta, 8)) - f(P(\delta,7))] \\
& \quad \quad + [f(P(\delta, 9)) - f(P(\delta,8))] + [f(P(\delta, 10)) - f(P(\delta,9))]\cdot 0 \\
& \quad\quad  + [f(P(\delta, 11)) - f(P(\delta,10))]\cdot 0 + [f(P(\delta, 12)) - f(P(\delta,11))]\cdot 0 \\ 
& = f(P(\delta, 9)).
\end{align*} \eged
\end{example}

{
\begin{remark}
The DR inequalities have simplified forms in special cases of $\mathcal{Z}(\mathcal{G},\mathbf{u})$. For example, consider cases (a)--(d) described in Remark \ref{remark:trivialA1}. Given any instance of $\mathcal{Z}(\mathcal{G},\mathbf{u})$ under case (a), a DR inequality associated with $\delta\in\mathfrak{S}(\mathcal{V})$ assumes the following form: 
\begin{equation}
\label{eq:special_a_DR}
w \geq \sum_{k=1}^{|\mathcal{V}|} [f(P(\delta, k)) - f(P(\delta, k-1))]z_{\delta(k)}/u_{\delta(k)}. 
\end{equation}
In cases (b), (c), and (d), a DR inequality associated with a valid permutation $\delta$ is 
\begin{equation}
\label{eq:special_bcd_DR}
w \geq \sum_{k=1}^{|\mathcal{V}|} [f(P(\delta, k)) - f(P(\delta, k-1))](z_{\delta(k)}- z_{{\delta(k)}^{\vartriangle}(\mathcal{T}^{\delta, k-1})})/(u_{\delta(k)} - u_{{\delta(k)}^\vartriangle(\mathcal{T}^{\delta, k-1})}).
\end{equation}
\end{remark}
}

With Lemma \ref{lem:tree_partial_valid} and Proposition \ref{prop:tree_valid_at_ext} below, we establish that a DR inequality is valid for $\mathcal{P}_f^{\mathcal{Z}(\mathcal{G},\mathbf{u})}$ if it is valid at $[P(\mathcal{S}), f(P(\mathcal{S}))]$ for all $\mathcal{S}\subseteq\mathcal{V}$. Recall that $\{P(\mathcal{S})\}_{\mathcal{S}\subseteq \mathcal{V}}$ are the extreme points of $\conv{\mathcal{Z}(\mathcal{G},\mathbf{u})}$.  

\begin{lemma}
\label{lem:tree_partial_valid}
For any $\overline{\mathbf{z}}\in \conv{\mathcal{Z}(\mathcal{G},\mathbf{u})}$, let $\tau\in\mathfrak{S}(\mathcal{V})$ be the corresponding output from Algorithm \ref{alg:tree_greedy}. Then the following inequality holds: 
\[f(\overline{\mathbf{z}}) \geq \sum_{k=1}^{|\mathcal{V}|} t(\tau,\overline{\mathbf{z}})_k[f(P(\tau,k)) - f(P(\tau,k-1))]. \]
\end{lemma}

We note that the property stated in Lemma \ref{lem:tree_partial_valid} is not equivalent to function $f$ being concave---the extreme points $\{P(\tau,k)\}_{k=0}^{|\mathcal{V}|}$ are associated with the carefully chosen permutation $\tau$. The same inequality does not necessarily hold for other possible convex combinations that represent $\overline{\mathbf{z}}\in\mathcal{Z}(\mathcal{G},\mathbf{u})$. 

\begin{proposition}
\label{prop:tree_valid_at_ext}
Let $\delta\in\mathfrak{S}(\mathcal{V})$ be any valid permutation. Suppose that the DR inequality associated with $\delta$ is valid at $[\hat{\mathbf{z}}, f(\hat{\mathbf{z}})]$ for every extreme point $\hat{\mathbf{z}}$ of $\conv{\mathcal{Z}(\mathcal{G},\mathbf{u})}$. Then this DR inequality is valid for $\mathcal{P}_f^{\mathcal{Z}(\mathcal{G},\mathbf{u})}$. 
\end{proposition}
\begin{proof} 
For any $\hat{\mathbf{z}} = P(\mathcal{S})$, $\mathcal{S}\subseteq\mathcal{V}$, we have
\begin{equation}
\label{eq:illustrate_valid_ext}
 f(\hat{\mathbf{z}}) \geq \sum_{k=1}^{|\mathcal{V}|}[f(P(\delta,k)) - f(P(\delta,k-1))] t(\delta,\hat{\mathbf{z}})_k  = \mathbf{d}^\delta\mathbf{T}^\delta \hat{\mathbf{z}},
 \end{equation}
where $\mathbf{d}^\delta = [f(P(\delta,k)) - f(P(\delta,k))]_{k=1}^{|\mathcal{V}|} \in\mathbb{R}^{1\times |\mathcal{V}|}$ and $\mathbf{T}^\delta$ is given by \eqref{T1}-\eqref{T3}. Now consider any $(\mathbf{z}, w)\in\mathcal{P}_f^{\mathcal{Z}(\mathcal{G},\mathbf{u})}$. We know that ${\mathbf{z}}\in \mathcal{Z}(\mathcal{G},\mathbf{u})$ and $w \geq f({\mathbf{z}})$. Let $\tau\in\mathfrak{S}(\mathcal{V})$ be the output from Algorithm \ref{alg:tree_greedy} corresponding to $\mathcal{G},\mathbf{u}$ and $\mathbf{z}$. We observe that 
\begingroup
\allowdisplaybreaks
\begin{align*}
w\geq f({\mathbf{z}}) & \geq \sum_{k=1}^{|\mathcal{V}|}  t(\tau,{\mathbf{z}})_k[f(P(\tau,k)) - f(P(\tau,k-1))] \quad \text{(by Lemma \ref{lem:tree_partial_valid})} \\
& = \sum_{k=0}^{|\mathcal{V}|} [t(\tau,{\mathbf{z}})_k - t(\tau,{\mathbf{z}})_{k+1}]f(P(\tau,k)) \\
& \geq \sum_{k=0}^{|\mathcal{V}|} [t(\tau,{\mathbf{z}})_k - t(\tau,{\mathbf{z}})_{k+1}] \mathbf{d}^\delta\mathbf{T}^\delta P(\tau,k) \\
& \hspace{14pt}  \text{(by \eqref{eq:illustrate_valid_ext} and the fact that $t(\tau,{\mathbf{z}})_k \geq t(\tau,{\mathbf{z}})_{k+1}$ from Proposition \ref{prop:tree_conv_z})}\\
& = \mathbf{d}^\delta\mathbf{T}^\delta \left[\sum_{k=0}^{|\mathcal{V}|} [t(\tau,{\mathbf{z}})_k - t(\tau,{\mathbf{z}})_{k+1}] P(\tau,k) \right] \\
& = \mathbf{d}^\delta\mathbf{T}^\delta {\mathbf{z}} \quad \text{(due to Proposition \ref{prop:tree_conv_z})}\\
& = \sum_{k=1}^{|\mathcal{V}|} t(\delta,{\mathbf{z}})_k[f(P(\delta,k)) - f(P(\delta,k-1))].
\end{align*}
\endgroup
We conclude that the DR inequalities are valid for $\mathcal{P}_f^{\mathcal{Z}(\mathcal{G},\mathbf{u})}$ if they are valid at $[P(\mathcal{S}), f(P(\mathcal{S}))]$ for all $\mathcal{S}\subseteq\mathcal{V}$.
\end{proof}

Next, we check validity of the DR inequalities with respect to $[P(\mathcal{S}), f(P(\mathcal{S}))]$ for every extreme point $P(\mathcal{S})$ of $\conv{\mathcal{Z}(\mathcal{G},\mathbf{u})}$. We do so by analyzing the following cases:
\begin{enumerate}
\item[(\hypertarget{1}{1})] $\Psi = \emptyset$. In this case, there is no fractionally upper-bounded continuous variable with any discrete descendant.  
\item[(\hypertarget{2}{2})] $\Psi \neq \emptyset$, and for every $\psi\in\Psi$, $0< u_\psi < 1$. In this case, every fractionally upper bounded continuous variable, with at least one discrete descendant, assumes an upper bound strictly between 0 and 1. 
\item[(\hypertarget{3}{3})] $\Psi \neq \emptyset$, and there is no restriction on $u_\psi$ for $\psi\in\Psi$. 
\end{enumerate} 
We jointly consider cases (\hyperlink{1}{1}) and (\hyperlink{2}{2}) in Section \ref{sect:valid_P}, and we attack case (\hyperlink{3}{3}) in Section \ref{sect:valid_N}.

\subsection{Validity for Cases (1) and (2)}
\label{sect:valid_P}
We consider any valid $\delta\in\mathfrak{S}(\mathcal{V})$ and any $\mathbf{z} = P(\mathcal{S})$, $\mathcal{S}\subseteq\mathcal{V}$. We first note a desirable property of $\mathbf{t}(\delta,\mathbf{z})$ in cases (\hyperlink{1}{1}) and (\hyperlink{2}{2}). 

\begin{observation}
\label{obs:pos_t}
In cases (\hyperlink{1}{1}) and (\hyperlink{2}{2}), $t(\delta,\mathbf{z})_k$ for every $k\in\{1,\dots, |\mathcal{V}|\}$ falls under \eqref{t_3}. To see this, $\Psi = \emptyset$ in case (\hyperlink{1}{1}), so \eqref{t_1}-\eqref{t_2} are not applicable. In case (\hyperlink{2}{2}), $\floor{u_\psi} = 0$ for every $\psi\in\Psi$, so $\delta^{-1}(\psi) > \delta^{-1}(\ch{\psi})$ because $\delta$ is valid. Again, only \eqref{t_3} applies. Given that $\mathbf{z}$ satisfies the monotonicity constraints, $\mathbf{t}(\delta,\mathbf{z})\geq \mathbf{0}$.
\end{observation}

For every $j\in\{0,\dots,|\mathcal{V}|\}$, we define
\begin{equation}
\label{eq:tildez}
\widetilde{\mathbf{z}}^j := \sum_{k=1}^{j} t(\delta, \mathbf{z})_k [P(\delta,k) - P(\delta,k-1)].   
\end{equation}

\begin{lemma}
\label{lem:tildez_and_P}
For every $j\in\{0,\dots,|\mathcal{V}|\}$, 
\[
\widetilde{\mathbf{z}}^j = \sum_{k=1}^{j} t(\delta, \mathbf{z})_k [P(\delta,k) - P(\delta,k-1)] \leq P(\delta, j). 
\]
\end{lemma}

\begin{observation}
\label{obs:P_tpp}
We note an implication from the proof of Lemma \ref{lem:tildez_and_P} which will be helpful in Section \ref{sect:valid_N}. For $j\in\{0,\dots,|\mathcal{V}|\}$ and $i\in\mathcal{V}$,  
\[z_{i^\vartriangle(\mathcal{T}^{\delta, j})} - z_{i^\vartriangle(\mathcal{T}^{\delta, j-1})}=  t(\delta, \mathbf{z})_j [P(\delta,j)_i - P(\delta,j-1)_i].\]
\end{observation}

\begin{lemma}
\label{lem:tree_t_small}
For any $j\in\{1,\dots,|\mathcal{V}|\}$, when $0\leq t(\delta, \mathbf{z})_j \leq 1$, 
\[ f\left(\widetilde{\mathbf{z}}^j\right) - f\left(\widetilde{\mathbf{z}}^{j-1}\right) \geq  t(\delta, \mathbf{z})_j  \left[ f\left(P(\delta,j)\right) - f\left(P(\delta,j-1)\right) \right].\]
\end{lemma}

\begin{lemma}
\label{lem:tree_t_large}
For $j\in\{1,\dots,|\mathcal{V}|\}$, when $t(\delta, \mathbf{z})_j > 1$, 
\[ f\left(\widetilde{\mathbf{z}}^j\right) - f\left(\widetilde{\mathbf{z}}^{j-1}\right) \geq  t(\delta, \mathbf{z})_j  \left[ f\left(P(\delta,j)\right) - f\left(P(\delta,j-1)\right) \right].\]
\end{lemma}

\begin{proposition}
\label{prop:valid_P}
The DR inequality associated with $\delta$ is valid at every $[P(\mathcal{S}), f(P(\mathcal{S}))]$, $\mathcal{S}\subseteq\mathcal{V}$ in cases (\hyperlink{1}{1}) and (\hyperlink{2}{2}).
\end{proposition}
\begin{proof}
Based on Lemmas \ref{lem:tree_t_small} and \ref{lem:tree_t_large}. 
\[
f(\mathbf{z})  = \sum_{j=1}^{|\mathcal{V}|} [f(\widetilde{\mathbf{z}}^j) - f(\widetilde{\mathbf{z}}^{j-1})]  \geq \sum_{j=1}^{|\mathcal{V}|}  t(\delta, \mathbf{z})_j  \left[ f\left(P(\delta,j)\right) - f\left(P(\delta,j-1)\right) \right]  
\]
 in cases (\hyperlink{1}{1}) and (\hyperlink{2}{2}).
\end{proof}

\subsection{Validity for Case (3)} 
\label{sect:valid_N}
Recall that $\mathbf{z} = P(\mathcal{S})$ for some $\mathcal{S}\subseteq\mathcal{V}$. In case (\hyperlink{3}{3}), $\Psi \neq \emptyset$, which means that there exists fractionally upper bounded continuous variables with at least one discrete descendant. In this case, $u_\psi$ is free for all $\psi\in\Psi$. Recall 
\[\eta_\psi(\mathbf{z}) = \frac{z_\psi - (u_\psi - \floor{u_\psi})u_{\ch{\psi}}}{u_{\ch{\psi}} - u_\psi}.\]
\begin{observation}
\label{obs:good_delta}
For all $\psi\in\Psi$ with $u_\psi > 1$, if a valid permutation $\delta$ satisfies $\delta^{-1}(\psi) > \delta^{-1}(\ch{\psi})$, then $\mathbf{t}(\delta, \mathbf{z})$ does not involve $\eta_\psi(\mathbf{z})$ for any $\psi\in\Psi$. It follows that $\mathbf{t}(\delta, \mathbf{z})\geq \mathbf{0}$, and the validity of the DR inequality associated with $\delta$ with respect to $\mathbf{z}$ follows from the arguments in Section \ref{sect:valid_P}. 
\end{observation}
In fact, a valid permutation $\delta$ could be such that $\delta^{-1}(\psi) < \delta^{-1}(\ch{\psi})$ for $\psi\in\overline{\Psi}$, where $\overline{\Psi}\subseteq \{\psi\in\Psi: u_\psi > 1\}$. For brevity, we denote $\delta^{-1}(\psi)$ by $j$ for an arbitrary $\psi\in\overline{\Psi}$. Given such a $\delta$, 
\begin{align*}
t(\delta, \mathbf{z})_{j} & = \frac{\eta_\psi(\mathbf{z}) - z_{\psi^\vartriangle(\mathcal{T}^{\delta, j-1})}}{\floor{u_\psi} - u_{\psi^\vartriangle(\mathcal{T}^{\delta, j-1})}}\\
& = \frac{z_\psi - (u_\psi - \floor{u_\psi}) z_{\ch{\psi}}  - (u_{\ch{\psi}} - u_\psi)z_{\psi^\vartriangle(\mathcal{T}^{\delta, j-1})}}{(\floor{u_\psi} - u_{\psi^\vartriangle(\mathcal{T}^{\delta, j-1})})(u_{\ch{\psi}} - u_\psi)}\\
& = \frac{(u_\psi - \floor{u_\psi}) z_\psi + (u_{\ch{\psi}} - u_\psi)z_\psi - (u_\psi - \floor{u_\psi}) z_{\ch{\psi}}  - (u_{\ch{\psi}} - u_\psi)z_{\psi^\vartriangle(\mathcal{T}^{\delta, j-1})}}{(\floor{u_\psi} - u_{\psi^\vartriangle(\mathcal{T}^{\delta, j-1})})(u_{\ch{\psi}} - u_\psi)}\\
& = \frac{z_\psi- z_{\psi^\vartriangle(\mathcal{T}^{\delta, j-1})}}{\floor{u_\psi} - u_{\psi^\vartriangle(\mathcal{T}^{\delta, j-1})}} - \frac{(u_\psi - \floor{u_\psi})(z_{\ch{\psi}} - z_\psi)}{(\floor{u_\psi} - u_{\psi^\vartriangle(\mathcal{T}^{\delta, j-1})})(u_{\ch{\psi}} - u_\psi)},
\end{align*}
where the denominators are strictly positive because $\delta$ is valid. Now we analyze the sign of $t(\delta, \mathbf{z})_{j}$ in relation to $\mathcal{S}$. Suppose $\psi, \ch{\psi}\in\mathcal{S}$. Then $z_{\ch{\psi}} = u_{\ch{\psi}}$ and $z_\psi = u_\psi$, such that 
\[t(\delta, \mathbf{z})_{j} = \frac{u_\psi- z_{\psi^\vartriangle(\mathcal{T}^{\delta, j-1})}}{\floor{u_\psi} - u_{\psi^\vartriangle(\mathcal{T}^{\delta, j-1})}} - \frac{u_\psi - \floor{u_\psi}}{\floor{u_\psi} - u_{\psi^\vartriangle(\mathcal{T}^{\delta, j-1})}} > 0.\]
This is positive because $z_{\psi^\vartriangle(\mathcal{T}^{\delta, j-1})} \leq u_{\psi^\vartriangle(\mathcal{T}^{\delta, j-1})} \leq \floor{u_\psi} - 1$ must be satisfied for such a valid $\delta$ to exist. Now suppose $\psi\in\mathcal{S}$ and $\ch{\psi}\notin\mathcal{S}$, or $\psi,\ch{\psi}\notin\mathcal{S}$. In either case, $z_{\ch{\psi}} = z_\psi$, so 
\[t(\delta, \mathbf{z})_{j} = \frac{z_\psi- z_{\psi^\vartriangle(\mathcal{T}^{\delta, j-1})}}{\floor{u_\psi} - u_{\psi^\vartriangle(\mathcal{T}^{\delta, j-1})}} \geq 0.\]
\begin{observation}
\label{obs:good_S}
Suppose a valid permutation $\delta$ is such that  $\delta^{-1}(\psi) < \delta^{-1}(\ch{\psi})$ for $\psi\in\overline{\Psi}$. For any $\mathbf{z} = P(\mathcal{S})$ such that for every $\psi\in\overline{\Psi}$, $\ch{\psi}\in\mathcal{S}$ only if $\psi\in\mathcal{S}$, we know that $t(\delta, \mathbf{z})_{\delta^{-1}(\psi)} \geq 0$. Therefore, $\mathbf{t}(\delta, \mathbf{z})\geq \mathbf{0}$, and the validity of the DR inequality associated with $\delta$ with respect to $\mathbf{z}$ follows from the arguments in Section \ref{sect:valid_P}. 
\end{observation}

We next explore the remaining scenario not covered by Observations \ref{obs:good_delta} and \ref{obs:good_S}. Given $\mathcal{S}\subseteq\mathcal{V}$, we let 
\[\widecheck{\Psi} := \{\psi\in\Psi: u_\psi > 1, \delta^{-1}(\psi) < \delta^{-1}(\ch{\psi}), \psi\notin\mathcal{S}, \ch{\psi}\in\mathcal{S}\}. \]
Suppose $\widecheck{\Psi}\neq\emptyset$. Then $z_{\ch{\psi}} = u_{\ch{\psi}}$ and $z_\psi \leq \floor{u_\psi}-1$ for any $\psi\in\widecheck{\Psi}$. It is now possible that $t(\delta, \mathbf{z})_{\delta^{-1}(\psi)} < 0$. \\

Consider any $\psi\in\widecheck{\Psi}$. Let $j = \delta^{-1}(\psi)$ for brevity. For any $j < j' < \delta^{-1}(\ch{\psi})$ with $\delta(j') \in R^+(\psi)$, we observe
\begin{align*}
t(\delta, \mathbf{z})_{j'} & = \frac{z_{\delta(j')} - \eta_\psi(\mathbf{z})}{u_{\delta(j')} - \floor{u_\psi}}\\
& = \frac{(u_{\ch{\psi}} - u_\psi)z_{\delta(j')} - z_\psi + (u_\psi - \floor{u_\psi})z_{\ch{\psi}}}{(u_{\delta(j')} - \floor{u_\psi})(u_{\ch{\psi}} - u_\psi)} \\
& = \frac{(u_{\ch{\psi}} - u_\psi)z_{\delta(j')} - z_\psi + (u_\psi - \floor{u_\psi})z_{\ch{\psi}}- z_{\ch{\psi}} + z_{\ch{\psi}}}{(u_{\delta(j')} - \floor{u_\psi})(u_{\ch{\psi}} - u_\psi)} \\
& = \frac{(u_{\ch{\psi}} - u_\psi)(z_{\delta(j')}-z_{\ch{\psi}}) + z_{\ch{\psi}} - z_\psi}{(u_{\delta(j')} - \floor{u_\psi})(u_{\ch{\psi}} - u_\psi)}.
\end{align*}
Particularly, when $u_{\delta(j')}=u_{\ch{\psi}}$, $z_{\delta(j')}=z_{\ch{\psi}}=u_{\ch{\psi}}$, which gives $t(\delta, \mathbf{z})_{j'}$ a special form: 
\[
t(\delta, \mathbf{z})_{j'} = \frac{ z_{\ch{\psi}} - z_\psi}{(u_{\ch{\psi}} - \floor{u_\psi})(u_{\ch{\psi}} - u_\psi)} = \frac{ z_{\ch{\psi}} - z_\psi}{u_{\ch{\psi}} - u_\psi}.
\] We notice that
\[
t(\delta, \mathbf{z})_{j} = \frac{z_\psi- z_{\psi^\vartriangle(\mathcal{T}^{\delta, j-1})}}{\floor{u_\psi} - u_{\psi^\vartriangle(\mathcal{T}^{\delta, j-1})}} - \frac{u_\psi - \floor{u_\psi}}{\floor{u_\psi} - u_{\psi^\vartriangle(\mathcal{T}^{\delta, j-1})}}t(\delta, \mathbf{z})_{j'}. 
\]
We impose the following assumption for case (\hyperlink{3}{3}), so that the special form of $t(\delta, \mathbf{z})_{j'}$ holds. \\

\textbf{Assumption \hypertarget{A2}{2}.} {When there exists a fractionally upper bounded continuous variable $\psi\in\Psi$ with $u_\psi > 1$}, we assume that $u_i = \ceil{u_\psi}$ for all $i\in R^+(\psi)\backslash \{\psi\}$. 

{
\begin{remark}
Any instance of $\mathcal{Z}(\mathcal{G},\mathbf{u})$ under the special cases (a)--(d) described in Remark \ref{remark:trivialA1} trivially satisfies Assumption \hyperlink{A2}{2} because $\Psi = \emptyset$. 
\end{remark}
}

For $j\in\{0,\dots,|\mathcal{V}|\}$ and $i\in\mathcal{V}$, we define $\widecheck{\mathbf{z}}^j\in\mathbb{R}^{|\mathcal{V}|}$ by  
\begin{equation}
\widecheck{z}^j_i := z_{i^\vartriangle(\mathcal{T}^{\delta, j})}. 
\end{equation}
We note that $z_{i^\vartriangle(\mathcal{T}^{\delta, j})} = u_{[i^\vartriangle(\mathcal{T}^{\delta, j})]^\vartriangle(\mathcal{S})}$, so $\widecheck{\mathbf{z}}^j \leq P(\delta, j)$ for all $j\in\{0,\dots,|\mathcal{V}|\}$. By definition, $\widecheck{\mathbf{z}}^0 = \mathbf{0}$ and $\widecheck{\mathbf{z}}^{|\mathcal{V}|} = \mathbf{z}$. As such, $\mathbf{z} = \sum_{j=1}^{|\mathcal{V}|} (\widecheck{\mathbf{z}}^j - \widecheck{\mathbf{z}}^{j-1})$, and 
\[f(\mathbf{z})  = \sum_{j=1}^{|\mathcal{V}|} [f(\widetilde{\mathbf{z}}^j) - f(\widetilde{\mathbf{z}}^{j-1})]. \]
\begin{observation}
\label{obs:main_r_step}
When $\delta(j) = \psi\in\widecheck{\Psi}$, $\widecheck{z}^j_i = z_{\psi}$ for every $i\in\sigma_{\psi}(\mathcal{T}^{\delta, j})$ where $\sigma_{\psi}(\mathcal{T}^{\delta, j}) = \{i\in R^+(\psi): \delta^{-1}(i) \geq j\}$. This means that 
\begin{align*}
\widecheck{\mathbf{z}}^j - \widecheck{\mathbf{z}}^{j-1} & = (z_\psi - z_{\psi^\vartriangle(\mathcal{T}^{\delta, j-1})})\sum_{i\in\sigma_{\psi}(\mathcal{T}^{\delta, j})}\mathbf{e}^i \\
& = \frac{z_\psi - z_{\psi^\vartriangle(\mathcal{T}^{\delta, j-1})}}{\floor{u_\psi} - u_{\psi^\vartriangle(\mathcal{T}^{\delta, j-1})}} (\floor{u_\psi} - u_{\psi^\vartriangle(\mathcal{T}^{\delta, j-1})})\sum_{i\in\sigma_{\psi}(\mathcal{T}^{\delta, j})}\mathbf{e}^i \\
& =  \frac{z_\psi - z_{\psi^\vartriangle(\mathcal{T}^{\delta, j-1})}}{\floor{u_\psi} - u_{\psi^\vartriangle(\mathcal{T}^{\delta, j-1})}}[P(\delta, j) - P(\delta, j-1)].
\end{align*}
We further notice that $\frac{z_\psi - z_{\psi^\vartriangle(\mathcal{T}^{\delta, j-1})}}{\floor{u_\psi} - u_{\psi^\vartriangle(\mathcal{T}^{\delta, j-1})}} \geq 0$. Therefore, 
\[
f(\widecheck{\mathbf{z}}^j) - f(\widecheck{\mathbf{z}}^{j-1}) \geq \frac{z_\psi - z_{\psi^\vartriangle(\mathcal{T}^{\delta, j-1})}}{\floor{u_\psi} - u_{\psi^\vartriangle(\mathcal{T}^{\delta, j-1})}} \left[f\left(P(\delta, j)\right) - f(P(\delta, j-1))\right] 
\]
by the same arguments as in Lemmas \ref{lem:tree_t_small} and \ref{lem:tree_t_large}.
\end{observation}
For ease of notation, we let 
\begin{equation}
J(\psi) := \{\delta^{-1}(i) : i\in \sigma_{\psi}(\mathcal{T}^{\delta, \delta^{-1}(\psi)})\backslash\{\psi\}\}
\end{equation}
for every $\psi\in\widecheck{\Psi}$. We sort the elements in $J(\psi)$ by ascending order. In other words, $J(\psi)_1$ is the closest to $\delta^{-1}(\psi)$ among elements of $J(\psi)$, and $\delta(J(\psi)_{|J(\psi)|}) = \ch{\psi}$. We let $J(\psi)_0 = \delta^{-1}(\psi)$. \\

For any $j\in J(\psi)$, let $i = \delta(j)$. We observe that for any $\ell\in\mathcal{V}$, 
\begin{equation}
\label{z_ell}
\widecheck{z}^{j}_{\ell} = \begin{cases}
z_i = u_{\ch{\psi}}, & \ell\in \sigma_{i}(\mathcal{T}^{\delta, j}) = \{i\}, \\
\widecheck{z}^{j-1}_{\ell}, & \text{otherwise.}
\end{cases}
\end{equation}
Here, $\sigma_{i}(\mathcal{T}^{\delta, j}) = \{i\}$ because all elements of $R^+(i)$ have the same upper bound and $\delta$ is valid. We now define two new sets of points $Q(\delta, j), T(\delta, j)\in\mathbb{R}^{|\mathcal{V}|}$. For every $\ell\in\mathcal{V}$, 
\begin{equation}
Q(\delta, j)_\ell := \begin{cases}
u_\psi, & \ell = \delta(j), \\
P(\delta, j)_\ell, & \text{otherwise.}
\end{cases}
\end{equation}
Moreover,
\begin{equation}
T(\delta, j)_\ell := 
\begin{cases}
u_\psi, & \ell \in \{\delta(j'): j' \leq j, j'\in J(\psi)\}, \text{ or } \ell = \psi \text{ and } \delta(j) = \ch{\psi},\\
P(\delta, \delta^{-1}(\psi))_\ell, & \text{otherwise.}
\end{cases} 
\end{equation}
We let $T(\delta, \delta^{-1}(\psi)) = P(\delta, \delta^{-1}(\psi))$.

\begin{lemma}
\label{lem:Jr}
For any $\psi\in\widecheck{\Psi}$ and any $j\in J(\psi)$, the following inequality holds:  
\[f(\widecheck{\mathbf{z}}^{j}) - f(\widecheck{\mathbf{z}}^{j-1}) \geq \frac{u_{\ch{\psi}}-z_\psi}{u_{\ch{\psi}}-u_\psi} [f(P(\delta, j)) - f(P(\delta, j-1)) + f(P(\delta, j-1)) - f(Q(\delta, j))].\]
\end{lemma}

The right-hand side of the inequality in Lemma \ref{lem:Jr} contains the term
\[\frac{u_{\ch{\psi}}-z_\psi}{u_{\ch{\psi}}-u_\psi} [f(P(\delta, j)) - f(P(\delta, j-1))] = t(\delta,\mathbf{z})_j[f(P(\delta, j)) - f(P(\delta, j-1))].\]
We next relate the remaining terms to $t(\delta,\mathbf{z})_{\delta^{-1}(\psi)} [f(P(\delta, \delta^{-1}(\psi))) - f(P(\delta, \delta^{-1}(\psi)-1))]$. Observation \ref{obs:PQT} provides relevant details for the derivation in Lemma \ref{lem:tr}. 

\begin{observation}
\label{obs:PQT}
Let any $j\in J(\psi)$ be given. We denote $\delta(j)$ by $i$, and $j = J(\psi)_{\ell}$ for some $\ell\in\{1,\dots, |J(\psi)|\}$.  For any $q\in \mathcal{V}\backslash \{i\}$, $P(\delta, j-1)_q = P(\delta, j)_q = Q(\delta, j)_q$ when $i \neq \ch{\psi}$, and $Q(\delta, j) - P(\delta, j-1) = (u_\psi - \floor{u_\psi}) \mathbf{e}^i$. When $i = \ch{\psi}$, $P(\delta, j-1)_q = Q(\delta, j)_q$ for all $q\in\mathcal{V}\backslash\{\psi, \ch{\psi}\}$, and $Q(\delta, j) - P(\delta, j-1) = (u_\psi - \floor{u_\psi}) (\mathbf{e}^{\psi} + \mathbf{e}^{\ch{\psi}})$. By construction, $Q(\delta, j)_q= T(\delta, j)_q$ for all $q \in \{\delta(j'): j'\in J(\psi)\}$, and $Q(\delta, j)_q\geq T(\delta, j)_q$ for all other entries. Therefore, $Q(\delta,j) \geq T(\delta, j)$ component-wise. Furthermore, $Q(\delta, j) - P(\delta, j-1) = T(\delta, J(\psi)_{\ell}) - T(\delta, J(\psi)_{\ell-1})$ regardless of $i = \ch{\psi}$ or not. 
\end{observation}

\begin{lemma}
\label{lem:tr}
For any $\psi\in\widecheck{\Psi}$, 
\[\sum_{j\in J(\psi)} [f(P(\delta, j-1)) - f(Q(\delta, j))] \geq - \frac{u_\psi - \floor{u_\psi}}{\floor{u_\psi}- u_{\psi^\vartriangle(\mathcal{T}^{\delta, j-1})}} [f(P(\delta, \delta^{-1}(\psi))) - f(P(\delta, \delta^{-1}(\psi)-1))].\]
\end{lemma}

Observation \ref{obs:main_r_step}, Lemma \ref{lem:Jr} and Lemma \ref{lem:tr} have shown that for any $\psi\in\widecheck{\Psi}$, 
\begin{equation}
\label{eq:Jr_tr}
\sum_{j\in J(\psi)} [f(\widecheck{\mathbf{z}}^{j}) - f(\widecheck{\mathbf{z}}^{j-1})] \geq \sum_{j\in J(\psi)\cup\{\psi\}} t(\delta,\mathbf{z})_j[f(P(\delta, j)) - f(P(\delta, j-1))]. 
\end{equation}

\begin{lemma}
\label{lem:other_j}
For all $j\notin \bigcup_{\psi\in\widecheck{\Psi}} \left(J(\psi)\cup \{\psi\}\right)$, 
\[f(\widecheck{\mathbf{z}}^{j}) - f(\widecheck{\mathbf{z}}^{j-1}) \geq t(\delta,\mathbf{z})_j [f(P(\delta, j)) - f(P(\delta, j-1))]. \]
\end{lemma}

\begin{proposition}
\label{prop:valid_N}
The DR inequality associated with $\delta$ is valid at every $[P(\mathcal{S}), f(P(\mathcal{S}))]$, $\mathcal{S}\subseteq\mathcal{V}$, in case (\hyperlink{3}{3}) under Assumption \hyperlink{A2}{2}.
\end{proposition}
\begin{proof}
By Observations \ref{obs:good_delta} and \ref{obs:good_S}, it suffices to discuss the case where $\widecheck{\Psi}\neq \emptyset$. When $\widecheck{\Psi}\neq \emptyset$, 
\begingroup
\allowdisplaybreaks
\begin{align*}
f(\mathbf{z}) & = \sum_{j=1}^{|\mathcal{V}|} [f(\widecheck{\mathbf{z}}^j) - f(\widecheck{\mathbf{z}}^{j-1})] \\
& = \sum_{\psi\in\widecheck{\Psi}} [f(\widecheck{\mathbf{z}}^{\delta^{-1}(\psi)}) - f(\widecheck{\mathbf{z}}^{{\delta^{-1}(\psi)}-1})] + \sum_{\psi\in\widecheck{\Psi}} \sum_{j\in J(\psi)} [f(\widecheck{\mathbf{z}}^j) - f(\widecheck{\mathbf{z}}^{j-1})]  \\
& \quad + \sum_{j\notin \bigcup_{\psi\in\widecheck{\Psi}} \left(J(\psi)\cup \{\psi\}\right)}[f(\widecheck{\mathbf{z}}^j) - f(\widecheck{\mathbf{z}}^{j-1})]  \\ 
& \geq \sum_{\psi\in\widecheck{\Psi}} \sum_{j\in J(\psi)\cup\{\psi\}} t(\delta,\mathbf{z})_j[f(P(\delta, j)) - f(P(\delta, j-1))] \quad \text{(from Observation \ref{obs:main_r_step}, Lemmas \ref{lem:Jr} and \ref{lem:tr})}\\
& \quad + \sum_{j\notin \bigcup_{\psi\in\widecheck{\Psi}} \left(J(\psi)\cup \{\psi\}\right)} t(\delta,\mathbf{z})_j [f(P(\delta, j)) - f(P(\delta, j-1))] \quad \text{(by Lemma \ref{lem:other_j})} \\
& =  \sum_{j=1}^{|\mathcal{V}|} t(\delta,\mathbf{z})_j [f(P(\delta, j)) - f(P(\delta, j-1))]. 
\end{align*}
\endgroup
\end{proof}

With the results derived in Sections \ref{sect:valid_P} and \ref{sect:valid_N}, we conclude on the validity of the DR inequalities. 
\begin{proposition}
\label{prop:tree_valid}
Every DR inequality \eqref{eq:DR} associated with a valid permutation $\delta\in\mathfrak{S}(\mathcal{V})$ is valid for $\mathcal{P}_f^{\mathcal{Z}(\mathcal{G},\mathbf{u})}$, where $\mathcal{Z}(\mathcal{G},\mathbf{u})$ satisfies Assumptions \hyperlink{A1}{1} and \hyperlink{A2}{2}. 
\end{proposition}
\begin{proof}
Propositions \ref{prop:valid_P} and \ref{prop:valid_N} show that the DR inequality associated with $\delta$ is valid at $[P(\mathcal{S}), f(P(\mathcal{S}))]^\top$ for all $\mathcal{S}\in\mathcal{V}$. By Proposition \ref{prop:tree_valid_at_ext}, the DR inequality is valid for $\mathcal{P}_f^{\mathcal{Z}(\mathcal{G},\mathbf{u})}$. 
\end{proof}

{
\begin{remark}
Suppose we are given an epigraph $\mathcal{P}^{\mathcal{Z}(\mathcal{G}, \mathbf{u})}_f$, where $\mathcal{Z}(\mathcal{G}, \mathbf{u})$ does not satisfy Assumption \hyperlink{A1}{1} or Assumption \hyperlink{A2}{2}. We remark that the DR inequalities associated with a slight relaxation of $\mathcal{P}^{\mathcal{Z}(\mathcal{G}, \mathbf{u})}_f$ are valid for $\mathcal{P}^{\mathcal{Z}(\mathcal{G}, \mathbf{u})}_f$. We first relax $\mathcal{Z}(\mathcal{G}, \mathbf{u})$ by increasing the fractional upper bounds of the continuous variables that violate the assumptions, so that $\Psi$ becomes empty. For example, given any $\psi\in\Psi$, we replace $u_\psi$ by $u'_\psi = \ceil{u_\psi}$. The resulting relaxed feasible set, $\mathcal{Z}(\mathcal{G}, \mathbf{u}')$, satisfies both assumptions by construction. We now consider the DR inequalities for $\mathcal{P}^{\mathcal{Z}(\mathcal{G}, \mathbf{u}')}_f$. Based on the argument for validity of these DR inequalities for $\mathcal{P}^{\mathcal{Z}(\mathcal{G}, \mathbf{u}')}_f$, we obtain that every DR inequality holds at $(\overline{\mathbf{z}}, \overline{w})$, for any $\overline{\mathbf{z}} \in \conv{\mathcal{Z}(\mathcal{G}, \mathbf{u}')} \supseteq \mathcal{Z}(\mathcal{G}, \mathbf{u})$ and any $\overline{w} \geq f(\overline{\mathbf{z}})$. Therefore, these DR inequalities are valid for $\mathcal{P}^{\mathcal{Z}(\mathcal{G}, \mathbf{u})}_f$. 
\end{remark}
}

\section{Characterization of $\conv{\mathcal{P}_f^{\mathcal{Z}(\mathcal{G},\mathbf{u})}}$ and Exact Separation}
\label{sect:punchline}
We are now ready to fully characterize the convex hull of the epigraph of $f$. In the next theorem, we show that the trivial inequalities and the DR inequalities are sufficient to capture $\conv{\mathcal{P}_f^{\mathcal{Z}(\mathcal{G},\mathbf{u})}}$. Throughout this section, we assume that Assumptions \hyperlink{A1}{1} and \hyperlink{A2}{2} hold for $\mathcal{Z}(\mathcal{G},\mathbf{u})$.

\begin{theorem}
\label{thm:tree_m_conv}
All the DR inequalities associated with valid permutations, along with the MIR inequalities, the box constraints, and the monotonicity constraints for $\conv{\mathcal{Z}(\mathcal{G},\mathbf{u})}$, fully describe $\conv{\mathcal{P}_f^{\mathcal{Z}(\mathcal{G},\mathbf{u})}}$.  
\end{theorem}
\begin{proof}
Let $\mathcal{CP}$ be the set constructed by all the aforementioned inequalities. We would like to show that $\mathcal{CP} = \conv{\mathcal{P}_f^{\mathcal{Z}(\mathcal{G},\mathbf{u})}}$. By validity of the DR inequalities and the MIR inequalities, $\mathcal{CP} \supseteq \conv{\mathcal{P}_f^{\mathcal{Z}(\mathcal{G},\mathbf{u})}}$. For the reverse containment, we show that any $(\mathbf{z}, w)\in \mathcal{CP}$ belongs to $\conv{\mathcal{P}_f^{\mathcal{Z}(\mathcal{G},\mathbf{u})}}$, by writing $(\mathbf{z}, w)$ as a convex combination of certain elements of $\conv{\mathcal{P}_f^{\mathcal{Z}(\mathcal{G},\mathbf{u})}}$. Given that $(\mathbf{z}, w)\in \mathcal{CP}$, we have $\mathbf{z} \in \conv{\mathcal{Z}(\mathcal{G},\mathbf{u})}$. Thus we can run Algorithm \ref{alg:tree_greedy} with input $\mathcal{G}$, $\mathbf{u}$ and $\mathbf{z}$ to determine a valid permutation $\delta\in\mathfrak{S}(\mathcal{V})$. For brevity, we let $\lambda_k = t(\delta,\mathbf{z})_k - t(\delta,\mathbf{z})_{k+1}$ for every $k\in\{0,\dots, |\mathcal{V}|\}$. It follows from Proposition \ref{prop:tree_conv_z} that $\mathbf{z} = \sum_{k=0}^{|\mathcal{V}|} \lambda_k P(\delta, k)$ is a convex combination of $\{P(\delta, k)\}_{k=0}^{|\mathcal{V}|}$. Moreover, $(\mathbf{z}, w)$ satisfies the DR inequality associated with $\delta$, which means that 
\[w \geq \sum_{k=1}^{|\mathcal{V}|} [f(P(\delta, k)) -f(P(\delta, k-1))] t(\delta,\mathbf{z})_k = \sum_{k=0}^{|\mathcal{V}|} \lambda_k f(P(\delta, k)). \]
To this point, we have shown that 
\[(\mathbf{z}, w) = \sum_{k=0}^{|\mathcal{V}|} \lambda_k [P(\delta, k), f(P(\delta, k))] + \left(w - \sum_{k=0}^{|\mathcal{V}|} \lambda_k f(P(\delta, k))\right)(\mathbf{0}, 1).\]
This is a convex combination of $ [P(\delta, k), f(P(\delta, k))] \in \mathcal{P}_f^{\mathcal{Z}(\mathcal{G},\mathbf{u})}$ for all $k\in\{0,1,\dots,|\mathcal{V}|\}$ plus a non-negative multiple of the ray $(\mathbf{0},1)$. Therefore, $(\mathbf{z}, w)\in \conv{\mathcal{P}_f^{\mathcal{Z}(\mathcal{G},\mathbf{u})}}$, and  $\mathcal{CP} \subseteq \conv{\mathcal{P}_f^{\mathcal{Z}(\mathcal{G},\mathbf{u})}}$. We conclude that $\mathcal{CP} = \conv{\mathcal{P}_f^{\mathcal{Z}(\mathcal{G},\mathbf{u})}}$.
\end{proof}

{
\begin{remark}
The convex hull characterization given in Theorem \ref{thm:tree_m_conv} can be stated more concisely in the special cases of $\mathcal{Z}(\mathcal{G},\mathbf{u})$ described in Remark \ref{remark:trivialA1}. Recall that Assumptions \hyperlink{A1}{1} and \hyperlink{A2}{2} trivially hold in all these cases. In case (a), when $\mathcal{Z}(\mathcal{G},\mathbf{u})$ is defined by box constraints only, $\conv{\mathcal{P}_f^{\mathcal{Z}(\mathcal{G},\mathbf{u})}}$ is completely described by the box constraints and the DR inequalities associated with all $\delta\in\mathfrak{S}(\mathcal{V})$, which assume the simplified form of \eqref{eq:special_a_DR}. Now consider $\mathcal{Z}(\mathcal{G},\mathbf{u})$ that belongs to any of the following cases. 
\begin{itemize}
\item[(b)] All variables are purely integers. 
\item[(c)] All variables are purely continuous. 
\item[(d)] The upper bounds $\mathbf{u}\in\mathbb{Z}^{|\mathcal{V}|}$, while each variable $z_i$, $i\in\mathcal{V}$, can be either discrete or continuous. 
\end{itemize}
Then $\conv{\mathcal{P}_f^{\mathcal{Z}(\mathcal{G},\mathbf{u})}}$ is completely described by the box constraints, the monotonicity constraints, and the DR inequalities associated with valid $\delta\in\mathfrak{S}(\mathcal{V})$, which assume the simplified form of \eqref{eq:special_bcd_DR}. 
\end{remark}
}

We propose Algorithm \ref{alg:tree_sepa} to separate DR inequalities at any infeasible solution $(\overline{\mathbf{z}}, \overline{w})\in\conv{\mathcal{Z}(\mathcal{G},\mathbf{u})}\times \mathbb{R}$. In this algorithm, we determine a valid permutation $\delta$ using Algorithm \ref{alg:tree_greedy} with respect to $\mathcal{G},\mathbf{u}$ and $\overline{\mathbf{z}}$. The DR inequality associated with this particular permutation $\delta$ is the most violated DR inequality at $(\overline{\mathbf{z}}, \overline{w})$. We justify this claim in Proposition \ref{prop:tree_exact_sepa}. 

\begin{algorithm}[H]
\label{alg:tree_sepa}
\SetAlgoLined
\textbf{Input} $(\overline{\mathbf{z}}, \overline{w})\in\conv{\mathcal{Z}(\mathcal{G},\mathbf{u})}\times \mathbb{R}$\;
$\delta \leftarrow$ \texttt{Permutation\_Finder}($\mathcal{G}, \mathbf{u}, \overline{\mathbf{z}}$) (Algorithm \ref{alg:tree_greedy})\;
   \textbf{Output} DR inequality associated with $\delta$.
   \caption{\texttt{DR\_Inequality\_Separation}}
\end{algorithm}

\begin{proposition}
\label{prop:tree_exact_sepa}
Algorithm \ref{alg:tree_sepa} is an exact separation method for DR inequalities. 
\end{proposition}
\begin{proof}
Given any infeasible solution $(\overline{\mathbf{z}}, \overline{w})\in\conv{\mathcal{Z}(\mathcal{G},\mathbf{u})}\times \mathbb{R}$, we consider the separation problem 
\[\max\left\{\overline{\mathbf{z}}^\top \boldsymbol\pi + \pi_0: \mathbf{z}^\top\boldsymbol\pi + \pi_0 \leq f(\mathbf{z}), \forall \: \mathbf{z}\in\mathcal{Z}(\mathcal{G},\mathbf{u})\right\}.\]
The permutation $\delta$ returned by Algorithm \ref{alg:tree_greedy} is valid by construction, so there is a DR inequality associated with $\delta$. Let $\boldsymbol\pi$ be the coefficients of this DR inequality and $\pi_0=0$; this is a feasible solution to the primal separation problem by Proposition \ref{prop:tree_valid}. \\

Next, we examine the dual separation problem. For each $\mathbf{z}\in\mathcal{Z}(\mathcal{G},\mathbf{u})$, we assign a dual variable $\lambda(\mathbf{z})$. There are infinitely many $\lambda(\cdot)$ because of the continuity in $\mathcal{Z}(\mathcal{G},\mathbf{u})$. The dual problem has the form:
\[\min\left\{\sum_{\mathbf{z}\in\mathcal{Z}(\mathcal{G},\mathbf{u})} \lambda(\mathbf{z}) f(\mathbf{z}): \sum_{\mathbf{z}\in\mathcal{Z}(\mathcal{G},\mathbf{u})} \lambda(\mathbf{z})\mathbf{z} = \overline{\mathbf{z}}, \sum_{\mathbf{z}\in\mathcal{Z}(\mathcal{G},\mathbf{u})} \lambda(\mathbf{z}) = 1, \lambda(\mathbf{z}) \geq 0, \forall \: \mathbf{z}\in\mathcal{Z}(\mathcal{G},\mathbf{u})\right\}.\]
We propose the following dual solution: 
\[
\lambda(\mathbf{z}) = 
\begin{cases}
t(\delta,\overline{\mathbf{z}})_k - t(\delta,\overline{\mathbf{z}})_{k+1}, & \mathbf{z} = P(\delta, k), k\in\{0,\dots, |\mathcal{V}|\}, \\
0, &\text{otherwise.}
\end{cases}
\]
This solution is dual feasible due to Proposition \ref{prop:tree_conv_z}. We next check complementary slackness for optimality. For $j\in\mathcal{V}$, we would like to show that 
\[f(P(\delta, j)) = \sum_{k=1}^{|\mathcal{V}|} [f(P(\delta, k)) - f(P(\delta, k-1))] t(\delta, P(\delta, j))_k. \]
For $1 \leq k \leq j$, $t(\delta, P(\delta, j))_k = 1$; whereas $t(\delta, P(\delta, j))_k = 0$ for $j+1 \leq k \leq |\mathcal{V}|$. Therefore, 
\begin{align*}
\sum_{k=1}^{|\mathcal{V}|} [f(P(\delta, k)) - f(P(\delta, k-1))] t(\delta, P(\delta, j))_k & = \sum_{k=1}^{j} [f(P(\delta, k)) - f(P(\delta, k-1))] \\
& = f(P(\delta, j)) - f(P(\delta, 0)) \\
& =  f(P(\delta, j)).
\end{align*}
Hence, the proposed primal solution to the separation problem is optimal and our separation method is exact.
\end{proof}

The separation method Algorithm \ref{alg:tree_sepa} runs in $\mathcal{O}(|\mathcal{V}|^2\log|\mathcal{V}|))$, following the complexity of Algorithm \ref{alg:tree_greedy}. Hence, we reach the next corollary regarding the complexity of this class of constrained mixed-integer DR-submodular minimization problems. 

\begin{corollary}
DR-submodular minimization over the mixed-integer feasible set $\mathcal{Z}(\mathcal{G},\mathbf{u})$, that satisfies Assumptions \hyperlink{A1}{1} and \hyperlink{A2}{2}, is polynomial-time solvable.
\end{corollary}

\begin{remark}
Minimizing a DR-submodular function $f$ over a mixed-integer feasible set defined by box constraints is a special case of problem \eqref{eq:DR_min}, in which the directed rooted forest $\mathcal{G}$ consists of $|\mathcal{V}|$ disjoint trees with height zero. Suppose, additionally, $M = \emptyset$ and $u_i = 1$ for all $i\in N$; in other words, $\mathbf{z}\in\{0,1\}^{|\mathcal{V}|}$. Then problem \eqref{eq:DR_min} reduces to unconstrained submodular set function minimization. We note that our exact separation algorithm (Algorithm \ref{alg:tree_sepa}) generalizes the well-known greedy algorithm. To see this, any permutation $\delta$ of $\mathcal{V}$ is valid, and function $\mathbf{t}(\delta,\mathbf{z})$ has a special form $t(\delta,\mathbf{z})_k = z_{\delta(k)}/u_{\delta(k)}$. Given any $(\overline{\mathbf{z}}, \overline{w})\in\conv{\mathcal{Z}(\mathcal{G},\mathbf{u})}\times \mathbb{R}$, Algorithm \ref{alg:tree_sepa} sorts $\overline{z}_i/u_{i}$ for all $i\in\mathcal{V}$ just like the greedy algorithm. 
\end{remark}

\section*{Acknowledgement}
We thank the reviewers for their helpful feedback which improved this paper. This research is supported, in part, by ONR Grant N00014-22-1-2602.

\bibliography{ref}{}
\bibliographystyle{apalike}

\appendix

\section{Proofs of Lemmas} \label{sec:app}

\begin{manualtheorem}{3.1}
Let any $\mathbf{z}\in\mathcal{X}$ and any index $i\in N$ be given. For all $\alpha,\beta\in\mathbb{Z}$ such that $0 < \alpha < \beta$ and $\mathbf{z}+\alpha \mathbf{e}^i, \mathbf{z}+\beta \mathbf{e}^i\in\mathcal{X}$, 
\[f(\mathbf{z}+\alpha \mathbf{e}^i)-f(\mathbf{z})\geq \frac{\alpha}{\beta}[f(\mathbf{z}+\beta \mathbf{e}^i)-f(\mathbf{z})].\] 
\end{manualtheorem}
\begin{proof}
We first observe that
\begingroup
\allowdisplaybreaks
\begin{subequations}
\begin{align}
[f(\mathbf{z}+\alpha \mathbf{e}^i)-f(\mathbf{z})]/\alpha & =\frac{1}{\alpha}\sum_{j = 1}^\alpha [f(\mathbf{z}+j\mathbf{e}^i) - f(\mathbf{z}+(j-1)\mathbf{e}^i)] \notag \\
& \geq \frac{1}{\alpha} \sum_{j = 1}^\alpha [f(\mathbf{z}+(\alpha+1)\mathbf{e}^i)-f(\mathbf{z}+\alpha \mathbf{e}^i)  \quad \text{($f$ is DR-submodular)} \notag\\
& = f(\mathbf{z}+(\alpha+1)\mathbf{e}^i)-f(\mathbf{z}+\alpha \mathbf{e}^i). \label{dr_min_relation}
\end{align}
\end{subequations}
\endgroup 
With this observation, we deduce the following: 
\begingroup
\allowdisplaybreaks
\begin{align*}
f(\mathbf{z}+\beta \mathbf{e}^i) - f(\mathbf{z}) & = f(\mathbf{z}+\alpha \mathbf{e}^i)-f(\mathbf{z}) + \sum_{j=1}^{\beta-\alpha} [f(\mathbf{z}+\alpha \mathbf{e}^i + j\mathbf{e}^i) - f(\mathbf{z}+\alpha \mathbf{e}^i + (j-1)\mathbf{e}^i)] \\
& \leq f(\mathbf{z}+\alpha \mathbf{e}^i)-f(\mathbf{z}) + \sum_{j=1}^{\beta-\alpha} [f(\mathbf{z}+(\alpha+1)\mathbf{e}^i)-f(\mathbf{z}+\alpha\mathbf{e}^i)]  \quad \text{($f$ is DR-submodular)}\\
& \leq f(\mathbf{z}+\alpha \mathbf{e}^i)-f(\mathbf{z}) + (\beta-\alpha)[f(\mathbf{z}+\alpha \mathbf{e}^i)-f(\mathbf{z})]/{\alpha} \quad \text{(from \eqref{dr_min_relation})} \\
& =  \frac{\beta}{\alpha}[f(\mathbf{z}+\alpha \mathbf{e}^i)-f(\mathbf{z})].
\end{align*}
\endgroup 
\end{proof}

\begin{manualtheorem}{3.2}
Let any $\mathbf{z}\in\mathcal{X}$ and any index $i\in M$ be given. For any $\alpha,\beta\in\mathbb{R}$ with $0 < \alpha < \beta$ such that $\mathbf{z}+\alpha \mathbf{e}^i, \mathbf{z}+\beta \mathbf{e}^i\in\mathcal{X}$, 
\[f(\mathbf{z}+\alpha \mathbf{e}^i)-f(\mathbf{z})\geq \frac{\alpha}{\beta}[f(\mathbf{z}+\beta \mathbf{e}^i)-f(\mathbf{z})].\] 
\end{manualtheorem}
\begin{proof}
Recall that $\mathcal{X}_i\subseteq \mathbb{R}$ is a closed interval. Consider any $\mathbf{x},\mathbf{y}\in\mathcal{X}$ such that $x_j = y_j$ for all $j \in N \cup M\backslash \{i\}$. Without loss of generality, suppose $\mu = y_i - x_i \geq 0$. Then $\mathbf{y} = \mathbf{x} + \mu\mathbf{e}^i$. By DR-submodularity, 
\[f\left(\mathbf{x}+ \frac{\lambda}{2}\mathbf{e}^i\right) - f(\mathbf{x}) \geq f(\mathbf{x}+\lambda\mathbf{e}^i)-  f\left(\mathbf{x}+ \frac{\lambda}{2}\mathbf{e}^i\right).\]
After rearranging the terms, we obtain
\[f\left(\frac{1}{2}\mathbf{x} + \frac{1}{2}\mathbf{y}\right) = f\left(\frac{1}{2}\mathbf{x} + \frac{1}{2}(\mathbf{x}+\lambda \mathbf{e}^i)\right) \geq  \frac{1}{2} f(\mathbf{x}) +  \frac{1}{2} f\left(\mathbf{x}+ \lambda \mathbf{e}^i\right) = \frac{1}{2} f(\mathbf{x}) +  \frac{1}{2} f\left(\mathbf{y}\right),\]
which shows that $f$ is concave in $\mathcal{X}_i$ for any $i\in M$. Therefore, 
\[ f(\mathbf{z}+\alpha \mathbf{e}^i) - f(\mathbf{z}) \geq \frac{\alpha}{\beta}[f(\mathbf{z}+\beta \mathbf{e}^i) - f(\mathbf{z})].\] 
\end{proof}

\begin{manualtheorem}{3.3}
Let any $\mathbf{z}\in\mathcal{X}$ and any non-negative direction $\mathbf{d} = \sum_{i=1}^{n+m} d_i\mathbf{e}^i \geq \mathbf{0}$ be given. For a real number $0\leq t\leq 1$, such that $\mathbf{z} + \sum_{j=1}^i d_j\mathbf{e}^j, \mathbf{z} + t\sum_{j=1}^i d_j\mathbf{e}^j \in\mathcal{X}$ for all $i\in\{1,\dots,n+m\}$, the inequality 
\[f(\mathbf{z}+t\mathbf{d})-f(\mathbf{z})\geq t[f(\mathbf{z}+\mathbf{d})-f(\mathbf{z})]\] 
is satisfied. 
\end{manualtheorem}
\begin{proof}
We rewrite the left-hand side of the given inequality and find that
\begingroup
\allowdisplaybreaks
\begin{align*}
\quad  f(\mathbf{z}+t\mathbf{d}) - f(\mathbf{z}) 
& = \sum_{i=1}^{n+m} \left[f\left(\mathbf{z}+ t\sum_{j=1}^{i-1} d_j\mathbf{e}^j + td_i\mathbf{e}^i\right) - f\left(\mathbf{z}+ t\sum_{j=1}^{i-1} d_j\mathbf{e}^j\right)\right] \\
& \geq \sum_{i=1}^{n+m} \left[f\left(\mathbf{z}+ \sum_{j=1}^{i-1} d_j\mathbf{e}^j + td_i\mathbf{e}^i\right) - f\left(\mathbf{z}+ \sum_{j=1}^{i-1} d_j\mathbf{e}^j\right)\right]\\
& \quad \text{(because $0<t<1$ and $f$ is DR-submodular)}\\
& \geq \sum_{i=1}^{n+m} t\left[f\left(\mathbf{z}+ \sum_{j=1}^{i-1} d_j\mathbf{e}^j + d_i\mathbf{e}^i\right) - f\left(\mathbf{z}+ \sum_{j=1}^{i-1} d_j\mathbf{e}^j\right)\right]\\
& \quad \text{(by Lemmas \ref{lem:int_scale_increment} and \ref{lem:con_scale_increment})}\\
& = t[f(\mathbf{z}+\mathbf{d}) - f(\mathbf{z})].
\end{align*}
\endgroup 
\end{proof}

\begin{manualtheorem}{4.4}
Function $F: \mathcal{X}^E\rightarrow\mathbb{R}$ defined by \eqref{def:F} is DR-submodular. 
\end{manualtheorem}
\begin{proof}
Let any $\mathbf{x},\mathbf{y}\in\mathcal{X}^E$ with $\mathbf{x}\leq \mathbf{y}$ component-wise be given. We consider any $i\in \mathcal{V}^E$ and an arbitrary positive real number $\alpha$ such that $\mathbf{x}+\alpha\mathbf{e}^i, \mathbf{y}+\alpha\mathbf{e}^i\in\mathcal{X}^E$. Since $\mathbf{x}\leq \mathbf{y}$, $\mathbf{z}^{\mathbf{x}} \leq \mathbf{z}^{\mathbf{y}}$ must hold. 
If $i\in \mathcal{V}$, then $\mathbf{z}^{\mathbf{x}+\alpha\mathbf{e}^i} = \mathbf{z}^{\mathbf{x}}+\alpha\mathbf{e}^i$, where we assume $\mathbf{e}^i$ to have appropriate dimension by abusing notation. Similarly, $\mathbf{z}^{\mathbf{y}+\alpha\mathbf{e}^i} = \mathbf{z}^{\mathbf{y}}+\alpha\mathbf{e}^i$. We observe that, 
\begin{align*}
F(\mathbf{x}+\alpha\mathbf{e}^i) - F(\mathbf{x}) & = f(\mathbf{z}^{\mathbf{x}+\alpha\mathbf{e}^i}) - f(\mathbf{z}^{\mathbf{x}}) \quad \text{(by \eqref{def:F})} \\
& = f(\mathbf{z}^{\mathbf{x}}+\alpha\mathbf{e}^i) - f(\mathbf{z}^{\mathbf{x}}) \\
& \geq f(\mathbf{z}^{\mathbf{y}}+\alpha\mathbf{e}^i) - f(\mathbf{z}^{\mathbf{y}}) \quad \text{(by DR-submodularity of $f$)} \\
& = F(\mathbf{y}+\alpha\mathbf{e}^i) - F(\mathbf{y}). 
\end{align*} 
On the other hand, if $i = \rho $, then $\mathbf{z}^{\mathbf{x}+\alpha\mathbf{e}^i} = \mathbf{z}^{\mathbf{x}}$ and $\mathbf{z}^{\mathbf{y}+\alpha\mathbf{e}^i} = \mathbf{z}^{\mathbf{y}}$. It follows that 
\begin{align*}
F(\mathbf{x}+\alpha\mathbf{e}^i) - F(\mathbf{x}) & = f(\mathbf{z}^{\mathbf{x}}) - f(\mathbf{z}^{\mathbf{x}}) \\
& = f(\mathbf{z}^{\mathbf{y}}) - f(\mathbf{z}^{\mathbf{y}}) \\
& = f(\mathbf{z}^{\mathbf{y}+\alpha\mathbf{e}^i}) - f(\mathbf{z}^{\mathbf{y}}) \\
& = F(\mathbf{y}+\alpha\mathbf{e}^i) - F(\mathbf{y}). 
\end{align*}
Hence, $F$ is a DR-submodular function.
\end{proof}

\begin{manualtheorem}{4.5}
Given any $\overline{\mathbf{z}}\in\mathcal{Z}(\mathcal{G}, \mathbf{u})$, the vector
\[\overline{\mathbf{x}} := \begin{cases}
\overline{z}_i, & i\in \mathcal{V}, \\
 \ceil{\overline{z}_{\psi}}, &  i = \rho , 
\end{cases}\]
is an element of $\mathcal{Z}(\mathcal{G}^E, \mathbf{u}^E)$. Furthermore, for any ${\mathbf{x}}\in\mathcal{Z}(\mathcal{G}^E, \mathbf{u}^E)$, $\mathbf{z}^{{\mathbf{x}}}\in\mathcal{Z}(\mathcal{G}, \mathbf{u})$.
\end{manualtheorem}
\begin{proof}
By definition, $\overline{\mathbf{x}}$ satisfies the integrality restrictions and box constraints in $\mathcal{Z}(\mathcal{G}^E, \mathbf{u}^E)$. For every $i\in\ch{\rho }$ from $\mathcal{G}^E$, $i\in \ch{\psi}$  and $i\in N$ in $\mathcal{G}$ by Assumption \hyperlink{A1}{1}, so $\overline{x}_{\rho } = \ceil{\overline{z}_{\psi}} \leq \overline{z}_i = \overline{x}_i$. The monotonicity constraints in $\mathcal{Z}(\mathcal{G}^E, \mathbf{u}^E)$ are satisfied by $\overline{\mathbf{x}}$. Therefore, $\overline{\mathbf{x}}\in\mathcal{Z}(\mathcal{G}^E, \mathbf{u}^E)$. Now given any ${\mathbf{x}}\in\mathcal{Z}(\mathcal{G}^E, \mathbf{u}^E)$, $x_i = {z}^{{\mathbf{x}}}_i$ satisfies the box constraints and integrality restrictions for all $i\in \mathcal{V}$. We further observe that ${z}^{{\mathbf{x}}}_\psi \leq \ceil{x_\psi} \leq x_{\rho } \leq x_j =  {z}^{{\mathbf{x}}}_j$  for all $j\in\ch{\psi}$ in $\mathcal{G}$. Hence, we conclude that for any ${\mathbf{x}}\in\mathcal{Z}(\mathcal{G}^E, \mathbf{u}^E)$, $\mathbf{z}^{{\mathbf{x}}}\in\mathcal{Z}(\mathcal{G}, \mathbf{u})$.
\end{proof}

\begin{manualtheorem}{4.8}
For any $\mathcal{S}\subseteq\mathcal{V}$, $P(\mathcal{S})\in\mathcal{Z}(\mathcal{G},\mathbf{u})$. 
\end{manualtheorem}
\begin{proof}
First, we show that $P(\mathcal{S})$ satisfies the integrality constraints. Consider any $i\in N$. By definition, either $P(\mathcal{S})_i$ is $\floor{u_\psi}$ for some $\psi\in\Psi$, which is an integer; or $P(\mathcal{S})_i = u_{i^\vartriangle(\mathcal{S})}$ where $i^\vartriangle(\mathcal{S})\in \{0\}\cup N\cup \{j\in M: u_j\in\mathbb{Z}\}$, so $P(\mathcal{S})_i \in\mathbb{Z}$ as well. For any $i\in\mathcal{V}$,  $0\leq \floor{u_{i^\vartriangle(\mathcal{S})}} \leq u_{i^\vartriangle(\mathcal{S})} \leq u_i$ because $i^\vartriangle(\mathcal{S}) \in R^-(i)$. Thus $P(\mathcal{S})$ satisfies the box constraints. Given any $(i,j)\in\mathcal{A}$, $R^-(i)\cap\mathcal{S} \subseteq R^-(j)\cap\mathcal{S}$, so $i^\vartriangle(\mathcal{S}) \in R^-(j^\vartriangle(\mathcal{S}))$ and $P(\mathcal{S})_i = u_{i^\vartriangle(\mathcal{S})} \leq u_{j^\vartriangle(\mathcal{S})} = P(\mathcal{S})_j$. Therefore, $P(\mathcal{S})\in\mathcal{Z}(\mathcal{G}, \mathbf{u})$.
\end{proof}

\begin{manualtheorem}{4.9}
For every $\mathcal{S}\subseteq \mathcal{V}$, $P(\mathcal{S})$ is feasible to problem \eqref{frac_convZ_lp}. 
\end{manualtheorem}
\begin{proof}
The feasible set of problem \eqref{frac_convZ_lp} is $\mathcal{CZ}$. By validity of the MIR inequalities, $\mathcal{CZ} \supseteq \mathcal{Z}(\mathcal{G},\mathbf{u})$. It follows from Lemma \ref{lem:F_P_feas} that $P(\mathcal{S})\in\mathcal{CZ}$.
\end{proof}

\begin{manualtheorem}{4.11}
The solution \eqref{p14}-\eqref{r14} is feasible to problem \eqref{frac_convZ_dual} in cases (\hyperlink{C1}{C1}) and (\hyperlink{C4}{C4}).
\end{manualtheorem}
\begin{proof}
We claim that in both (\hyperlink{C1}{C1}) and (\hyperlink{C4}{C4}), $s^i < 0$ for all $i\in\mathcal{S}^*$ and $s^j \geq 0$ for all $j\in\mathcal{V}\backslash\mathcal{S}^*$. It would follow from this claim that $\overline{\mathbf{p}}, \overline{\mathbf{q}} \leq \mathbf{0}$. We justify the claim using the following observations. By construction of $\mathcal{S}^2$, $s^i < 0$ for all $i\in\mathcal{S}^2$ and $s^j \geq 0$ for all $j\notin\mathcal{S}^2\backslash\{\psi,\rho \}$. Now we focus on $\psi$ and $\rho $. In case (\hyperlink{C1}{C1}), $\psi,\rho \notin\mathcal{S}^*$. We observe that $s^{\rho } \geq -\min(0, a_{\psi}) \geq 0$, and $s^\psi = a_\psi + s^{\rho } \geq \min(0, a_{\psi}) + s^{\rho } \geq 0$. In case (\hyperlink{C4}{C4}), $s^{\rho } < 0$ and $\rho  \in\mathcal{S}^*$. By \eqref{si} and \eqref{Sd}, $\mathcal{S}^1 = \{\rho \}\cup\mathcal{S}^2$ and $s^\psi = a_\psi$. When $a_\psi<0$, $\{\psi\} = \mathcal{R} \subseteq \mathcal{S}^*$. Else, $\mathcal{R} = \emptyset$ and $\psi \notin \mathcal{S}^*$. \\

Next, we verify feasibility with respect to constraints \eqref{constr:frac_chain_dual_1}-\eqref{constr:frac_chain_dual_3}. Given $\overline{r} = 0$, \eqref{constr:frac_chain_dual_1}-\eqref{constr:frac_chain_dual_3} become 
\begin{equation}
\label{constr:tree_dual}
p_i + \sum_{j\in\ch{i}}q_{i,j} - q_{\pa{i},i} \leq a_i, \:\forall\: i\in\mathcal{V},
\end{equation} 
where $q_{\pa{\psi},\psi} = 0$. We observe that $s^i = \sum_{k\in\sigma_i(\mathcal{S}^*)} a_k = a_i + \sum_{j\in\ch{i}\backslash \mathcal{S}^*} \sum_{k\in\sigma_j(\mathcal{S}^*)} a_k$. Therefore, in the case of $i\in\mathcal{S}^*$, the left-hand side of constraint \eqref{constr:tree_dual} is 
\[
\overline{p}_i + \sum_{j\in\ch{i}} \overline{q}_{i,j} - \overline{q}_{\pa{i},i} = \sum_{k\in\sigma_i(\mathcal{S}^*)} a_k - \sum_{j\in\ch{i}\backslash \mathcal{S}^*} \sum_{k\in\sigma_j(\mathcal{S}^*)} a_k  = s^i + (a_i - s^i)  = a_i.
\]
On the other hand, suppose $i\notin\mathcal{S}^*$. If $\pa{i}$ exists, then 
\[
\overline{p}_i + \sum_{j\in\ch{i}} \overline{q}_{i,j} - \overline{q}_{\pa{i},i}  = - \sum_{j\in\ch{i}\backslash \mathcal{S}^*} \sum_{k\in\sigma_j(\mathcal{S}^*)} a_k + \sum_{k\in\sigma_i(\mathcal{S}^*)} a_k  = (a_i - s^i) + s^i  = a_i.
\]
Else, $i = \psi$, and 
\[
\overline{p}_i + \sum_{j\in\ch{i}} \overline{q}_{i,j} - \overline{q}_{\pa{i},i}  = - \sum_{j\in\ch{i}\backslash \mathcal{S}^*} \sum_{k\in\sigma_j(\mathcal{S}^*)} a_k  = a_i - s^i \leq a_i
\]
because $s^i \geq 0$ given $i\notin\mathcal{S}^*$. Hence, $[\overline{\mathbf{p}}, \overline{\mathbf{q}},\overline{r}]$ is a feasible dual solution.
\end{proof}

\begin{manualtheorem}{4.12}
In both cases (\hyperlink{C1}{C1}) and (\hyperlink{C4}{C4}), the primal objective values at $\overline{\mathbf{z}} = P(\mathcal{S}^*)$ match the dual objective values at $[\overline{\mathbf{p}}, \overline{\mathbf{q}},\overline{r}]$ given in \eqref{p14}-\eqref{r14}.
\end{manualtheorem}
\begin{proof}
In these two cases, it is impossible to have $\psi\in\mathcal{S}^*$ and $\rho \notin\mathcal{S}^*$, so $P(\mathcal{S}^*)_j = u_{j^\vartriangle(\mathcal{S}^*)}$ for all $j\in\mathcal{V}$. Hence, 
\[
\mathbf{a}^\top \overline{\mathbf{z}}  = \sum_{j\in\mathcal{V}} a_j u_{j^\vartriangle(\mathcal{S}^*)}  =\sum_{i\in\mathcal{S}^*} u_i \sum_{j\in\sigma_i(\mathcal{S}^*)} a_j = \mathbf{u}^\top \overline{\mathbf{p}} + 0\cdot \frac{\floor{u_\psi}(\ceil{u_\psi} - u_\psi)}{u_\psi - \floor{u_\psi}}. 
\]
\end{proof}

\begin{manualtheorem}{4.13}
The solution \eqref{p2}-\eqref{r2} is feasible to problem \eqref{frac_convZ_dual} in case (\hyperlink{C2}{C2}).
\end{manualtheorem}
\begin{proof}
First, we verify that the dual solution is non-positive. Given the observation that $a_\psi + s^{\rho } < 0$, we have $\overline{r} < 0$. By construction, $s^i < 0$ for all $i\in \mathcal{S}^2$, and $s^j \geq 0$ for all $j\notin\mathcal{S}^2 \cup \{\psi,\rho \}$. Therefore, $\overline{\mathbf{p}} \leq \mathbf{0}$ and $\overline{q}_{i,j} \leq 0$ for $j \in\mathcal{V}\backslash  \{\psi,\rho \}$. Lastly, 
\begin{align*}
-s^\rho  - \overline{r} & = -s^\rho  -  \frac{u_\psi - \floor{u_\psi}}{\ceil{u_\psi}- u_\psi}\left(a_\psi + s^\rho \right) \\
& = -\left(\frac{\ceil{u_\psi}- u_\psi}{\ceil{u_\psi}- u_\psi} + \frac{u_\psi - \floor{u_\psi}}{\ceil{u_\psi}- u_\psi} \right) s^\rho  -  \frac{u_\psi - \floor{u_\psi}}{\ceil{u_\psi}- u_\psi}a_\psi \\
& = - \frac{1}{\ceil{u_\psi}- u_\psi} [s^\rho  + (u_\psi - \floor{u_\psi}) a_\psi] \\
& \leq 0 \quad \text{(because $0 \leq s^{\rho } +  (u_\psi - \floor{u_\psi}) a_\psi$ in (\hyperlink{C2}{C2})). }
\end{align*}
Hence, $\overline{\mathbf{q}} \leq \mathbf{0}$ as well. We next check feasibility of $[\overline{\mathbf{p}}, \overline{\mathbf{q}},\overline{r}]$ with respect to \eqref{constr:frac_chain_dual_1}-\eqref{constr:frac_chain_dual_3}. For $i\in\mathcal{V}\backslash\{\psi,\rho \}$, $\pa{i}$ exists, and by construction, $\overline{p}_i - \overline{q}_{\pa{i},i} = s^i = a_i + \sum_{j\in \ch{i}\backslash \mathcal{S}^*} s^j$. Thus, $\overline{p}_i + \sum_{j\in\ch{i}}\overline{q}_{i,j} - \overline{q}_{\pa{i},i} = a_i$ and constraints \eqref{constr:frac_chain_dual_1} are satisfied. In constraint \eqref{constr:frac_chain_dual_2}, 
\begin{align*}
\overline{p}_\psi + \overline{q}_{\psi,\rho } + \frac{\overline{r}}{u_\psi - \floor{u_\psi}} & = 0 - s^\rho  -  \frac{u_\psi - \floor{u_\psi}}{\ceil{u_\psi}- u_\psi}\left(a_\psi + s^\rho \right) + \frac{1}{\ceil{u_\psi}- u_\psi}\left(a_\psi + s^\rho \right) \\
 & = - s^\rho  + \frac{1-u_\psi + \floor{u_\psi}}{\ceil{u_\psi}- u_\psi}\left(a_\psi + s^\rho \right) \\
 & = a_\psi. 
\end{align*}
Lastly, in constraint  \eqref{constr:frac_chain_dual_3}, 
\begin{align*}
\overline{p}_\rho  + \sum_{j\in\ch{\rho }} \overline{q}_{\rho ,j} - \overline{q}_{\psi,\rho } - \overline{r} & = 0 - \sum_{j\in \ch{\rho }\backslash \mathcal{S}^*} s^j + s^\rho  + \overline{r} - \overline{r} \\
& = - \sum_{j\in \ch{\rho }\backslash \mathcal{S}^*} s^j + a_\rho  + \sum_{j\in \ch{\rho }\backslash \mathcal{S}^*} s^j + \overline{r} - \overline{r} \\
& = a_\rho .
\end{align*}
\end{proof}

\begin{manualtheorem}{4.14}
In case (\hyperlink{C2}{C2}), the primal objective value at $\overline{\mathbf{z}} = P(\mathcal{S}^*)$ matches the dual objective value at $[\overline{\mathbf{p}}, \overline{\mathbf{q}},\overline{r}]$ given in \eqref{p2}-\eqref{r2}.
\end{manualtheorem}
\begin{proof}
We have observed earlier that $\psi\in\mathcal{S}^*$ in this case. The primal objective value at $\overline{\mathbf{z}} = P(\mathcal{S}^*)$ is 
\begin{align*}
\mathbf{a}^\top \overline{\mathbf{z}} & = \floor{u_\psi} \sum_{i\in \sigma_\psi(\mathcal{S}^*)} a_i + \sum_{j\in \mathcal{S}^2} u_j \sum_{i\in \sigma_j(\mathcal{S}^*)} a_i \\
& = \floor{u_\psi} (a_\psi + s^\rho ) + \sum_{j\in \mathcal{S}^2} u_j s^j \\
& = \sum_{j\in\mathcal{V}} u_j \overline{p}_j + \frac{\floor{u_\psi}(\ceil{u_\psi} - u_\psi)}{u_\psi - \floor{u_\psi}} \frac{u_\psi - \floor{u_\psi}}{\ceil{u_\psi}- u_\psi}\left(a_\psi + s^\rho \right) \\
& = \sum_{j\in\mathcal{V}} u_j \overline{p}_j + \frac{\floor{u_\psi}(\ceil{u_\psi} - u_\psi)}{u_\psi - \floor{u_\psi}} \overline{r}, 
\end{align*}
which corresponds to the dual objective evaluated at $[\overline{\mathbf{p}}, \overline{\mathbf{q}},\overline{r}]$ defined by \eqref{p2}-\eqref{r2}.
\end{proof}

\begin{manualtheorem}{4.15}
The solution \eqref{p3}-\eqref{r3} is feasible to problem \eqref{frac_convZ_dual} in case (\hyperlink{C3}{C3}).
\end{manualtheorem}
\begin{proof}
By construction, $s^i < 0$ for all $i\in\mathcal{S}^2$, and $-s^j \leq 0$ for all $j\notin\mathcal{S}^*$. Based on the observation that $a_\psi < 0$ and $\mathcal{R} = \{\psi\}$, the conditions of case (\hyperlink{C3}{C3}) become $(u_\psi - \floor{u_\psi})a_\psi  + s^{\rho } < 0$ and $(u_\psi - \floor{u_\psi})a_\psi < 0$. Therefore, $[\overline{\mathbf{p}}, \overline{\mathbf{q}},\overline{r}]\leq\mathbf{0}$. For every $ i\in\mathcal{V}\backslash\{\psi,\rho \}$, 
\[\overline{p}_i + \sum_{j\in\ch{i}}\overline{q}_{i,j} - \overline{q}_{\pa{i},i} = s^i - \sum_{j\in \ch{i}\backslash \mathcal{S}^*} s^j = a_i,\] so constraints \eqref{constr:frac_chain_dual_1} are satisfied. Constraint \eqref{constr:frac_chain_dual_2} holds because 
\[\overline{p}_\psi + \overline{q}_{\psi,\rho } + \frac{\overline{r}}{u_\psi - \floor{u_\psi}}  = 0 + 0 + a_\psi.\] Lastly, the left-hand side of constraint \eqref{constr:frac_chain_dual_3} yields
\[
 \overline{p}_\rho  + \sum_{j\in\ch{\rho }} \overline{q}_{\rho ,j} - \overline{q}_{\psi,\rho } - \overline{r}  = (u_\psi - \floor{u_\psi})a_\psi  + s^{\rho } - \sum_{j\in \ch{\rho }\backslash \mathcal{S}^*} s^j - 0 - (u_\psi - \floor{u_\psi})a_\psi = a_\rho.\]
 Hence, \eqref{p3}-\eqref{r3} is dual feasible in case (\hyperlink{C3}{C3}).
\end{proof}

\begin{manualtheorem}{4.16}
In case (\hyperlink{C3}{C3}), the primal objective value at $\overline{\mathbf{z}} = P(\mathcal{S}^*)$ matches the dual objective value at the solution defined by \eqref{p3}-\eqref{r3}.
\end{manualtheorem}
\begin{proof}
Given the observation that $\psi, \rho \in\mathcal{S}^*$,  $P(\mathcal{S}^*)_j = u_{j^\vartriangle(\mathcal{S}^*)}$ for all $j\in\mathcal{V}$. It follows that
\begin{align*}
\mathbf{a}^\top \overline{\mathbf{z}} & = \sum_{j\in\mathcal{V}} a_j u_{j^\vartriangle(\mathcal{S}^*)}  \\
& = \sum_{i\in\mathcal{S}^*} u_i \sum_{j\in\sigma_i(\mathcal{S}^*)} a_j \\
& = \sum_{i\in\mathcal{S}^2} u_i s^i + u_\rho  s^\rho  + u_\psi s^\psi \\
& = \sum_{i\in\mathcal{S}^2} u_i s^i + u_\rho  [s^\rho  + (u_\psi - \floor{u_\psi}) a_\psi] - u_\rho  (u_\psi - \floor{u_\psi}) a_\psi + u_\psi a_\psi \\
& = \sum_{i\in\mathcal{V}} u_i \overline{p}_i  - u_\rho  (u_\psi - \floor{u_\psi}) a_\psi + u_\psi a_\psi \\
& = \sum_{i\in\mathcal{V}} u_i \overline{p}_i  - u_\rho  u_\psi a_\psi  +  u_\rho \floor{u_\psi} a_\psi + u_\rho u_\psi a_\psi - \floor{u_\psi}u_\psi a_\psi  \\
& = \sum_{i\in\mathcal{V}} u_i \overline{p}_i   +  \floor{u_\psi} (\ceil{u_\psi} - u_\psi)a_\psi  \\
& =  \sum_{i\in\mathcal{V}} u_i \overline{p}_i   + \frac{\floor{u_\psi} (\ceil{u_\psi} - u_\psi)}{u_\psi - \floor{u_\psi}}\overline{r}. 
\end{align*}
\end{proof}

\begin{manualtheorem}{6.3}
For any $\overline{\mathbf{z}}\in \conv{\mathcal{Z}(\mathcal{G},\mathbf{u})}$, let $\tau\in\mathfrak{S}(\mathcal{V})$ be the corresponding output from Algorithm \ref{alg:tree_greedy}. Then the following inequality holds: 
\[f(\overline{\mathbf{z}}) \geq \sum_{k=1}^{|\mathcal{V}|} t(\tau,\overline{\mathbf{z}})_k[f(P(\tau,k)) - f(P(\tau,k-1))]. \]
\end{manualtheorem}
\begin{proof}
For every $j\in\{0,1,2,\dots, |\mathcal{V}|\}$, let $\mathbf{z}^j = \sum_{k=1}^{j} t(\tau,\overline{\mathbf{z}})_k[P(\tau,k) - P(\tau,k-1)]$. Due to Lemmas \ref{lem:first_and_one}, \ref{lem:monotone_t}, and \ref{lem:last_nonneg}, $0\leq t(\tau,\overline{\mathbf{z}})_k\leq 1$ for every $k=1,\dots, j$, so 
\begin{equation}
\label{obs:lem_tree_partial_valid}
\mathbf{z}^j \leq \sum_{k=1}^{j} 1\cdot[P(\tau,k) - P(\tau,k-1)] = P(\tau,j)
\end{equation}
for each $j\in\{1,2,\dots, |\mathcal{V}|\}$. Moreover,
\begin{align*}
f\left(\mathbf{z}^j\right) - f\left(\mathbf{z}^{j-1}\right) 
& = f\left(\mathbf{z}^{j-1} + [P(\tau,j)-P(\tau,j-1)] t(\tau, \overline{\mathbf{z}})_j\right) - f\left(\mathbf{z}^{j-1}\right) \\
& \geq t(\tau, \overline{\mathbf{z}})_j  \left[ f\left(\mathbf{z}^{j-1} +  [P(\tau,j)-P(\tau,j-1)] \right) - f\left(\mathbf{z}^{j-1}\right) \right] \\
& \quad \text{(by Lemma \ref{lem:pos_dir_increment})} \\
& \geq t(\tau, \overline{\mathbf{z}})_j  \left[ f\left(P(\tau,j-1) +  [P(\tau,j)-P(\tau,j-1)] \right) - f\left(P(\tau,j-1)\right) \right] \\
&  \quad \text{(by \eqref{obs:lem_tree_partial_valid} and DR-submodularity of $f$)} \\
& = t(\tau, \overline{\mathbf{z}})_j  \left[ f\left(P(\tau,j)\right) - f\left(P(\tau,j-1)\right) \right].
\end{align*}
Due to Proposition \ref{prop:tree_conv_z}, 
\[\overline{\mathbf{z}} = \sum_{k=0}^{|\mathcal{V}|} \left[t(\tau, \overline{\mathbf{z}})_k-t(\tau, \overline{\mathbf{z}})_{k+1}\right] P(\tau, k) = \sum_{k=1}^{|\mathcal{V}|} t(\tau, \overline{\mathbf{z}})_k \left[P(\tau, k) - P(\tau, k-1)\right] = \mathbf{z}^{|\mathcal{V}|}.\]

Therefore, we obtain
\[f(\overline{\mathbf{z}})  = f(\mathbf{z}^{|\mathcal{V}|}) - f(\mathbf{0})  = \sum_{j=1}^{|\mathcal{V}|} [f(\mathbf{z}^j) - f(\mathbf{z}^{j-1})] \geq \sum_{j=1}^{|\mathcal{V}|}  t(\tau, \overline{\mathbf{z}})_j  \left[ f\left(P(\tau,j)\right) - f\left(P(\tau,j-1)\right) \right].\]
\end{proof}

\begin{manualtheorem}{6.6}
For every $j\in\{0,\dots,|\mathcal{V}|\}$, 
\[
\widetilde{\mathbf{z}}^j = \sum_{k=1}^{j} t(\delta, \mathbf{z})_k [P(\delta,k) - P(\delta,k-1)] \leq P(\delta, j). 
\]
\end{manualtheorem}
\begin{proof}
We prove this lemma by showing that $\widetilde{z}^j_i = z_{i^\vartriangle(\mathcal{T}^{\delta, j})}$ for every $i\in \mathcal{V}$ and $j\in\{0,\dots,|\mathcal{V}|\}$. For each $k\in\{1,\dots,j\}$ and $i\in\mathcal{V}$, 
\[P(\delta,k)_i - P(\delta,k-1)_i = 
\begin{cases}
u_{\delta(k)} - u_{\delta(k)^\vartriangle(\mathcal{T}^{\delta,k-1})}, & i\in\sigma_{\delta(k)}(\mathcal{T}^{\delta,k-1}), \\
0, & \text{otherwise.}
\end{cases}\]
It follows that 
\[t(\delta,\mathbf{z})_k[P(\delta,k)_i - P(\delta,k-1)_i]  = 
\begin{cases}
z_{\delta(k)} - z_{\delta(k)^\vartriangle(\mathcal{T}^{\delta,k-1})}, & i\in\sigma_{\delta(k)}(\mathcal{T}^{\delta,k-1}), \\
0, & \text{otherwise. }
\end{cases}\]

Now we fix an arbitrary $j\in\{0,\dots,|\mathcal{V}|\}$. For every $i\in\mathcal{V}$, we define 
\begin{equation}
\label{I}
\mathcal{I} = \{\delta(k): 1\leq k\leq j \text{ such that } i\in\sigma_{\delta(k)}(\mathcal{T}^{\delta,k})\}.
\end{equation}
Every $\delta(k)\in\mathcal{I}$ is the ascendant of $i$ with maximal depth in $\mathcal{T}^{\delta, k}$. We note that for any $\delta(k)\notin\mathcal{I}$, $1\leq k\leq j$, $P(\delta, k)_i = P(\delta, k-1)_i$. When $\mathcal{I}\neq \emptyset$, all the elements of $\mathcal{I}$ fall in the same directed path and have distinct depths. We note that for any $\delta(k),\delta(k')\in \mathcal{I}$, $\dep{\delta(k)} < \dep{\delta(k')}$ when $k < k'$ and vice versa. For a contradiction, if $\dep{\delta(k)} > \dep{\delta(k')}$ and $k < k'$, then $\delta(k)\in \mathcal{T}^{\delta, k'}$ and $i^\vartriangle(\delta, k')$ must be a descendant of $\delta(k)$; this contradicts the fact that $i^\vartriangle(\mathcal{T}^{\delta,k'}) = \delta(k')$. Now we sort the elements in $\mathcal{I}$ by ascending depths; that is, $\mathcal{I}_1$ and $\mathcal{I}_{|\mathcal{I}|}$ have the minimal and maximal depths, respectively. This is equivalent to sorting $q\in \mathcal{I}$ by ascending $\delta^{-1}(q)$. By definition, $\mathcal{I}_{|\mathcal{I}|}$ is the maximal depth ascendant of $i$ in $\mathcal{T}^{\delta, j}$, so
$i^\vartriangle(\mathcal{T}^{\delta, j}) = \mathcal{I}_{|\mathcal{I}|}$.
Due to the fact that $\mathcal{I}_1$ is the minimal depth ascendant of $i$ in $\mathcal{T}^{\delta, j}$, we have
$\mathcal{I}_1^\vartriangle(\mathcal{T}^{\delta, \delta^{-1}(\mathcal{I}_1)-1}) = 0$. For ease of notation, we let $\mathcal{I}_0 = 0$. For any $\ell\in\{2,\dots, |\mathcal{I}|\}$, 
$\mathcal{I}_{\ell-1} = \mathcal{I}_\ell^\vartriangle(\mathcal{T}^{\delta, \delta^{-1}(\mathcal{I}_{\ell})-1})$ 
because $\mathcal{I}_\ell^\vartriangle(\mathcal{T}^{\delta, \delta^{-1}(\mathcal{I}_{\ell})-1})\in R^+(\mathcal{I}_{\ell-1})$ and it must be an element of $\mathcal{I}$. We then obtain 
\begingroup
\allowdisplaybreaks
\begin{align*}
\widetilde{{z}}^j_i & = \sum_{\ell=1}^{|\mathcal{I}|} t(\delta,{\mathbf{z}})_{\delta^{-1}(\mathcal{I}_\ell)} [P(\delta,\delta^{-1}(\mathcal{I}_\ell)) - P(\delta,\delta^{-1}(\mathcal{I}_\ell)-1)]_{\mathcal{I}_\ell} \\
& =  \sum_{\ell=1}^{|\mathcal{I}|} (z_{\mathcal{I}_\ell} - z_{\mathcal{I}_{\ell-1}}) \\
& = z_{\mathcal{I}_{|\mathcal{I}|}}  - z_{0} \\
& = z_{i^\vartriangle(\mathcal{T}^{\delta, j})} - 0\\
& = z_{i^\vartriangle(\mathcal{T}^{\delta, j})} \\
& \leq P(\delta, j)_i
\end{align*}
\endgroup
for every $j\in\{0,\dots,|\mathcal{V}|\}$ and $i\in\mathcal{V}$. Hence, $\widetilde{\mathbf{z}}^j \leq P(\delta, j)$ component-wise. 
\end{proof}

\begin{manualtheorem}{6.8}
For any $j\in\{1,\dots,|\mathcal{V}|\}$, when $0\leq t(\delta, \mathbf{z})_j \leq 1$, 
\[ f\left(\widetilde{\mathbf{z}}^j\right) - f\left(\widetilde{\mathbf{z}}^{j-1}\right) \geq  t(\delta, \mathbf{z})_j  \left[ f\left(P(\delta,j)\right) - f\left(P(\delta,j-1)\right) \right].\]
\end{manualtheorem}
\begin{proof}
Recall that $t(\delta, \mathbf{z})_j \geq 0$ and $P(\delta,j)-P(\delta,j-1) \geq \mathbf{0}$. We obtain the sequence of inequalities:
\begin{align*}
f\left(\widetilde{\mathbf{z}}^j\right) - f\left(\widetilde{\mathbf{z}}^{j-1}\right) 
& = f\left(\widetilde{\mathbf{z}}^{j-1} + [P(\delta,j)-P(\delta,j-1)] t(\delta, \mathbf{z})_j\right) - f\left(\widetilde{\mathbf{z}}^{j-1}\right) \\
& \geq t(\delta, \mathbf{z})_j  \left[ f\left(\widetilde{\mathbf{z}}^{j-1} +  [P(\delta,j)-P(\delta,j-1)] \right) - f\left(\widetilde{\mathbf{z}}^{j-1}\right) \right] \\
&  \quad \text{(by Lemma \ref{lem:pos_dir_increment} when $t(\delta, \mathbf{z})_j < 1$ and is immediate if $t(\delta, \mathbf{z})_j = 1$)} \\
& \geq t(\delta, \mathbf{z})_j  \left[ f\left(P(\delta,j-1) +  [P(\delta,j)-P(\delta,j-1)] \right) - f\left(P(\delta,j-1)\right) \right] \\
&  \quad \text{(by DR-submodularity of $f$ and Lemma \ref{lem:tildez_and_P})} \\
& \geq t(\delta, \mathbf{z})_j  \left[ f\left(P(\delta,j)\right) - f\left(P(\delta,j-1)\right) \right].
\end{align*}
\end{proof}

\begin{manualtheorem}{6.9}
For $j\in\{1,\dots,|\mathcal{V}|\}$, when $t(\delta, \mathbf{z})_j > 1$, 
\[ f\left(\widetilde{\mathbf{z}}^j\right) - f\left(\widetilde{\mathbf{z}}^{j-1}\right) \geq  t(\delta, \mathbf{z})_j  \left[ f\left(P(\delta,j)\right) - f\left(P(\delta,j-1)\right) \right].\]
\end{manualtheorem}
\begin{proof}
Given $t(\delta, \mathbf{z})_j > 1$, we know that $\widetilde{\mathbf{z}}^j - [P(\delta,j)-P(\delta,j-1)] > \widetilde{\mathbf{z}}^j - [P(\delta,j)-P(\delta,j-1)]t(\delta, \mathbf{z})_j = \widetilde{\mathbf{z}}^{j-1}$. Let \[\gamma = \frac{\widetilde{\mathbf{z}}^j- \widetilde{\mathbf{z}}^{j-1}-(P(\delta,j)-P(\delta,j-1))}{\widetilde{\mathbf{z}}^j- \widetilde{\mathbf{z}}^{j-1}}.\] We notice that $\widetilde{\mathbf{z}}^{j-1} + \gamma(\widetilde{\mathbf{z}}^j- \widetilde{\mathbf{z}}^{j-1}) = \widetilde{\mathbf{z}}^j - (P(\delta,j)-P(\delta,j-1))$. In addition, $0 < \gamma < 1$ and 
\[  \frac{1}{1-\gamma} = \frac{\widetilde{\mathbf{z}}^j- \widetilde{\mathbf{z}}^{j-1}}{P(\delta,j)-P(\delta,j-1)} = t(\delta, \mathbf{z})_j. \]
 Then
\begin{align*}
f\left(\widetilde{\mathbf{z}}^j\right) - f\left(\widetilde{\mathbf{z}}^{j-1}\right) &  \geq \frac{1}{1-\gamma}\left[f\left(\widetilde{\mathbf{z}}^{j-1}+ \widetilde{\mathbf{z}}^j- \widetilde{\mathbf{z}}^{j-1}\right) - f\left(\widetilde{\mathbf{z}}^{j-1} + \gamma (\widetilde{\mathbf{z}}^j- \widetilde{\mathbf{z}}^{j-1})\right) \right] \quad \text{(by Lemma \ref{lem:pos_dir_increment})} \\
& = t(\delta, \mathbf{z})_j \left[f\left(\widetilde{\mathbf{z}}^j\right) - f\left( \widetilde{\mathbf{z}}^j - (P(\delta,j)-P(\delta,j-1))\right) \right] \\
& \geq t(\delta, \mathbf{z})_j  \left[ f\left(P(\delta,j)  \right) - f\left(P(\delta,j) - [P(\delta,j)-P(\delta,j-1)]\right) \right] \\
&  \quad \text{(by DR-submodularity of $f$ and Lemma \ref{lem:tildez_and_P})} \\
& \geq t(\delta, \mathbf{z})_j  \left[ f\left(P(\delta,j)\right) - f\left(P(\delta,j-1)\right) \right].
\end{align*}
\end{proof}

\begin{manualtheorem}{6.15}
For any $\psi\in\widecheck{\Psi}$ and any $j\in J(\psi)$, the following inequality holds:  
\[f(\widecheck{\mathbf{z}}^{j}) - f(\widecheck{\mathbf{z}}^{j-1}) \geq \frac{u_{\ch{\psi}}-z_\psi}{u_{\ch{\psi}}-u_\psi} [f(P(\delta, j)) - f(P(\delta, j-1)) + f(P(\delta, j-1)) - f(Q(\delta, j))].\]
\end{manualtheorem}
\begin{proof}
Let $i = \delta(j)$. Note that $Q(\delta, j) \geq \widecheck{\mathbf{z}}^{j-1} + u_\psi\mathbf{e}^i$, and $P(\delta, j) - Q(\delta,j) = (u_{\ch{\psi}} - u_\psi)\mathbf{e}^i$. We obtain
\begin{align*}
f(\widecheck{\mathbf{z}}^{j}) - f(\widecheck{\mathbf{z}}^{j-1}) & = f(\widecheck{\mathbf{z}}^{j-1} + (u_{\ch{\psi}}-z_\psi)\mathbf{e}^i) - f(\widecheck{\mathbf{z}}^{j-1}) \quad \text{(from \eqref{z_ell})} \\
& \geq \frac{1}{1-(u_\psi-z_\psi)/(u_{\ch{\psi}}-z_\psi)} [f(\widecheck{\mathbf{z}}^{j-1} + (u_{\ch{\psi}}-z_\psi)\mathbf{e}^i) - f(\widecheck{\mathbf{z}}^{j-1} + (u_\psi-z_\psi)\mathbf{e}^i)] \\
& \quad \text{(by Lemma \ref{lem:pos_dir_increment})} \\
& = \frac{u_{\ch{\psi}}-z_\psi}{u_{\ch{\psi}}-u_\psi} [f(\widecheck{\mathbf{z}}^{j-1} + (u_{\ch{\psi}}-z_\psi)\mathbf{e}^i) - f(\widecheck{\mathbf{z}}^{j-1} + (u_\psi-z_\psi)\mathbf{e}^i)]  \\
& \geq \frac{u_{\ch{\psi}}-z_\psi}{u_{\ch{\psi}}-u_\psi} [f(P(\delta, j)) - f(Q(\delta,j))] \quad \text{(because $f$ is DR-submodular)} \\
& = \frac{u_{\ch{\psi}}-z_\psi}{u_{\ch{\psi}}-u_\psi} [f(P(\delta, j)) - f(P(\delta, j-1)) + f(P(\delta, j-1)) - f(Q(\delta,j))]. 
\end{align*}
\end{proof}

\begin{manualtheorem}{6.17}
For any $\psi\in\widecheck{\Psi}$, 
\[\sum_{j\in J(\psi)} [f(P(\delta, j-1)) - f(Q(\delta, j))] \geq - \frac{u_\psi - \floor{u_\psi}}{\floor{u_\psi}- u_{\psi^\vartriangle(\mathcal{T}^{\delta, j-1})}} [f(P(\delta, \delta^{-1}(\psi))) - f(P(\delta, \delta^{-1}(\psi)-1))].\]
\end{manualtheorem}
\begin{proof}
For any $j = J(\psi)_{\ell}$, $\ell\in\{1,\dots, |J(\psi)|\}$, $P(\delta, j-1) \leq Q(\delta, j)$ component-wise, $Q(\delta, j) - P(\delta, j-1) = T(\delta, J(\psi)_{\ell}) - T(\delta, J(\psi)_{\ell-1})$, and $Q(\delta,j) \geq T(\delta, j)$ according to Observation \ref{obs:PQT}. By DR-submodularity of $f$, 
\[f(T(\delta, J(\psi)_{\ell})) - f(T(\delta, J(\psi)_{\ell-1})) \geq f(Q(\delta, j)) - f(P(\delta, j-1)),  \]
which is equivalent to 
\[f(P(\delta, j-1)) - f(Q(\delta, j)) \geq f(T(\delta,J(\psi)_{\ell-1})) - f(T(\delta, J(\psi)_{\ell})).  \]
Following from this inequality, 
\begin{align*}
\sum_{j\in J(\psi)} [f(P(\delta, j-1)) - f(Q(\delta, j))] & \geq \sum_{\ell=1}^{|J(\psi)|} [f(T(\delta, J(\psi)_{\ell-1})) - f(T(\delta,J(\psi)_\ell))] \\
& = f(T(\delta, J(\psi)_0)) - f(T(\delta,J(\psi)_{|J(\psi)|})) \\
& = f(P(\delta, \delta^{-1}(\psi))) -  f(T(\delta, \delta^{-1}(\ch{\psi}))).
\end{align*}
We note that $T(\delta, \delta^{-1}(\ch{\psi})) = P(\delta, \delta^{-1}(\psi)-1) + (u_\psi - u_{\psi^\vartriangle(\mathcal{T}^{\delta, j-1})})\sum_{i'\in\sigma_{\psi}(\mathcal{T}^{\delta, \delta^{-1}(\psi)})}\mathbf{e}^{i'}$, and 
\begin{align*}
& \quad  f(P(\delta, \delta^{-1}(\psi))) - f(P(\delta, \delta^{-1}(\psi)-1)) \\
 & = f\left(P(\delta, \delta^{-1}(\psi)-1) + (\floor{u_\psi}  - u_{\psi^\vartriangle(\mathcal{T}^{\delta, j-1})})\sum_{i'\in\sigma_{\psi}(\mathcal{T}^{\delta, \delta^{-1}(\psi)})}\mathbf{e}^{i'}\right) - f(P(\delta, \delta^{-1}(\psi)-1)) \\
 & \geq  \frac{\floor{u_\psi}  - u_{\psi^\vartriangle(\mathcal{T}^{\delta, j-1})}}{u_\psi  - u_{\psi^\vartriangle(\mathcal{T}^{\delta, j-1})}} \left[f\left(P(\delta, \delta^{-1}(\psi)-1) + (u_\psi  - u_{\psi^\vartriangle(\mathcal{T}^{\delta, j-1})})\sum_{i'\in\sigma_{\psi}(\mathcal{T}^{\delta, \delta^{-1}(\psi)})}\mathbf{e}^{i'}\right) - f(P(\delta, \delta^{-1}(\psi)-1))\right] \\
 & \quad \text{(by Lemma \ref{lem:pos_dir_increment})} \\
 & =  \frac{\floor{u_\psi}  - u_{\psi^\vartriangle(\mathcal{T}^{\delta, j-1})}}{u_\psi  - u_{\psi^\vartriangle(\mathcal{T}^{\delta, j-1})}} \left[ f(T(\delta, \delta^{-1}(\ch{\psi}))) - f(P(\delta, \delta^{-1}(\psi)-1))\right]. 
\end{align*}
In other words, 
\[- \left[ f(T(\delta, \delta^{-1}(\ch{\psi}))) - f(P(\delta, \delta^{-1}(\psi)-1))\right] \geq -  \frac{u_\psi  - u_{\psi^\vartriangle(\mathcal{T}^{\delta, j-1})}}{\floor{u_\psi}  - u_{\psi^\vartriangle(\mathcal{T}^{\delta, j-1})}} \left[f(P(\delta, \delta^{-1}(\psi))) - f(P(\delta, \delta^{-1}(\psi)-1))\right]. \] 
Therefore, 
\begin{align*}
& \quad \sum_{j\in J(\psi)} [f(P(\delta, j-1)) - f(Q(\delta, j))] \\
& \geq  f(P(\delta, \delta^{-1}(\psi))) - \frac{u_\psi}{\floor{u_\psi}} f(P(\delta, \delta^{-1}(\psi))) \\
& = \left[f(P(\delta, \delta^{-1}(\psi))) - f(P(\delta, \delta^{-1}(\psi)-1))\right] - \left[ f(T(\delta, \delta^{-1}(\ch{\psi}))) - f(P(\delta, \delta^{-1}(\psi)-1))\right] \\
& \geq  \left[f(P(\delta, \delta^{-1}(\psi))) - f(P(\delta, \delta^{-1}(\psi)-1))\right] -  \frac{u_\psi  - u_{\psi^\vartriangle(\mathcal{T}^{\delta, j-1})}}{\floor{u_\psi}  - u_{\psi^\vartriangle(\mathcal{T}^{\delta, j-1})}} \left[f(P(\delta, \delta^{-1}(\psi))) - f(P(\delta, \delta^{-1}(\psi)-1))\right] \\
& = - \frac{u_\psi  - \floor{u_\psi}}{\floor{u_\psi}  - u_{\psi^\vartriangle(\mathcal{T}^{\delta, j-1})}}\left[f(P(\delta, \delta^{-1}(\psi))) - f(P(\delta, \delta^{-1}(\psi)-1))\right] .
\end{align*}
\end{proof}

\begin{manualtheorem}{6.18}
For all $j\notin \bigcup_{\psi\in\widecheck{\Psi}} \left(J(\psi)\cup \{\psi\}\right)$, 
\[f(\widecheck{\mathbf{z}}^{j}) - f(\widecheck{\mathbf{z}}^{j-1}) \geq t(\delta,\mathbf{z})_j [f(P(\delta, j)) - f(P(\delta, j-1))]. \]
\end{manualtheorem}
\begin{proof}
Based on Observation \ref{obs:P_tpp}, $\widecheck{\mathbf{z}}^j = \widecheck{\mathbf{z}}^{j-1} + t(\delta,\mathbf{z})_j [P(\delta, j) - P(\delta, j-1)]$ for $j\notin \bigcup_{\psi\in\widecheck{\Psi}} \left(J(\psi)\cup \{\psi\}\right)$. We know that $t(\delta,\mathbf{z})_j \geq 0$ and $\widecheck{\mathbf{z}}^{j-1} \leq P(\delta, j-1)$. Therefore, by the same analysis as in Lemmas \ref{lem:tree_t_small} and \ref{lem:tree_t_large}, the stated inequality holds. 
\end{proof}

\end{document}